\documentclass[11pt,letterpaper]{article}
\usepackage{float}

\usepackage[T1]{fontenc}
\usepackage{amsmath, amsthm, amsfonts,amssymb}
\usepackage{derivative,mathtools,mathrsfs,yhmath}
\usepackage[applemac]{inputenc}		
\usepackage{hyperref}
\usepackage{comment,color,graphicx, subcaption}
\mathtoolsset{showonlyrefs}
\usepackage{indentfirst}
\usepackage{tikz}
\usetikzlibrary{calc,arrows.meta}

\hypersetup{colorlinks=true,urlcolor=black, pdftitle="Fourier Mean Bodies"}

\newtheorem{conj}{Conjecture}
\newtheorem{thm}[conj]{Theorem}
\newtheorem{theorem}{Theorem}

\newtheorem{rem}[conj]{Remark}
\newtheorem{lem}[conj]{Lemma}
\newtheorem{prop}[conj]{Proposition}

\newtheorem{defn}[conj]{Definition}
\newtheorem{cor}[conj]{Corollary}

\newcommand{\vol}{\mathrm{Vol}}
\newcommand{\supp}{\mathrm{supp}}

\newcommand{\R}{\mathbb{R}}
\newcommand{\N}{\mathbb{N}}
\newcommand{\Z}{\mathbb{Z}}

\def\s{\mathbb{S}}

\def\R{{\mathbb R}}
\def\B{B_2^n}

\def\phi{\varphi}

\newcommand\nnfootnote[1]{%
  \begin{NoHyper}
  \renewcommand\thefootnote{}\footnote{#1}%
  \addtocounter{footnote}{-1}%
  \end{NoHyper}
}

\begin{document}

\title{On the Fourier Mean Bodies of a Convex Body \nnfootnote{MSC Classification: 52A40, 52A30;  Secondary: 28A75, 42B10 \\ Keywords: Intersection bodies of star bodies, isotropic position of convex bodies, geometric tomography, Radial Mean Bodies, Fourier Mean Bodies} }
\author{Dylan Langharst, Auttawich Manui and Artem Zvavitch}
\date{}
\maketitle
\begin{abstract}
In 1998, R. Gardner and G. Zhang introduced the radial $p$th mean bodies $R_pK$ of a convex body $K\subset\mathbb R^n$, $p>-1$, which have since become important objects in geometric tomography. In this paper, we study the Fourier transforms of the radial functions of $R_pK$. This leads to a new family of star-shaped sets $F_pK$, which we call the Fourier $p$th mean bodies of $K$. We prove Fourier inversion formulas connecting $R_pK$ and $F_pK$, realizing them as $p$-intersection bodies in the sense of A. Koldobsky. We develop the basic affine geometry of $F_pK$; this includes affine covariance and monotonicity properties. 

We give sufficient conditions for the compactness of $F_pK$ in terms of the decay of $|\widehat{\chi_K}|^2$. We show that $F_pK$ is an origin-symmetric convex body for every $0<p\le1$. This range is sharp in general: already for the cube,
$F_p[-1,1]^n$ is not convex for $1<p<2$ and $n\geq 2,$ while $F_p[-1,1]^n$ is not compact for $p\geq 2$. We further investigate the features Fourier mean bodies share with intersection bodies: we prove Hensley-type estimates for $F_pK$ when $K$ is isotropic and investigate an affine isoperimetric inequality.
\end{abstract}

\newpage

\tableofcontents

\section{Introduction}
Throughout, $\N=\{0,1,2,\dots\}$. Letting $\R^n$ denote the usual $n$-dimensional Euclidean space with inner product $\langle \,,\rangle$, we write $\vol_n$ for the full-dimensional Lebesgue measure on $\R^n$ and $\s^{n-1}$ for the unit sphere. We use the counting measure on $\s^0=\{-1,1\}$ and the convention $\vol_0(\{0\})=1$. For $\theta \in \s^{n-1}$, let $\theta^\perp:=\{x\in\R^n: \langle x,\theta \rangle=0\}$ denote the hyperplane through the origin orthogonal to $\theta$. We say a set $M\subset\R^n$ is origin-symmetric if $M=-M$.

A set $M\subset\R^n$ is \textit{star-shaped} if it contains the origin $o$ and if $x\in M$ implies $[o,x]\subset M$. For a star-shaped set $M$, define its radial function by $\rho_M(\theta)=\sup\{t>0:t\theta\in M\}, \theta\in \s^{n-1}$. We say $M$ is \textit{an $L^p$-star} if $\rho_M$ is not identically zero a.e. and $\rho_M\in L^{p}(\s^{n-1})$. The \textit{polar coordinate formula} for the volume of star-shaped sets, $\vol_n(M)=\frac{1}{n}\int_{\s^{n-1}}\rho_M^n(\theta)d\theta,$ shows that $M$ is \textit{an $L^n$-star} if and only if $0<\vol_n(M)<+\infty$. We say $M$ is a \textit{star body} if it is a compact, star-shaped set with non-empty interior whose radial function is continuous and positive. Finally, we say $K\subset \R^n$ is a \textit{convex body} if it is a compact, convex set with non-empty interior.

\subsection{Intersection Bodies} In 1988, E. Lutwak \cite{LE98} introduced the \textit{intersection  body} of a star body $M$ as the star body $IM$ given by the radial function $$\rho_{IM}(\theta) := \vol_{n-1}(M\cap \theta^\perp), \qquad \theta\in \s^{n-1}.$$
A celebrated theorem of H. Busemann \cite{HB49} (see also \cite{MP89}) asserts that, if $K \subset \R^n$ is an origin-symmetric convex body, then so too is $IK$. The intersection body played a fundamental role in resolving the Busemann-Petty problem, originally posed in \cite{BP56} and ultimately solved through a series of works, e.g. \cite{MP92,GZ94, GR94_2, AK98_3, GZ99_2, GKS99, GKS99_2}. We refer the reader to the monograph by A. Koldobsky \cite{AK05} and the book by R. Gardner \cite{gardner_book} for a thorough background on the history of the Busemann-Petty problem and the development of intersection bodies. 

We say that a convex body $K$ is \textit{isotropic} if it has center of mass at the origin, if $\vol_n(K)=1$, and if there exists a constant $L_{K}>0$ such that
\begin{equation}
\label{eq:isotropic_constant}
\int_{K}\langle x,\theta \rangle^2 dx = L_{K}^2, \quad \theta\in\s^{n-1}.
\end{equation}
It is easy to see that one can always apply an affine transformation to a convex body so that its image is isotropic. A theorem of D. Hensley \cite{HD80} (see also \cite{MP89,BGVV14}), originally stated in the origin-symmetric case, asserts that if $K$ is any isotropic convex body, then $IK$ is isomorphic to the Euclidean ball. Specifically, there exist absolute constants $c,b>0$ such that
\begin{equation}
\label{eq:hen}
cL_K^{-1}\leq\vol_{n-1}(K\cap \theta^\perp)<bL_K^{-1}, \qquad \forall \; \theta\in\s^{n-1},
\end{equation}
or  $c L^{-1}_K \le \rho_{IK} \le b L^{-1}_K$ pointwise. The verification of \eqref{eq:hen} in the non-symmetric case is due to M. Fradelizi \cite{FM99}, who also obtained the sharp constants and equality characterization; see Section~\ref{sec:hen} for details.

Intersection bodies serve as a link between geometric tomography and Fourier analysis. For $n\geq 2$, \cite[Lemma 3.7]{AK05} asserts for an origin-symmetric star body $M\subset \R^n$,
\begin{equation}
\label{eq:intersection_fourier}
    \rho_{IM} = \frac{1}{\pi(n-1)}\widehat{\rho_M^{n-1}}.
\end{equation}
Here and throughout, $\widehat{\cdot}$ denotes the Fourier transform in the distributional sense; see Section~\ref{sec:fourier_facts} for details. In our work, we consider a Fourier analytic classification of a series of convex bodies introduced by R. Gardner and G. Zhang. 

\subsection{Radial Mean Bodies and Polar Mean Zonoids}
For a convex body $K\subset \R^n$, R. Gardner and G. Zhang introduced the so-called radial $p$th mean bodies $R_p K$ of $K$, where $p>-1$, in their remarkable work \cite{GZ98}. We present here the following definition, which is entirely equivalent to that in \cite{GZ98}; see, e.g., \cite{GZ98,LP25}.  We recall the \textit{covariogram} function of a convex body $K\subset \R^n$:
\begin{equation}
g_K(x)=\vol_n(K\cap(K+x)), \qquad x\in\R^n.
\label{eq:covario}
\end{equation}
Let $K\subset \R^n$ be a convex body. Then, for $p>-1$, its {\it radial $p$th mean body $R_p K\subset \R^n$} is the star body given by the radial function, for $\theta\in\s^{n-1}$,
    \begin{equation}
\rho_{R_p K}(\theta)=\begin{cases}
    \left(p\int_0^{+\infty}\left(\frac{g_K(r\theta)}{\vol_n(K)}\right)r^{p-1}dr\right)^\frac{1}{p}, & p>0,
    \\
    \exp\left(\int_0^{+\infty}\frac{\partial}{\partial r}\left(\frac{-g_K(r\theta)}{\vol_n(K)}\right)\log(r)dr\right), & p=0,
    \\
    \left(p\int_{0}^{+\infty} \left(\frac{g_K(r\theta)}{\vol_n(K)}-1\right) r^{p-1} d r\right)^\frac{1}{p}, & p\in (-1,0).
\end{cases}
\label{eq:radial_ell}
\end{equation}
In particular, $R_\infty K=DK:=\{x\in \R^n: K\cap (K+x)\neq \emptyset\}$ is the difference body of $K$. The evenness of $g_K$ implies that the sets $R_p K$ are origin-symmetric for all $p>-1$. R. Gardner and G. Zhang \cite[Corollary 4.2]{GZ98} used a result by K. Ball \cite{Ball88} to establish the convexity of $R_p K$ for $p\geq 0$. Recently, the first-named author extended Ball's result to $p>-1$ \cite{DL26}, thereby yielding convexity of $R_p K$ for all $n \geq 1$; the convexity when $n=2$ was previously established directly by J. Haddad \cite{JH26} using a different method.
\begin{theorem}[Convexity of Radial Mean Bodies]
\label{t:gardner_zhang}
    Let $n\geq 1$ and let $K\subset \R^n$ be a convex body. Then, for all $p>-1$, $R_p K$ is an origin-symmetric convex body.
\end{theorem}
For $q>-1$, $q\neq 0$, the polar $L^q$ centroid body (see, e.g., \cite{LZ97} and \cite{LYZ00}) of a star body $M$ is the origin-symmetric star body $\Gamma^\circ_q M$ given by the radial function
\begin{equation}
\label{eq:L_p}
    \rho_{\Gamma^\circ_q M}(\theta)=\left(\frac{1}{\vol_n(M)}\int_{M}|\langle x,\theta \rangle|^q dx\right)^{-\frac{1}{q}}, \qquad \theta\in\s^{n-1}.
\end{equation}
For $q\neq 0$, the \textit{polar $q$th mean zonoid} $Z^\circ_q K$ of a convex body $K$ is given by
    \begin{equation}Z^\circ_{q}K = \left(\frac{\vol_n(K)}{\vol_n(R_{n+q}K)}\right)^\frac{1}{q}\Gamma_q^\circ (R_{n+q} K).
 \label{eq:mean_zonoid_relate}
 \end{equation}
 It follows from  \eqref{eq:mean_zonoid_relate} that $Z^\circ_{q}K$ are origin-symmetric star bodies. These were first introduced by R. Gardner and A. Giannopoulos \cite{GG99}, under a slightly different normalization. The following theorem of G. Berck \cite{BG09} shows that $Z_q^\circ K$ is a convex body for all $q> -1$.
 
 \begin{theorem}[Convexity of Polar Centroid Bodies]
\label{t:berck}
    Let $n\geq 1$, $q > -1$, $q\neq 0$, and let $K\subset \R^n$ be an origin-symmetric convex body. Then, $\Gamma_q^\circ K$ is an origin-symmetric convex body.
\end{theorem} 

In this work, we are interested in studying $\widehat{\rho_{R_p K}^p}$ and uncovering the radial functions of other sets that depend on $K$. We first present the Fourier transform formula of $R_pK$ for ``large'' $p$, i.e., when $p>n-1$.
\begin{thm}
\label{t:p_n_+_1}
    Let $n\geq 1$ and let $K\subset\R^n$ be a convex body. Then, there exists an explicit constant $m(n,p)$ (see \eqref{eq:m} below) depending only on $n$ and $p$ such that, for $p >n-1$, $p\neq n+2k$ with $k\in \N$,
    \begin{equation}
        m(n,p)\widehat{\rho_{R_{p} K}^{p}}=\vol_n(K)\rho_{Z_{p-n}^\circ K}^{n-p}.
        \label{eq:p=n+1}
    \end{equation}
    In particular, $
m(n,p)\rho_{R_p K}^p
$
is a tempered, positive-definite distribution.
\end{thm}
We elaborate on the extension of $\rho_{R_p K}^p$ as a tempered distribution with an example: for $n<p<n+1$ its action on the Schwartz class $\mathcal{S}(\R^n)$, defined in Section~\ref{sec:fourier_facts}, of test functions $\varphi:\R^n\to\R$ is given by
\[
\left\langle \rho_{R_pK}^{p},\varphi\right\rangle
=
\int_{\mathbb R^n}
\rho_{R_pK}(\xi)^p
\bigl(\varphi(\xi)-\varphi(0)\bigr)\,d\xi.
\]

\subsection{Fourier Mean Bodies}
Our next goal is to obtain the analogue of Theorem~\ref{t:p_n_+_1} for $0<p<n$, which we achieve in Theorem~\ref{t:fourier_small_p} below. To this end, we define a new star-shaped set associated with a given convex body. It is essential for us that these sets are not necessarily compact. 
\begin{defn}
\label{def:F_pK}
    Let $n\geq 1$ and let $K\subset\R^n$ be a convex body. Then, for $p>0$, its Fourier $p$th mean body $F_p K \subset\R^n$ is the origin-symmetric star-shaped set given by the radial function
    \begin{equation}
        \rho_{F_p K}(\theta) =\left(\frac{p}{\vol_n(K)}\int_{0}^{+\infty}|\widehat{\chi_K}(r\theta)|^2r^{p-1}dr\right)^\frac{1}{p}, \quad \theta\in\s^{n-1}.
        \label{eq:fourier_body}
    \end{equation}
    \end{defn}

    We claim that for $0<p<n+1$, the function $\rho_{F_p K}$ is not identically infinite and, moreover, $F_p K$ has finite volume, i.e., that $F_p K$ is an $L^n$-star for $p\in (0,n+1)$. 
\begin{prop}
\label{p:mono}
    Let $n\geq 1$ and let $K$ be a convex body in $\R^n$. Then, the following holds.
    \begin{enumerate} 
    \item The set $F_n K$ is an $L^n$-star; in fact, $\vol_n(F_n K) = (2\pi)^n$. 
    \item 
    We have the set inclusion
    \begin{equation}F_p K \subset \vol_n(K)^{\frac{1}{p}-\frac{1}{q}} F_q K, \qquad 0<p<q< n+1.
    \label{eq:growing_sets}
    \end{equation}
    Combining (1) and (2) gives that $F_p K$ is an $L^n$-star for $p\in (0,n)$. By H\"older's inequality, $F_p K$ is then an $L^p$-star.
    \item Moreover, for $p\in(n,n+1)$, $F_pK$ is an $L^p$-star; by H\"older's inequality, $F_pK$ is then an $L^n$-star. In particular, $0<\vol_n(F_pK)<+\infty$.
    
        \item Finally, for $p\in (0,n+1)$, it holds that $$K\mapsto \vol_n(K)^\frac{p-n}{p}\vol_n(F_p K)$$ is an affine-invariant functional over the set of convex bodies in $\R^n$.
    \end{enumerate}
\end{prop}

With this knowledge in hand, we obtain the following Fourier-analytic connection between $R_p K$ and $F_p K$. In Section~\ref{sec:intersection_bodies}, Definition~\ref{d:p_intersection_fourier}, we define $I_p M$, what we call the $p$-intersection star of an $L^{n-p}$-star $M$; again the sets we consider are not necessarily compact. Note that $I:=2I_1$ when $n\geq2$.

\begin{thm}
\label{t:fourier_small_p}
    Let $n\geq 1$ and let $K\subset \R^n$ be a convex body and $p \in (0,n)$. Then, $\widehat{\rho_{R_p K}^p}=\frac{p}{n-p}\rho_{F_{n-p} K}^{n-p}$ and $\widehat{\rho_{F_{p} K}^{p}}=\frac{p(2\pi)^n}{n-p}\rho_{R_{n-p} K}^{n-p}$
    in the sense of tempered distributions. In particular:
    \begin{enumerate}
        \item  The function $\rho_{R_p K}^p$ is a positive-definite distribution, i.e., $R_p K$ is a $p$-intersection body:  \begin{equation}
    \label{eq:R_K_I_p_false_start}
    R_p K = (2\pi)^\frac{p-n}{p}I_p \left(F_{n-p}K\right).\end{equation}
        \item The function $\rho_{F_{p} K}^{p}$ is a positive-definite distribution, i.e., $F_p K$ is a $p$-intersection star:
        \begin{equation}
    \label{eq:F_K_I_p}
    F_{p} K = (2\pi)I_p \left(R_{n-p}K\right).\end{equation} 
    \end{enumerate}
\end{thm}

Next, we consider the Fourier characterization of $R_p K$ when $p\in (-1,0)$. The difficulty here is that the action of $\widehat{\rho_{R_p K}^p}$ is not merely integration in general. Indeed, we will show the following.
\begin{cor}
\label{c:radial_action}
Let $n\geq1$ and let $K\subset\R^n$ be a convex body. Then, for $-1<p<0$, the action of
$\widehat{\rho_{R_pK}^p}$ as a tempered distribution is
\[
\left\langle\widehat{\rho_{R_pK}^p},\varphi\right\rangle
=
\frac{p}{n-p}
\int_{\R^n}
\rho_{F_{n-p}K}(\xi)^{n-p}
\bigl(\varphi(\xi)-\varphi(0)\bigr)\,d\xi,
\qquad
\varphi\in\mathcal{S}(\R^n).
\]
\end{cor}
We demonstrate in Proposition~\ref{p:new_radial_fourier_identity} and Theorem~\ref{t:radial_fourier_neg_p} below that the above corollary can be expressed as a centroid-body-type identity and as a certain type of distributional identity (in the sense of $L^p$ embedding, see Definition~\ref{d:embed} below).

The fact that $F_p K$ is an $L^n$-star for $p\in (0,n+1)$ is unexpected; the asymptotics of $|\widehat{\chi_K}|$ vary substantially as a function of the geometry of $K$. This, in turn, affects when the integral in \eqref{eq:fourier_body} is finite \textit{for a fixed $\theta$}. However, L. Brandolini, S. Hofmann, and A. Iosevich showed that \cite[Theorem 1.1]{BHI03}, for a convex body $K\subset \R^n$, $
    \left(\int_{\s^{n-1}}
    \!\!|\widehat{\chi_K}(r\theta)|^2d\theta\right)^\frac{1}{2} \!=\!\mathcal{O}\left(\!r^{-\frac{n+1}{2}}\!\right)$ as $r\to \infty$. That is, averaging increases decay rates, which served as a guiding principle in our work. 
    
With this in mind, we turn to compactness, and, more specifically, study when $F_p K$ is a star body. To this end, we define
\begin{equation}
\label{eq:pck}
q(K)
:= \sup\Big\{q > 0 \;:\;
\sup_{\theta \in \s^{n-1}}
\int_{0}^\infty
\big|\widehat{\chi_K}(r\theta)\big|^2r^{q-1}dr < \infty
\Big\}.
\end{equation}
In the following proposition, we establish that $q(K)$ is precisely the threshold for compactness of $F_p K$, and that compactness coincides with the stronger property of being a star body, i.e., having a continuous radial function on $\s^{n-1}$.
\begin{prop}
\label{p:compact=star}
Let $n\geq 1$ and let $K\subset \R^n$ be a convex body. For every $p>0$, $F_p K$ is compact if and only if it is a star body. Moreover,
\[
q(K)
= \sup\Big\{r>0\;:\; F_r K\text{ is compact}\Big\}
= \sup\Big\{r>0\;:\; F_r K\text{ is a star body}\Big\}.
\]
In particular, $F_p K$ is a star body for $0<p<q(K)$.
\end{prop}

However, this parameter can be difficult to determine in practice. We therefore introduce another parameter, which can be estimated using classical Fourier-decay bounds.
\begin{thm}
\label{t:compactness}
Let $n\geq 1$ and let $K\subset\R^n$ be a convex body. Define
\begin{equation}
\label{eq:pk}
p(K):= \sup\Big\{\alpha \ge 0 \;:\;
\sup_{\theta \in \s^{n-1}}
\sup_{r \ge 1}
r^{\alpha}\,\big|\widehat{\chi_K}(r\theta)\big|^2 < \infty
\Big\}.
\end{equation}
Then, $p(K)$ enjoys the following properties:
\begin{enumerate}
    \item $p(K)$ is an affine-invariant quantity; and
    \item $F_p K$ is a star body for $p \in (0,p(K))$, that is, $p(K) \leq q(K)$.
\end{enumerate}
\end{thm}

We refer to the quantity $p(K)$ from Theorem~\ref{t:compactness} as the \textit{Fourier index of $K$}; it is nothing but the optimal decay rate of $\widehat{g_K}$. We have the following sharp bounds on the Fourier index of $K$.
\begin{thm}
\label{t:pk_iso}
Let $n \ge 1$ and let $K\subset \R^n$ be a convex body. Then,
\[
q([-1,1]^n)=p([-1,1]^n)=2\leq p(K)\leq q(K) \leq  n+1 = p(\B)=q(\B).
\]
Additionally, $F_2 [-1,1]^n$ is not compact, and $F_p K$ is unbounded for every $p\geq n+1$.
\end{thm}

Some other pertinent results on the asymptotics of $|\widehat{\chi_K}|$ can be found in \cite{HCS62,SMM98,SI71,BRT98}. Having discussed when $F_p K$ is a star body, we now turn to the natural question: \textit{is $F_p K$ convex?}
For $p\in (0,1]$, we answer this question in the affirmative. 

\begin{thm}
\label{t:main_convexity}
    Let $n\geq 1$ and let $K\subset \R^n$ be a convex body and let $p\in (0,1]$. Then, $F_p K$ is an origin-symmetric convex body.
\end{thm}

\begin{figure}[H]
  \centering
  \begin{subfigure}[b]{0.49\textwidth}
    \centering
    \includegraphics[width=\textwidth]{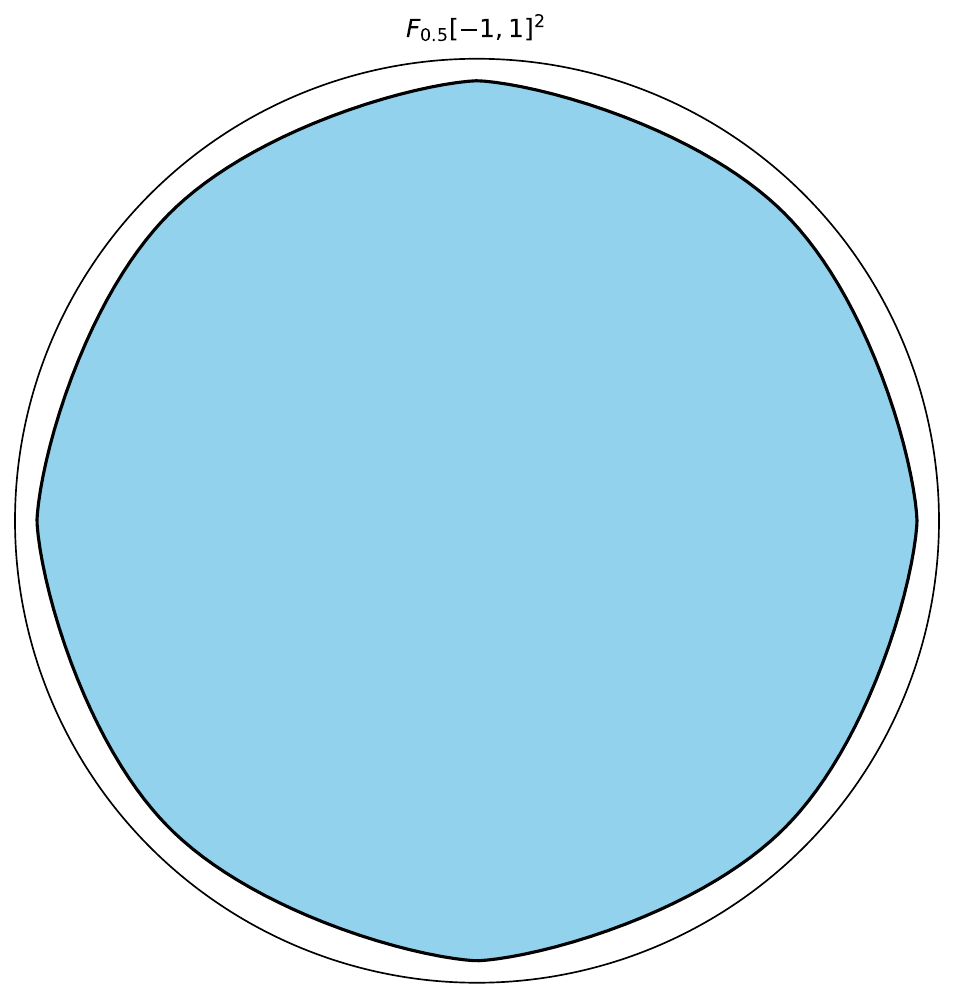}
    \label{fig:F0.5}
  \end{subfigure}
  \hfill
  \begin{subfigure}[b]{0.49\textwidth}
    \centering
    \includegraphics[width=\textwidth]{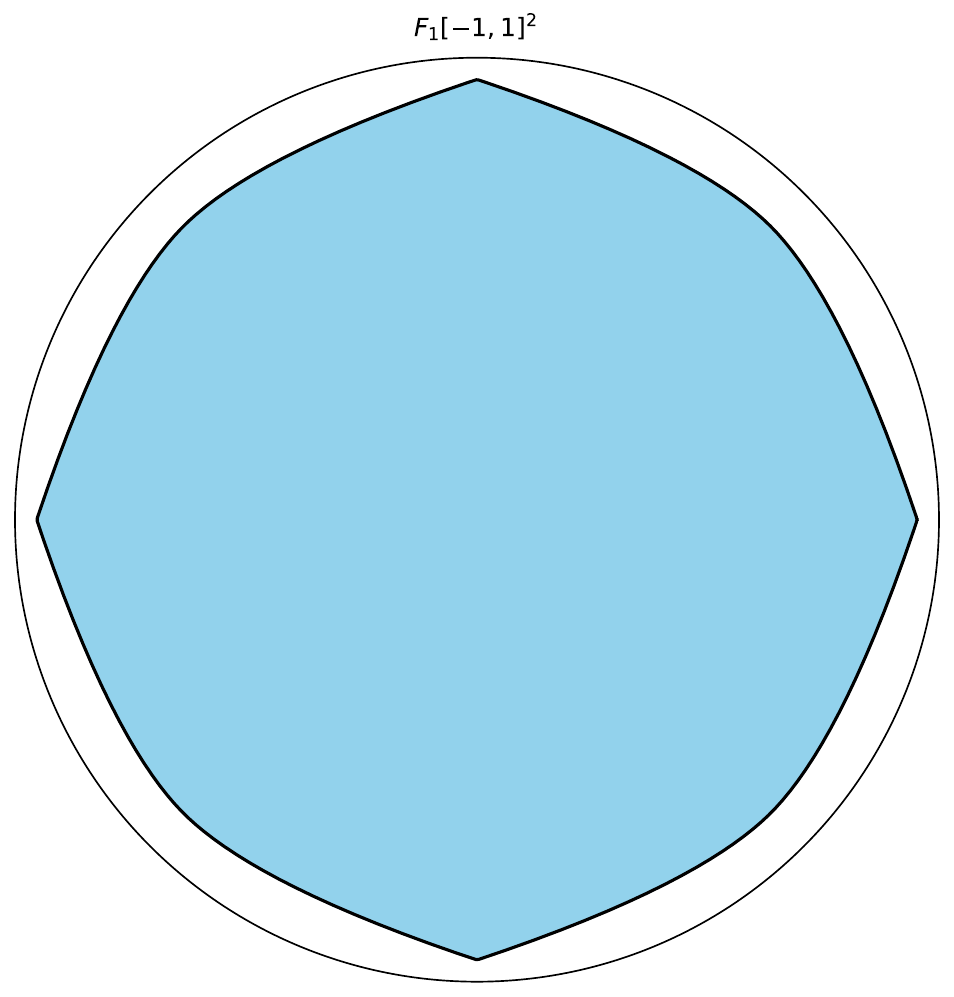}
    \label{fig:F1}
  \end{subfigure}
  \caption{The body $F_p K$ when $p \in (0,1]$ and $K=[-1,1]^2$. Convex.}
  \label{fig:F_qonvex}
\end{figure}

\noindent Concerning the $p=1$ case of Theorem~\ref{t:main_convexity}, we actually show the following more precise relation.

\begin{lem}
\label{l:F_1K_is_convex}
    Let $n\geq 1$ and let $K\subset\R^n$ be a convex body. Then, $F_1 K =\pi I(R_{n-1} K).$ In particular, $F_1 K$ is an origin-symmetric convex body. 
\end{lem}

The informed reader may recognize that $I(R_{n-1} K)$ is the first cross-sectional body of $K$, introduced in \cite{GG99}; see the discussion at the end of Section~\ref{sec:convex} for more information. Despite these results, the star-shaped set $F_p K$ is \textit{not} convex in general when $p>1$, even in the case of an origin-symmetric convex body. Our counterexample is the unit cube in $\R^n$, $n\geq 2$. Actually, it follows from the planar case. We dedicate Section~\ref{sec:non} to the proof of the following theorem, which is a technical and delicate examination of the radial function of $F_p [-1,1]^2$.
\begin{thm}
\label{t:cube_not}
    For $n\geq 2$ and $p\in (1,2)$, $F_p [-1,1]^n$ is not convex.
\end{thm}

In Lemma~\ref{l:F_1K_is_convex}, we saw that $F_1 K$ is a dilate of the intersection body of $R_{n-1} K$. We recall in \eqref{eq:R_pinvariance} that $R_{n-1} K$ is linear in $K$ with respect to linear transformations; one can always apply a linear transformation to $K$ so that $R_{n-1} K$ is isotropic. In this particular case, we then have by Hensley's theorem \eqref{eq:hen} that $\rho_{F_{1} K}$ is bounded from above and below by $L_{R_{n-1}K}^{-1}$, up to absolute constants. Motivated by this observation, we prove the following.

\begin{figure}[H]
\centering
\begin{subfigure}[b]{0.49\textwidth}
    \centering
    \includegraphics[width=\textwidth]{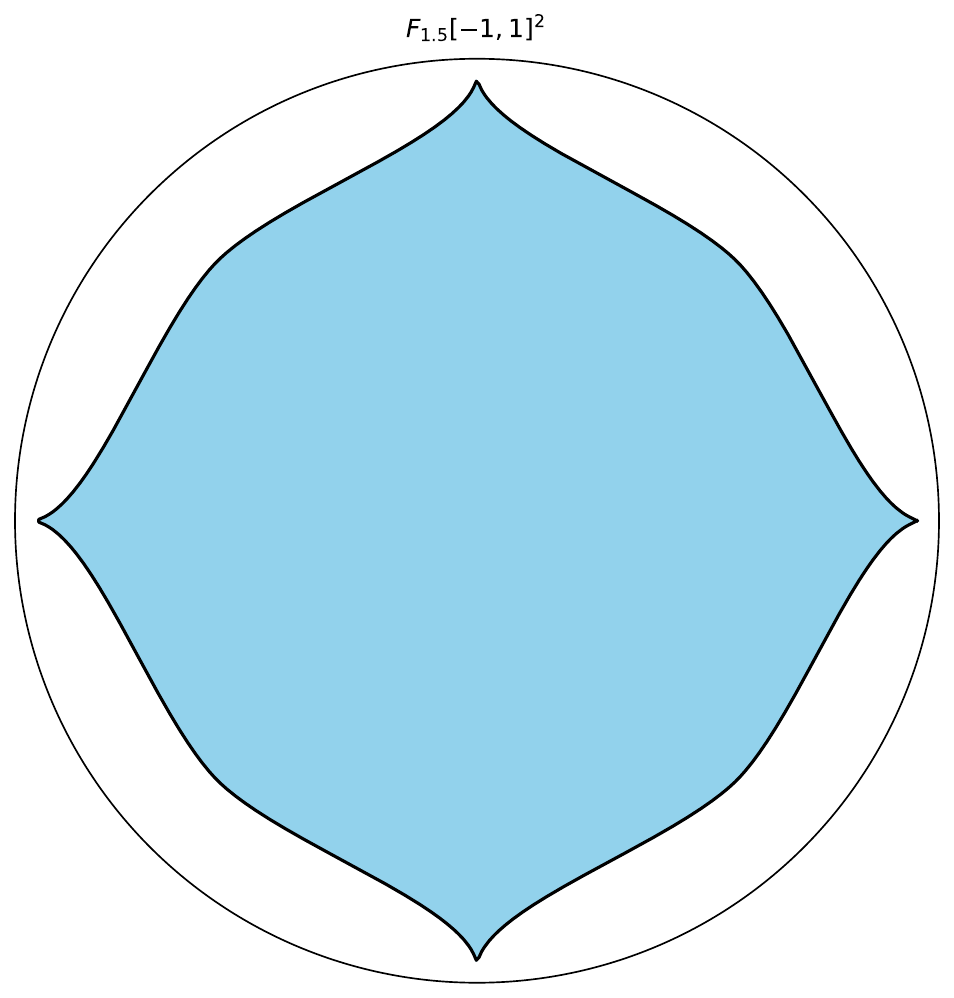}
    \label{fig:F1.5}
  \end{subfigure}
  \hfill
  \begin{subfigure}[b]{0.49\textwidth}
    \centering
    \includegraphics[width=\textwidth]{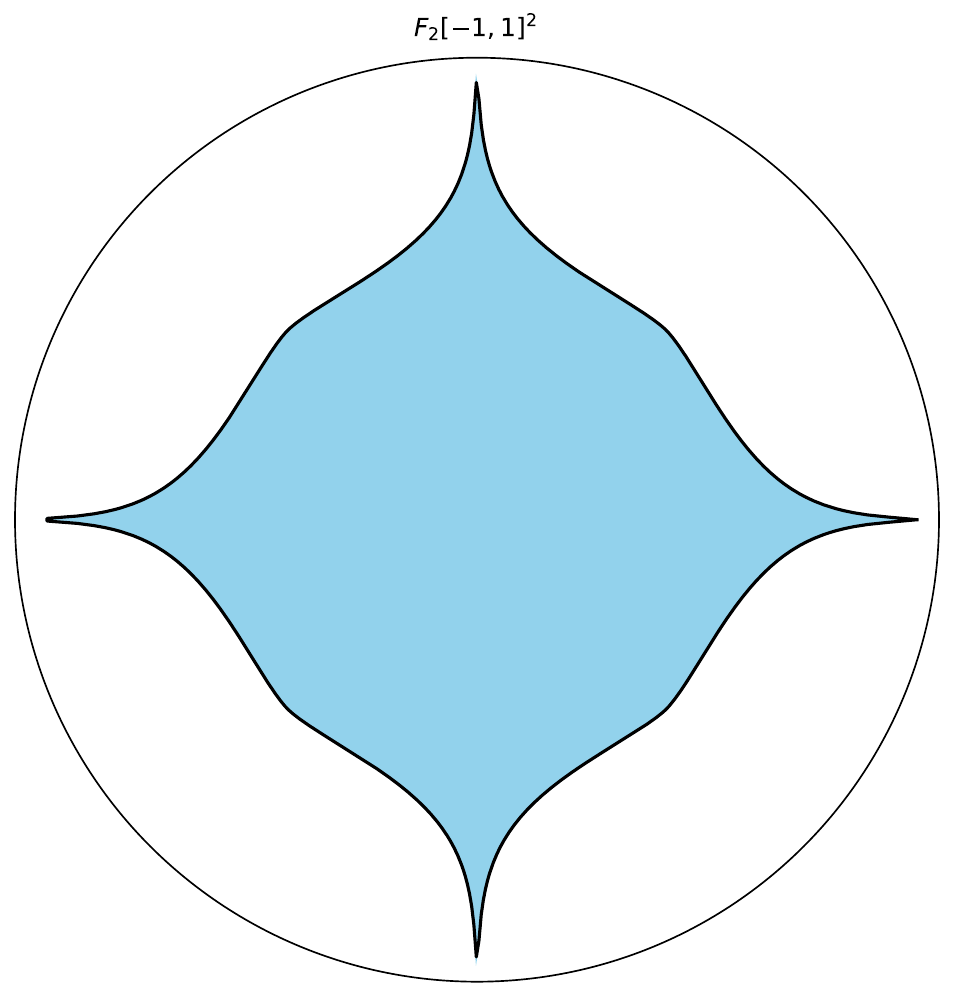}
    \label{fig:F2}
  \end{subfigure}
\caption{The body $F_p K$ when $p=1.5$ and $p=2$; $K=[-1,1]^2$. Not Convex. Note that $F_2[-1,1]^2$ actually contains the lines $\operatorname{span}\{e_1\}$ and $\operatorname{span}\{e_2\}$.}
  \label{fig:F_p_not_convex}
\end{figure}

\begin{thm}
\label{t:use_slicing}
    Let $n\geq 1$ and let $K\subset\R^n$ be an isotropic convex body. For every $0<p<2$,
\begin{equation}
\label{eq:depend_p}
\sqrt{2}e^{-3/4}\le L_K\cdot \rho_{F_pK}\le
\begin{cases}
\displaystyle\sqrt{\frac{2}{3}}\cdot \left(\sqrt{6}\cdot\Gamma\left(1+\frac{p}{2}\right)^{2}\right)^{\frac{1}{p}}, & 0<p\le 1,\\[6pt]
\displaystyle
\left(\frac{p\,4^{p-1}}{2-p}
\left(\frac{\pi}{2}\right)^{2-p}\right)^\frac{1}{p}, & 1<p<2.
\end{cases}
\end{equation}
Moreover, we have the following bounds for $0<p \leq 1$:
    \begin{equation}
    \label{eq:independ_p}
    \sqrt{2}e^{-3/4} \cdot L_K^{-1} \leq \rho_{F_p K} \leq \frac{\pi}{2} \cdot L_K^{-1}
    \end{equation}
\end{thm}

\noindent Since $F_2 [-1,1]^n$ is unbounded, the dependence on $p$ cannot be removed for $p$ close enough to $2$.

 Using the recent resolution of the hyperplane slicing conjecture by B. Klartag and J. Lehec \cite{KL25}, we now know that $L_K$ can be controlled from above and below by absolute constants when $K$ is isotropic (they considered the origin-symmetric case; see \cite{BK05} for the reduction to the general case). Thus, by Theorem~\ref{t:use_slicing}, there are two absolute constants $C_1,C_2>0$ such that  $$C_1 <\rho_{F_pK} <C_2, \quad \forall \, \, p\in (0,1].$$
Our methods produce a formal reversal of \eqref{eq:growing_sets} in the following proposition.
\begin{prop}
\label{p:decreasing_sets}
    Let $n\geq 1$, $0<p<q\leq 1$, and let $K\subset \R^n$ be a convex body. Then, 
    \begin{equation}\frac{2}{\pi}\vol_n(K)^{\frac{1}{p}-\frac{1}{q}} F_q K\subset F_p K \subset \vol_n(K)^{\frac{1}{p}-\frac{1}{q}} F_q K.
    \label{eq:decreasing_sets}
    \end{equation} 
\end{prop}

\subsection{Affine Isoperimetric-Type Inequalities}
Our investigation would not be complete without considering the relevant inequalities for the sets we consider. Very recently, J. Haddad and M. Ludwig \cite{HL26} established the following affine isoperimetric-type inequalities for the bodies $R_p K$. We note that the set $R_p B_2^n$ is a particular dilate of the unit Euclidean ball $\B$; see Proposition~\ref{p:R_pE} below.
\begin{theorem}[The Haddad--Ludwig inequalities]
\label{t:radial_isos}
    Let $n\geq 1$, $p>-1$, and let $K\subset\R^n$ be a convex body. Then, $\vol_n(R_n K)=\vol_n(K)$ and, for $p\neq n$,
\begin{align}
\label{eq:iso_R_pK}
\frac{\vol_n(R_pK)}{\vol_n(K)} &\leq \frac{\vol_n(R_pB_2^n)}{\vol_n(B_2^n)}, \quad -1<p<n.
\\
\frac{\vol_n(R_pK)}{\vol_n(K)} &\geq \frac{\vol_n(R_pB_2^n)}{\vol_n(B_2^n)}, \quad p>n.
\label{eq:iso_R_pK_two}
\end{align}
For $p\neq n$, there is equality if and only if $K$ is an ellipsoid.
\end{theorem}

Next, we have the relevant affine isoperimetric-type inequalities for $Z_p^\circ K$.
\begin{thm}
\label{t:polar_BS}
    Let $n\geq 1$, let $K$ be a convex body in $\R^n$, and let $$p\in (0,\infty)\cup \left\{ p\in (-1,0):n/|p| \in \N\right\}.$$ Then,
$$\vol_n(K)\vol_n(Z^\circ_pK) \leq \vol_n(B_2^n)\vol_n(Z^\circ_pB_2^n),$$
with equality if and only if $K$ is an ellipsoid.
\end{thm}
\noindent The inequality in Theorem~\ref{t:polar_BS} holds for $p=0$ as well by taking the limit in $p$; we carefully define $Z_0^\circ K$ in \eqref{eq:Z_o_circ} below. It is natural to ask if Theorem~\ref{t:polar_BS} can hold for all $p>-1$. Our proof relies on the $L^p$-Blaschke-Santal\'o inequality. This inequality was proven by E. Lutwak and G.~Zhang \cite{LZ97} for $p\geq 1$, and we recall it below in Theorem~\ref{t:LZ}. In terms of extensions to larger ranges of $p$, the current state-of-the-art, as shown by R. Adamczak, G. Paouris, P. Pivovarov, and P. Simanjuntak \cite{APPS24}, is for the range of $p$ mentioned in Theorem~\ref{t:polar_BS}. An extension of the $L^p$-Blaschke-Santal\'o inequality to $p>-1$ would extend Theorem~\ref{t:polar_BS} to the same range.

We consider analogues of the sharp reverse isoperimetric-type Gardner-Zhang inequalities for radial mean bodies (recalled in \eqref{eq:radial_set_inclusion} below) for polar mean zonoids and Fourier mean bodies; see Theorem~\ref{t:z_opposite_chain} and Corollary~\ref{c:F_set_inclusions} below, respectively.

For the sake of completeness, we wish to establish affine isoperimetric inequalities for the Fourier mean bodies. However, directly deriving sharp volume bounds for $F_p K$ via the Fourier transform in Definition~\ref{def:F_pK} presents formidable technical obstacles. To bypass this, we leverage a relationship uncovered in Lemma~\ref{l:F_p<1} below to map the problem into the setting of Theorem~\ref{t:polar_BS} when $0<p<1$. Finally, the case $p=1$ utilizes the Busemann intersection inequality.

    \begin{thm}
    \label{t:Fpk_affine_iso}
        Let $n\geq 1$ and let $K\subset\R^n$ be a convex body and $p\in (0,1]$ be such that $n/p\in \N$. Then,
        \begin{equation}
        \label{eq:F_p_affine}
        \vol_n(K)^\frac{p-n}{p}\vol_n(F_pK) \leq \vol_n(\B)^\frac{p-n}{p}\vol_n(F_p\B).
        \end{equation}
        There is equality if and only if $K$ is an ellipsoid.      
    \end{thm}
    \noindent We do not know if \eqref{eq:F_p_affine} can hold for any $p >1$.

This paper is organized as follows. Section~\ref{sec:prep} provides basic definitions, known facts, and their immediate consequences. Section~\ref{sec:concave} recalls concave functions and the related inequality by L. Berwald. In Section~\ref{sec:fourier_facts}, we mention pertinent facts from Fourier analysis. Section~\ref{sec:intersection_bodies} concerns intersection bodies and polar centroid bodies, including our extension of $p$th intersection bodies. In Section~\ref{sec:hen}, we list some inequalities satisfied by polar centroid bodies, including an extension of Hensley's theorem and the $L^p$--Blaschke-Santal\'o inequality. We then recall facts specific to the covariogram function and radial mean bodies in Section~\ref{sec:radial_mean_bodies}.

With these preparations concluded, we study the basic properties of Fourier $p$th mean bodies $F_p K$ in Section~\ref{sec:foundational}. To this end, we study in Section~\ref{sec:contravariance} their geometric invariances and prove Proposition~\ref{p:mono}. Next, in Section~\ref{sec:index}, we study compactness and the Fourier index, proving Proposition~\ref{p:compact=star} and Theorems~\ref{t:compactness} and~\ref{t:pk_iso}.

We then turn to our original inquiry, the Fourier characterization of radial $p$th mean bodies, in Section~\ref{sec:fourier_radial}. In Section~\ref{sec:polar_fourier}, we prove Theorem~\ref{t:p_n_+_1}, which concerns large $p$ relative to dimension and relates $R_p K$ to $Z_{p-n}^\circ K$. In Section~\ref{sec:fourier_fourier}, we establish the Fourier characterization of $R_p K$ for $p\in(0,n)$, proving Theorem~\ref{t:fourier_small_p} and relating $R_p K$ to $F_{n-p} K$. Then in Section~\ref{sec:fourier_neg}, we supply the characterization for $p\in(-1,0)$ using the concept of embedding in $L^{-p}$, establishing Corollary~\ref{c:radial_action}.

In Section~\ref{sec:fourier_bodies}, we study the geometry of Fourier mean bodies. We prove in Section~\ref{sec:isotropic} the isotropic estimates in Theorem~\ref{t:use_slicing}. Then, in Section~\ref{sec:convex}, we show that $F_p K$ is convex for $p\in (0,1]$ by proving Lemmas~\ref{l:F_p<1} and \ref{l:F_1K_is_convex}. 

In Section~\ref{sec:mono_iso}, we supply monotonicity and isoperimetric inequalities for our operators. In Section~\ref{sec:polar_functions}, we prove Theorem~\ref{t:polar_BS} for $Z_p^\circ K$. Our results for $F_p K$ are in Section~\ref{sec:mono}, which proves Proposition~\ref{p:decreasing_sets} and Theorem~\ref{t:Fpk_affine_iso}. 

In Section~\ref{sec:cube}, we analyze how the cube behaves under the operators $R_p$ and $F_p$. Firstly, in Section~\ref{sec:radial_cube}, we study the radial mean bodies of the cube, in particular their Banach-Mazur distance from the ball. Then, in Section~\ref{sec:non}, we supply the proof of the fact that, for $n\geq 2$, $F_p[-1,1]^n$ is not a convex body when $p>1$. Finally, we use translated cubes to prove in Proposition~\ref{p:valuation} that $F_p$ is neither an $L^p$ valuation nor an $L^p$ radial valuation for $p\in(0,2)\setminus\{1\}$.

\section{Preliminaries}
\label{sec:prep}
We introduce the notation, for $q>-2$, $\omega_q=\frac{\pi^\frac{q}{2}}{\Gamma(1+\frac{q}{2})}$; so, $\vol_n(\B)=\omega_n.$ The usual Gamma function is $\Gamma(z)=\int_0^\infty e^{-t}t^{z-1}dt$ and its logarithmic derivative, the digamma function, is $\psi(z)=(d/dz) \log \Gamma(z)=\Gamma^\prime(z)/\Gamma(z)$. The generalized binomial coefficients are 
\begin{equation}
\label{eq:generalized_binomial}
\binom{n+p}{n}^\frac{1}{p}=\begin{cases}\left(\frac{\Gamma(n+p+1)}{\Gamma(n+1)\Gamma(p+1)}\right)^\frac{1}{p}, & p\neq 0,
\\
e^{H_n}, & p=0,
\end{cases}\end{equation} 
where $H_n=\sum_{j=1}^n\frac{1}{j}$ is the $n$th harmonic number.

\subsection{On Functions with Concavity}
\label{sec:concave}
Recall that a non-negative, measurable function $g$ on $\R^n$ is not identically zero if its support $$\supp (g) = \overline{\{x\in\R^n:g(x)>0\}}$$ has positive Lebesgue measure. We say a non-negative function $f:\R^n\rightarrow\R$ is \textit{concave} if it is upper semi-continuous and concave on its support, i.e., it is not identically zero and for every $x,y\in\supp(f)$ and $t\in(0,1),$
 \begin{equation}f((1-t)x + t y) \geq (1-t)f(x)+tf(y).
\label{eq:s_concave}
\end{equation}

We will need a reverse H\"older-type inequality for integrals of powers of concave functions. This inequality, established by L. Berwald \cite{Berlem}, is known as Berwald's inequality. The equality conditions are from \cite{GZ98}; see also \cite{LP25}. To state them, we recall the \textit{tent function} of a convex body $K\subset \R^n$ containing the origin:
\begin{equation}
\label{eq:lk}
\ell_K(x)=\left(1-\|x\|_K\right)_+, \qquad x\in \R^n.
\end{equation}
Here, we use the extended gauge
\[
\|x\|_K:=\inf\{t>0:x\in tK\},
\]
with the conventions $\inf\varnothing=+\infty$ and $(1-\infty)_+=0.$

\begin{theorem}[Berwald's Inequality]
\label{t:berwald}
    Fix $n\geq 1$ and a concave function $f:\R^n\to \R_+$ with support a convex body $K\subset \R^n$. Define the continuous function $\mathcal G:(-1,\infty)\to\R_+$ by

    \begin{equation}\label{classical_Berwald} \mathcal G(p)=\begin{cases}\left({\binom{n+p}{p}}\int_Kf^p(x)\frac{dx}{\vol_n(K)}\right)^\frac{1}{p}, & p>-1,p\neq 0,
    \\
    e^{H_n}\exp\left(\int_{K}\log f(x)\frac{dx}{\vol_n(K)}\right), & p=0.
    \end{cases}
    \end{equation}
      Then, $\mathcal G$ is either strictly decreasing or constant. Furthermore, $\mathcal G$ is constant if and only if $f(x)=A\cdot \ell_{K-y}(x-y)$ for some $y\in K$ and $A>0$.
\end{theorem}

\subsection{Fourier Analysis}
\label{sec:fourier_facts}
Recall that a \textit{distribution} on $\R^n$ is a continuous linear functional on the space of smooth functions with compact support, $C^\infty_c(\R^n)$. Locally integrable functions on $\R^n$ define distributions by integration: for a complex-valued $f\in L^1_{\operatorname{loc}}(\R^n)$,
\[
\langle f,\varphi\rangle = \int_{\R^n}f(x)\varphi(x)dx, \qquad \varphi\in C_c^\infty(\R^n).
\] 
We say a distribution $T$ is a \textit{positive distribution} if $\langle T,\phi\rangle\geq0$ for every non-negative $\phi\in C^\infty_c(\R^n)$.

We say a function $\phi$ is a \textit{Schwartz test function}, or belongs to the Schwartz class, and write $\phi\in\mathcal{S}(\R^n)$, if $\phi\in C^\infty(\R^n)$ and $\phi$ along with all its derivatives decay at infinity faster than any inverse power of the $\ell_2$-norm. Equivalently, for every $k\in\N$,
\[
\sup_{|\alpha|\leq k}\sup_{x\in\R^n}(1+|x|)^k|D^\alpha\phi(x)|<+\infty,
\]
where $D^\alpha\phi=\frac{\partial^{|\alpha|}\phi}{\partial x_1^{\alpha_1}\cdots\partial x_n^{\alpha_n}}$, $|\alpha|=\sum_{i=1}^n\alpha_i$, and $\alpha\in\N^n$ is a multi-index. These seminorms define the usual topology on $\mathcal{S}(\R^n)$.

A distribution $T$ is called a \textit{tempered distribution} if its action on $C^\infty_c(\R^n)$ extends to a continuous linear functional on $\mathcal{S}(\R^n)$. This extension is unique because $C^\infty_c(\R^n)$ is dense in $\mathcal{S}(\R^n)$, and we continue to denote it by $T$. For a tempered distribution, positivity is equivalently tested against all non-negative functions in $\mathcal{S}(\R^n)$.

A non-negative, locally finite Borel measure $\mu$ on $\R^n$ defines a tempered distribution if and only if there exists $\beta>0$ such that
\begin{equation}
\label{eq:tempered}
\int_{\R^n}(1+|x|)^{-\beta}\,d\mu(x)<+\infty.
\end{equation}
Consequently, a complex-valued function $f\in L^1_{\operatorname{loc}}(\R^n)$ defines a tempered distribution whenever $d\mu(x)=|f(x)|\,dx$ satisfies \eqref{eq:tempered}. We continue to use the symbol $f$ for this tempered distribution; its action on the Schwartz class is via integration:
\[
\langle f,\phi\rangle=\int_{\R^n}f(x)\phi(x)\,dx,
\qquad \phi\in\mathcal{S}(\R^n).
\]

The Fourier transform of an integrable function $f:\R^n\to\mathbb{C}$ is the function $\hat f:\R^n\to\mathbb{C}$ given by
\begin{equation}
\label{eq:fourier}
\hat f(y):=\int_{\R^n}e^{-i\langle x,y\rangle}f(x)\,dx.
\end{equation}
The mapping $\phi\mapsto\hat\phi$ is an isomorphism on $\mathcal{S}(\R^n)$, and for every $\phi\in\mathcal{S}(\R^n)$, one has
\begin{equation}
\label{eq:inversion}
\widehat{\widehat{\phi}}(x)=(2\pi)^n\phi(-x).
\end{equation}
The Fourier transform of a tempered distribution $T$ is defined by Parseval's formula:
\begin{equation}
\label{eq:parseval_1}
\langle\hat T,\phi\rangle=\langle T,\hat\phi\rangle,
\qquad \phi\in\mathcal{S}(\R^n).
\end{equation}
Thus, $\hat T$ is again a tempered distribution, and $\hat T_1=\hat T_2$ implies $T_1=T_2$. Specializing \eqref{eq:parseval_1} to even $\phi\in\mathcal{S}(\R^n)$, one has from \eqref{eq:inversion} the identity
\begin{equation}
\label{eq:parseval_2}
\langle\hat T,\hat\phi\rangle=(2\pi)^n\langle T,\phi\rangle
\end{equation}
for every tempered distribution $T$. We say a tempered distribution $T$ is a \textit{positive-definite distribution} if its Fourier transform is positive, i.e., if $\langle\hat T,\phi\rangle\geq0$ for every non-negative $\phi\in\mathcal{S}(\R^n)$.

In this work, we will $-(n+q)$-homogeneously extend even, $L^1(\s^{n-1})$ functions from the sphere to $\R^n$ as tempered distributions. If $q<0$, then this extension can be defined in the usual way by $f(y)=f\left(\frac{y}{|y|}\right)|y|^{-n-q}, y \in \R^n\setminus\{o\},$ and is integrable near zero. For example, if $n\geq2$ and $L\subset \R^n$ is an origin-symmetric $L^1$-star, then its radial function $\rho_{L}$ is $(-1)$-homogeneously extended to $\R^n\setminus\{0\}$, which corresponds to $q=1-n$. More generally, if $L$ is an $L^p$-star, $\rho_L^p$ is $(-p)$-homogeneously extended to $\R^n\setminus\{0\}$ in the same manner, which corresponds to $q=p-n$, when $p\in (0,n)$.

However, if $q \geq 0$, then this extension is not locally integrable at the origin. Instead, we follow the convention from \cite[Sections 2--3]{GYY09}: the extension will be a tempered distribution obtained by analytic continuation in the parameter $q$, with its restriction to $\s^{n-1}$ held fixed. When $q\notin 2\N$, this is the unique even homogeneous extension. For $q\in 2\N$, we use the finite part, namely the constant term in the corresponding Laurent expansion. In particular, the Fourier transform of such homogeneously extended $L^1$ functions is the Fourier transform of these tempered distributions. This affects $\rho_L^p$ for $p\geq n$.

Throughout, identities involving Fourier transforms are understood in the sense of distributions on the stated domain. Functions appearing in these identities act by integration whenever locally integrable; otherwise, we use the distributional extension specified above.

Let $f$ be an integrable function that is also integrable on every hyperplane, and recall the \textit{Radon transform} of $f$ in the direction $\theta$ at a distance $t$ is given by
\begin{equation}
    \label{eq:og_radon}
\mathcal{R}f(\theta;t)=\int_{\{x\in\R^n:\langle x,\theta\rangle=t\}}f(x)dx.
\end{equation}
It follows from \cite[Lemma 2.11]{AK05} that, for a fixed $\theta\in\s^{n-1}$, we have the identity
\begin{equation}
    \hat{f}(r\theta)=\widehat{\mathcal{R}f(\theta;\cdot)}(r).
    \label{eq:og_radon_fourier}
\end{equation}
Also, the Fourier transform enjoys the following homogeneity property: for $f\in L^1(\R^n)$,
\begin{equation}
    \widehat{f(tx)}(y) = t^{-n}\hat{f}\left(\frac{y}{t}\right), \quad t>0, \;y\in\R^n.
    \label{eq:Fourier_factoring}
\end{equation}

Finally, let $f,g:\R^n\to \R_+$ be two measurable functions. Then, we recall that their convolution is precisely the function $(f\ast g):\R^n\to\R_+$ given by
$$(f\ast g)(y)=\int_{\R^n}f(x)g(y-x)dx.$$
Note that the Fourier transform of a convolution of integrable functions is the product of their Fourier transforms: $\widehat{(f\ast g)}=\hat f \cdot \hat g$.

\subsection{On Intersection Bodies and Polar Centroid Bodies}
\label{sec:intersection_bodies}
In this work, we will need the Busemann intersection inequality; see e.g. \cite[Corollary 9.4.5 and Remark 9.4.6 on pg. 373]{gardner_book}.
 \begin{theorem}[The Busemann Intersection Inequality]
 \label{t:BI}
      Let $n\geq 2$ and $M\subset \R^n$ be a star body. Then, 
 \begin{equation*}
     \frac{\vol_n(IM)}{\vol_n(M)^{n-1}} \leq \frac{\omega^n_{n-1}}{\omega^{n-2}_{n}},
     \end{equation*}
     with equality if and only if $n=2$ and $M$ is origin-symmetric or $n\geq 3$ and $M$ is an origin-symmetric ellipsoid.
 \end{theorem}

In \cite{AK99}, A. Koldobsky introduced the $p$-intersection body of an origin-symmetric star body $M$ for $p\in \N\cap (0,n)$. He provided a Fourier-inversion-type relation between $M$ and $I_p M$, and we use this formulation as the basis of the following definition.
\begin{defn}
\label{d:p_intersection_fourier}
Fix $n\geq 1$ and $p\in (0,n)$. Let $M\subset\R^n$ be an origin-symmetric $L^{n-p}$-star. We say an origin-symmetric, star-shaped set $K\subset \R^n$ is the $p$-intersection star of $M$ if its radial function satisfies
\[
    \rho_{K}^p = \frac{p}{(n-p)(2\pi)^p}\widehat{\rho_{M}^{n-p}}
    \]
    on $\R^n\setminus\{o\}$, where we allow radial functions to take the value $+\infty$. We write $K=I_p M$. Moreover, for $n\geq2$ and $p=1$, we define $IM:=2\,I_1M$.
\end{defn}
 We clarify that Definition~\ref{d:p_intersection_fourier} explicitly requires the right-hand side of the identity to be represented on $\R^n\setminus\{o\}$ by the non-negative, locally integrable function $\rho_K^p$. Note that $I_p M$ may not always exist, even when $p\in \N\cap (0,n)$. Indeed, for  $p=n-1$ and $M$ a star body that is not an intersection body, $\widehat{\rho_{M}^{n-p}}$ will not be positive (see \cite{AK05}). The case of Definition~\ref{d:p_intersection_fourier} when $M$ is a star body and $p$ is not an integer was first suggested by V. Yaskin \cite{YV14}. 

We will show that, if $M$ is an $L^{n-p}$ star such that $I_p M$ exists, then it is an $L^p$-star. To this end, we state a preliminary lemma, where we say a function $f$ is $q$-homogeneous, $q\neq 0$, if $f(tx)=t^qf(x)$ for every $t>0$ and $x\in\R^n\setminus\{o\}$.
\begin{lem}
\label{l:hom}
Let $n\geq 1$ and let $f:\R^n\to \R$ be a non-negative, measurable, $(-p)$- homogeneous function for some $p<n$. Then, one has
$$\int_{\R^n}f(x)e^{-\frac{|x|^2}{2}}dx=2^{\frac{n-p}{2}-1}\Gamma\left(\frac{n-p}{2}\right)\int_{\s^{n-1}}f(u)du.$$
\end{lem}
\begin{proof}
Integrating in polar coordinates, we directly compute:
\begin{align*}
\int_{\R^n}f(x)e^{-\frac{|x|^2}{2}}dx&=\int_{\s^{n-1}}\int_{0}^{+\infty}f(ru)r^{n-1}e^{-\frac{r^2}{2}}drdu
\\
&=\int_{0}^{+\infty}r^{n-p-1}e^{-\frac{r^2}{2}}dr\int_{\s^{n-1}}f(u)du.
\end{align*}
We conclude with the substitution $t=\frac{r^2}{2}$ and an application of the Gamma function. 
\end{proof}
\begin{thm}
    Fix $n\geq 1$ and $p\in (0,n)$ and let $M\subset\R^n$ be an origin-symmetric $L^{n-p}$-star with $p$-intersection star $I_p M$. Then, $I_p M$ is an $L^p$-star. In fact:
    \[
        \int_{\s^{n-1}}\rho_{I_pM}^p(\theta) d\theta =\frac{2^\frac{p}{2}p}{n-p}\frac{\Gamma\left(\frac{p}{2}\right)}{\Gamma\left(\frac{n-p}{2}\right)}\frac{\pi^\frac{n-p}{2}}{(2\pi)^\frac{p}{2}}\int_{\s^{n-1}}\rho_{M}^{n-p}(u)du.
    \]
\end{thm}
\begin{proof}
    By Definition~\ref{d:p_intersection_fourier}, $\rho_{I_pM}^p$ is locally integrable on $\R^n\setminus\{o\}$. Thus, by polar coordinate integration,
    \[
    \infty>
    \int_{1<|x|<2}\rho_{I_pM}(x)^p\,dx
    =\left(\int_1^2r^{n-p-1}\,dr\right)
    \int_{\s^{n-1}}\rho_{I_pM}(\theta)^p\,d\theta.
    \]
    It follows that $\rho_{I_pM}^p\in L^1(\s^{n-1})$. Since $0<p<n$, its homogeneous extension is also locally integrable at the origin, and uniqueness of the $(-p)$-homogeneous extension gives the identity in Definition~\ref{d:p_intersection_fourier} on all of $\R^n$.

    We now establish the asserted integral identity. We apply Lemma~\ref{l:hom} with $f=\rho_{I_pM}^p$ to obtain
    \begin{equation}
    \label{eq:hom}
    \int_{\s^{n-1}}\rho_{I_pM}^p(\theta) d\theta =\frac{1}{2^{\frac{n-p}{2}-1}\Gamma(\frac{n-p}{2})}\int_{\R^n}\rho_{I_pM}^p(x)e^{-\frac{|x|^2}{2}}dx.
    \end{equation}
    We apply Definition~\ref{d:p_intersection_fourier} to \eqref{eq:hom} and obtain
    \begin{equation}
        \label{eq:hom_fourier}
        \int_{\s^{n-1}}\rho_{I_pM}^p(\theta) d\theta =\frac{2p}{n-p}\frac{1}{\pi^p}\frac{1}{2^\frac{n+p}{2}\Gamma(\frac{n-p}{2})}\int_{\R^n}\widehat{\rho_{M}^{n-p}}(x)e^{-\frac{|x|^2}{2}}dx.
    \end{equation}
    We apply Parseval's formula \eqref{eq:parseval_1}, which is valid since $e^{-\frac{|\,\cdot\,|^2}{2}} \in \mathcal{S}(\R^n)$, and obtain
    \begin{equation}
        \label{eq:hom_fourier_e}
        \int_{\s^{n-1}}\rho_{I_pM}^p(\theta) d\theta =\frac{2p}{n-p}\frac{1}{\pi^p}\frac{1}{2^\frac{n+p}{2}\Gamma(\frac{n-p}{2})}\int_{\R^n}\rho_{M}^{n-p}(x)\widehat{e^{-\frac{|\,\cdot\,|^2}{2}}}(x)dx.
    \end{equation}
    Recall the classical fact that the Gaussian density is an eigenfunction of the Fourier transform:
    \begin{equation}
        \widehat{e^{-\frac{|\,\cdot\,|^2}{2}}}(x) = (2\pi)^\frac{n}{2}e^{-\frac{|x|^2}{2}}.
        \label{eq:fixed}
    \end{equation}
    Inserting \eqref{eq:fixed} into \eqref{eq:hom_fourier_e} yields the relation
    \[
        \int_{\s^{n-1}}\rho_{I_pM}^p(\theta) d\theta =\frac{2p}{n-p}\frac{\pi^\frac{n-p}{2}}{(2\pi)^\frac{p}{2}}\frac{1}{\Gamma(\frac{n-p}{2})}\int_{\R^n}\rho_{M}^{n-p}(x)e^{-\frac{|x|^2}{2}}dx.
    \]
    Using again Lemma~\ref{l:hom} allows us to return to spherical integration and deduce the claim.
\end{proof}

For $q>-1$, the $q$th cosine transformation of a measurable function $f$ on $\s^{n-1}$ is given by
\begin{equation}
    \mathcal{C}_q(f)(x) = \int_{\s^{n-1}}|\langle x,u\rangle|^qf(u)du,\quad x\in\R^n.
    \label{eq:cosine_function}
\end{equation}
We will need the Fourier transform of the cosine transform from \cite[Corollary 3.15]{AK05}, when $q$ is not an even integer and from  \cite[Equation 3.9]{GYY09}, when $q$ is an even integer. 

\begin{prop}
\label{p:Koldobsky_cosine}
   Let $n\geq1$, let $q>-1$, and let $f$ be an even, continuous function on $\s^{n-1}$. If $q$ is not an even integer, then the tempered distribution $\widehat{\mathcal{C}_q(f)}$ is represented on $\R^n\setminus\{0\}$ by a continuous homogeneous function of degree $-(n+q)$, given by, for $y\in\R^n\setminus\{0\}$,
   \begin{equation}
       \widehat{\mathcal{C}_{q}(f)}(y) = -4(2\pi)^{n-1}\Gamma\left(q+1\right)\sin\left(\frac{\pi q}{2}\right)f\left(\frac{y}{|y|}\right)|y|^{-n-q}.
       \label{eq:fourier_cosine}
   \end{equation}
   Furthermore, interpreting $f$ as its tempered-distributional extension of degree $-(n+q)$, Fourier inversion gives
   \begin{equation}
   \label{eq:fourier_cosine_relate}
   \hat f = -\frac{\pi}{2}\frac{1}{\Gamma\left(q+1\right)\sin\left(\frac{\pi q}{2}\right)}\mathcal{C}_{q}(f).
   \end{equation}
   If $q=0,2,4,\ldots$, the Fourier transform of the finite-part extension of $f$ is given by the following formula, where $\alpha_q=\psi(1)+H_q$:
    \[
   \hat f(x) = \frac{(-1)^\frac{q}{2}}{q!}\int_{\s^{n-1}}\left(\alpha_q-\ln |\langle x,u\rangle|\right)|\langle x,u\rangle|^qf(u)du.
   \]
\end{prop}

\begin{rem}
\label{eq:koldob_extend}
    Let $q>-1$, $q\notin 2\N$. The identity \eqref{eq:fourier_cosine} also holds, in the sense of distributions on $\R^n\setminus\{o\}$, when $f$ is an even function in $L^1(\s^{n-1})$. To see this, let $f_j$ be even, continuous functions converging to $f$ in $L^1(\s^{n-1})$. For $q>0$ and $x\in \R^n\setminus\{o\}$,
\[
\bigl|\mathcal{C}_{q}(f_j)(x)-\mathcal{C}_{q}(f)(x)\bigr|=\bigl|\mathcal{C}_{q}(f_j-f)(x)\bigr|
\leq
|x|^{q}\|f_j-f\|_{L^1(\s^{n-1})},
\]
which shows pointwise convergence and, by integration against Schwartz test functions, convergence in tempered distributions.

To handle $q\in (-1,0)$, observe that, for a Schwartz function $\psi\in \mathcal{S}(\R^n)$,
\begin{align*}
&\bigl|\langle\mathcal C_q(f_j),\psi\rangle-\langle\mathcal C_q(f),\psi\rangle\bigr|=\bigl|\langle\mathcal C_q(f_j-f),\psi\rangle\bigr|
\\
&\qquad\leq
\|f_j-f\|_{L^1(\s^{n-1})}
\sup_{u\in\s^{n-1}}
\int_{\R^n}|\langle x,u\rangle|^q|\psi(x)|\,dx
\longrightarrow0.
\end{align*}
The supremum is finite: using $|\psi(x)|\leq C_N(1+|x|)^{-N}$ with $N>n$ and rotating $u$ to $e_1$ reduces this to the integrability of $|x_1|^q(1+|x|)^{-N}$, which follows from $q>-1$.

Applying this to $\psi=\hat\varphi$, where $\varphi\in C_c^\infty(\R^n\setminus\{o\})$, we deduce
\[
\bigl|\langle  \widehat{\mathcal C_q(f_j)},\varphi\rangle-\langle\widehat{\mathcal C_q(f)},\varphi\rangle\bigr|=\bigl|\langle  \widehat{\mathcal C_q(f_j-f)},\varphi\rangle\bigr| = \bigl|\langle  \mathcal C_q(f_j-f),\hat{\varphi}\rangle\bigr|\longrightarrow 0.
\]
Meanwhile, on every annulus $0<a\leq |y|\leq b<\infty$,
\begin{align*}
&\int_{a\leq|y|\leq b}
\left|
f_j\!\left(\frac y{|y|}\right)
-f\!\left(\frac y{|y|}\right)
\right||y|^{-n-q}\,dy\\
&\qquad=
\|f_j-f\|_{L^1(\s^{n-1})}
\int_a^b r^{-q-1}\,dr
\longrightarrow0.
\end{align*}
    Since \eqref{eq:fourier_cosine} holds for every $f_j$, passing to the limit proves the same identity for $f$ in the sense of distributions on $\R^n\setminus\{o\}$.
\end{rem}

For an origin-symmetric star body $M$, in \eqref{eq:cosine_function}, set $f=\frac{1}{n+q}\frac{1}{\vol_n(M)}\rho_M^{n+q}$; using \eqref{eq:L_p} and using polar coordinates, we obtain
\begin{equation}
\label{eq:Lq_intersection_bodies}
\rho_{\Gamma_q^\circ M}^{-q}=\frac{1}{(n+q)\vol_n(M)}\mathcal{C}_{q}\left(\rho_M^{n+q}\right).
\end{equation}
By Proposition~\ref{p:Koldobsky_cosine}, \eqref{eq:fourier_cosine_relate}, the right-hand side of \eqref{eq:Lq_intersection_bodies} is a multiple of $\widehat{\rho_M^{n+q}}$ when $q>-1$ is not an even integer. This serves to extend the notion of $p$-intersection bodies to negative $p$, by identifying $p$ in Definition~\ref{d:p_intersection_fourier} with $(-q)$ in \eqref{eq:Lq_intersection_bodies}. In fact, in the literature, the bodies $\Gamma_q^\circ M$ are sometimes called $L^q$ intersection bodies, motivated further by revelations from \cite{HL06,CH08,KK05}. See \cite{LE90,GrZ99,YY06} for relevant Busemann-Petty type problems concerning these bodies.

\subsection{Inequalities for Polar Centroid Bodies}
\label{sec:hen}
In \cite[Theorem 3]{FM99}, M. Fradelizi showed the following extension of Hensley's theorem.
\begin{theorem}[The Fradelizi-Hensley Theorem]
\label{t:fradelizi}
    Let $n\geq 1$ and $q\geq 1$. Set
    \[
c(n,q)=\frac{n}{2}\binom{n+q}{n}^{-\frac{1}{q}}, \qquad c(q):= c(1,q)= \frac{1}{2} \frac{1}{(q+1)^\frac{1}{q}}.
\]

    Suppose $K$ is a convex body with center of mass at the origin. Then, for every $\theta\in \s^{n-1}$,
\begin{align*}
c(q)\left(\frac{1}{\vol_n(K)}\int_{K}|\langle x,\theta\rangle|^qdx\right)^{-\frac{1}{q}} &\leq \frac{\vol_{n-1}(K\cap \theta^\perp)}{\vol_n(K)} 
\\
&\leq c(n,q)\left(\frac{1}{\vol_n(K)}\int_{K}|\langle x,\theta\rangle|^qdx\right)^{-\frac{1}{q}}.
\end{align*}
There is equality in the first inequality if and only if $K$ is cylindrical in the direction $\theta$ and in the second inequality if and only if $K$ is a double-cone in the direction $\theta$.

In particular, if one selects isotropic $K$ and $q=2$ in Theorem~\ref{t:fradelizi}, then \eqref{eq:isotropic_constant} produces the formula
\begin{equation}
    \label{eq:fradelizi_2}
    c(2)L_K^{-1} \leq \vol_{n-1}(K\cap \theta^\perp) \leq c(n,2) L_K^{-1}, \qquad \theta \in \s^{n-1}.
\end{equation}
One can optimize over $n$ and show that 
\begin{equation}
\label{eq:dimensional_free}
0<c:=\frac{1}{2\sqrt{3}}=c(2)\leq c(n,2) < \lim_{n\to \infty}c(n,2)=\frac{1}{\sqrt{2}}=:b<\infty,
\end{equation}
completing the recovery of \eqref{eq:hen}.
\end{theorem}

We recall the $L^p$-Blaschke-Santal\'o inequality. It was originally proven, with equality characterization, by E. Lutwak and G. Zhang \cite{LZ97} for $p \geq 1$. See \cite{CG06} for another proof. It was subsequently proven by R. Adamczak, G. Paouris, P. Pivovarov, and P. Simanjuntak that the inequality holds for $p\in (0,1)$ \cite[Theorem 2.2]{APPS24} and for $\{p\in (-1,0) :n=pk,k\in \Z\}$ \cite[Theorem 2.3]{APPS24}.
\begin{theorem}[$L^p$-Blaschke-Santal\'o inequality]
\label{t:LZ}
Let $n\geq 1$, let $M\subset \R^n$ be a star body, and let $$p\in (0,\infty)\cup \left\{ p\in (-1,0):n/|p| \in \N\right\}.$$
Then,
\begin{equation*}
    \vol_n(M)\vol_n(\Gamma_p^\circ M) \leq \vol_n(\B)\vol_n(\Gamma_p^\circ \B).
\end{equation*}
    There is equality when $M$ is an origin-symmetric ellipsoid. Moreover, if $p\geq 1$, then this is the only case of equality.
\end{theorem}
Here, $\Gamma_p^\circ \B$ is a particular dilate of $\B$. 

\subsection{The Geometry of Radial Mean Bodies}
\label{sec:radial_mean_bodies}
In this section, we recall known facts about radial mean bodies that are pertinent to our investigations herein. First, we analyze the covariogram function $g_K$ of a convex body $K\subset \R^n$ from \eqref{eq:covario}.  The interested reader can see the recent survey \cite{GB23} for a complete and detailed review of the covariogram function.

Firstly, we note that 
\begin{equation}
\label{eq:convolution}
g_K=(\chi_K\ast \chi_{-K}), \qquad \supp(g_K)=DK.
\end{equation}
Therefore, 
\begin{equation}
    \label{eq:covario_Fourier}
    \widehat{g_K}(y)=\widehat{\chi_K}(y)\cdot \widehat{\chi_{-K}}(y)=|\widehat{\chi_K}(y)|^2, \qquad y\in \R^n.
\end{equation}
Also, for any $x \in \R^n$, 
\begin{equation}
g_K(x) \leq g_K(o)=\vol_n(K)=\frac{1}{\vol_n(K)}\int_{\R^n}g_K(x)dx.
\label{eq:covario_zero}
\end{equation} 
The Minkowski sum of two Borel sets $A,B\subset \R^n$ is precisely $$A+B=\left\{a+b:a\in A,b\in B\right\}.$$ The covariogram then factors with respect to Minkowski sums of sets from orthogonal spaces. Indeed, for $n_1,n_2\in \N$, suppose $M\subset \R^{n_1}$ and $K\subset \R^{n_2}$ are compact, convex sets. Then,
\begin{equation}
\label{eq:factor_g_K}
\begin{aligned}
g_{M+K}\left(x_1, x_2\right) & =\vol_{n_1+n_2}\left((M+K) \cap\left(M+K+\left(x_1, x_2\right)\right)\right) \\
& =\vol_{n_1+n_2}\left(\left(M \cap\left(M+x_1\right)\right)+\left(K \cap\left(K+x_2\right)\right)\right) \\
& =\vol_{n_1}\left(M \cap\left(M+x_1\right)\right) \vol_{n_2}\left(K \cap\left(K+x_2\right)\right) 
\\
&=g_M\left(x_1\right) g_K\left(x_2\right).
\end{aligned}
\end{equation}

Recall that the parallel section function of a convex body $K$ in direction $\theta\in\s^{n-1}$ is given by
\begin{equation}
    A_{K,\theta}(t)=\vol_{n-1}(K\cap (\theta^\perp+t\theta)) = \int_{\{x\in\R^n:\langle x,\theta \rangle=t\}}\chi_K(x)dx, \qquad t\in\R.
    \label{eq:parallel_section}
\end{equation}
It follows from \eqref{eq:og_radon_fourier}, the relation between the Fourier transforms and the Radon transform, with $f=\chi_K$ that
\begin{equation}
\label{eq:parallel_fubini}
\begin{split}
    \widehat{A_{K,\theta}}(r) = \widehat{\chi_K}(r\theta),
    \end{split}
\end{equation}
where the first Fourier transform is on $\R$ and the second is on $\R^n$. On the other hand, we can apply \eqref{eq:og_radon_fourier} with $f=g_K$ and deduce
\begin{equation}
    \widehat{g_K}(r\theta) = \widehat{\mathcal{R}g_K(\theta;\cdot)}(r), \qquad r\in \R.
    \label{eq:radon_covariogram}
\end{equation}
Combining \eqref{eq:parallel_fubini}, \eqref{eq:covario_Fourier}, and \eqref{eq:radon_covariogram}, we obtain
\begin{equation}
    \widehat{\mathcal{R}g_K(\theta;\cdot)}(r)=\widehat{g_K}(r\theta)= |\widehat{A_{K,\theta}}(r)|^2, \qquad \theta\in\s^{n-1},\; r\in\R.
    \label{eq:fourier_covario_parallel}
\end{equation}
Finally, by the Brunn-Minkowski inequality, $g_K^\frac{1}{n}$ is concave, i.e., satisfies \eqref{eq:s_concave}.

R. Gardner and G. Zhang established \cite[Theorem 5.5]{GZ98} reverse affine isoperimetric inequalities for the radial $p$th mean bodies $R_p K$, with extremizers being simplexes. Their result can be stated as the following set-inclusions: 
if $-1<p<q<\infty$,
\begin{equation}
    DK\subseteq  \binom{n+q}{n}^\frac{1}{q}R_q K \subseteq  \binom{n+p}{n}^\frac{1}{p}R_p K\subseteq n\vol_n(K)\Pi^\circ K,
    \label{eq:radial_set_inclusion}
\end{equation}
with equality in any, and hence all, set inclusions if and only if $K$ is an $n$-dimensional simplex. Here, the polar projection body $\Pi^\circ K$ of $K$ is the origin-symmetric convex body given by the radial function $$\rho_{\Pi^\circ K}(\theta)=\vol_{n-1}(P_{\theta^\perp} K)^{-1},$$ where $P_{\theta^\perp} K$ is the orthogonal projection of $K$ onto $\theta^\perp$. 

R. Gardner and G. Zhang \cite{GZ98} showed that the interaction between $R_p K$ and affine transformations is precisely \cite[Theorem 2.3]{GZ98}, for $T\in \operatorname{GL}_n(\R)$ and $x_0\in\R^n$,
\begin{equation}
\label{eq:R_pinvariance}
    R_p(TK+x_0)=TR_pK.
\end{equation}
By \eqref{eq:R_pinvariance}, $R_p \B$ is a dilate of $\B$, whose radius can be computed directly; see G. Xiong and W.S. Cheung \cite[Theorem 3]{XC08} for the details when $p\neq 0$.
\begin{prop}
\label{p:R_pE}
    Let $n\geq 1$ and consider an ellipsoid $E:=T\B +x_0$ for $T\in \operatorname{GL}(n)$ and $x_0\in\R^n$. Then, 
    \begin{equation*}
    R_p E =\left(\frac{2^{p+1} \omega_{n+p}}{(p+1) \omega_n \omega_{p+1}}\right)^\frac{1}{p}T\B, \quad p>-1,\; p\neq 0.
\end{equation*}
For the case $p=0$, we have by taking the limit as $p\to 0$:
\begin{equation}
    R_0 E=2 \exp\left(\frac{\psi\left(\frac{1}{2}\right)-\psi\left(\frac{n}{2}+1\right)}{2}\right)T\B.
\end{equation}
\end{prop} 
From the definition, it follows that if $L,M\subset \R^n$ are star bodies such that $L\subset M$, then $IL\subset IM$. A similar monotonicity holds for $R_p K$, after rescaling. Indeed, if $K,M$ are convex bodies such that $K\subset M$, then $g_K\leq g_M$ and, therefore, it follows from the definition \eqref{eq:radial_ell} that
\begin{equation}
\label{eq:monotonic_R_pK}
R_pK \subset \left(\frac{\vol_n(M)}{\vol_n(K)}\right)^\frac{1}{p} R_p M, \qquad p>0.
\end{equation}
We mention now that some recent extensions of radial $p$th mean bodies can be found in \cite{LRZ22,LP25,HLPRY25,LX24,HL26}.  

\section{Basic Properties of Fourier Mean Bodies}
\label{sec:foundational}
\subsection{Affine Properties and Finite Volume}
\label{sec:contravariance}
We first establish that the operator $F_p$ is a translation invariant, $\operatorname{GL}(n)$ contravariant mapping of degree $\frac{n-p}{p}$. 

\begin{prop}
\label{p:contravariance}
    Let $n\geq 1$ and let $K\subset \R^n$ be a convex body and $p>0$. Then, $K\mapsto F_p K$ is translation invariant. Additionally, for every $T\in \operatorname{GL}(n)$,
    \begin{equation}
        F_pTK = |\operatorname{det}(T)|^\frac{1}{p}T^{-t}F_p K.
        \label{eq:gl_relate}
    \end{equation}
    In particular, by taking $T$ to be a multiple of the identity matrix, 
    \begin{equation}
\label{eq:homogeneous}
F_p (cK) = c^\frac{n-p}{p} F_p K, \quad c>0.
\end{equation}
\end{prop}
\begin{proof}
The translation invariance is due to the fact that $|\widehat{\chi_{K+x}}|=|\widehat{\chi_{K}}|$ for all $x\in\ \R^n$. 
    It suffices to prove \eqref{eq:gl_relate} at the level of radial functions.
    Observe that 
    \begin{align*}
        \widehat{\chi_{TK}}(y) &= \int_{\R^n} e^{-i\langle x,y\rangle} \chi_{TK}(x)dx = \int_{\R^n} e^{-i\langle T^{-1}x,T^ty\rangle} \chi_{K}(T^{-1}x)dx
        \\
        &= |\operatorname{det}(T)| \int_{\R^n} e^{-i\langle z,T^ty\rangle} \chi_{K}(z)dz = |\operatorname{det}(T)|\widehat{\chi_K}\left(T^ty\right).
    \end{align*}
    We deduce from this computation that 
    \[
    \frac{1}{\vol_n(TK)}|\widehat{\chi_{TK}}(r\theta)|^2 = |\operatorname{det}(T)|\frac{1}{\vol_n(K)}|\widehat{\chi_{K}}(rT^t\theta)|^2.
    \]
    Inserting this into the formula \eqref{eq:fourier_body} for the radial function of $F_p K$, we derive the identity
    \begin{equation}
        \rho_{F_p TK}(\theta) = |\operatorname{det}(T)|^\frac{1}{p}\rho_{F_p K}(T^t\theta), \quad \forall \; \theta\in\s^{n-1}.
        \label{eq:gl_radial_form}
    \end{equation}
    The following two facts follow from the definition of radial function: for a star-shaped set $M\subset \R^n,$ $c>0$ and $T\in \operatorname{GL}(n)$, one has $\rho_M(T\theta) = \rho_{T^{-1}M}(\theta)$ and $c\rho_M = \rho_{cM}.$ Therefore, \eqref{eq:gl_radial_form} implies, $\rho_{F_p TK} = \rho_{|\operatorname{det}(T)|^\frac{1}{p}T^{-t}F_p K},$ which, in turn, implies the claimed covariance. 
\end{proof}

The transformation law in Proposition~\ref{p:contravariance} is also satisfied by the $L^p$ projection body $\Pi_pK$ when the origin lies in the interior of $K$. For $n\geq2$ and $p>1$, Ludwig classified homogeneous, $\operatorname{SL}(n)$-contravariant $L^p$-Minkowski valuations on convex polytopes containing the origin \cite[Theorem 2p]{ML05}. However, we show in Proposition~\ref{p:valuation} below that $F_p$ is not an $L^p$-valuation, nor an $L^p$ radial valuation, when $0<p<2, p\neq 1$. See also \cite{HC09} for related studies.

An immediate consequence of Proposition~\ref{p:contravariance} is that, if $D$ is a linear image of $K$, and $F_p K$ is an $L^p$-star, then so too is $F_p D$. We now establish that $F_p K$ is both an $L^n$-star and an $L^p$-star when $0<p<n+1$.
\begin{proof}[Proof of Proposition~\ref{p:mono}]
For the first claim, we obtain from polar coordinates, \eqref{eq:covario_Fourier} and \eqref{eq:covario_zero}
    \[
    \vol_n(F_n K) = \frac{1}{\vol_n(K)}\int_{\R^n} |\widehat{\chi_K}(x)|^2dx = \frac{1}{\vol_n(K)}\int_{\R^n} \widehat{g_K}(x)dx = (2\pi)^n.
    \]

We now show the third claim. Take $p\in (n,n+1)$. We show that, for such $p$, $F_pK$ is an $L^p$-star. Indeed, by Tonelli's theorem, we have
\begin{align*}
\frac{\vol_n(K)}{p}
\int_{\s^{n-1}}\rho_{F_pK}^p(\theta)\,d\theta
&=
\int_{\s^{n-1}}\int_0^\infty
|\widehat{\chi_K}(r\theta)|^2r^{p-1}\,dr\,d\theta
\\
&=
\int_0^\infty\int_{\s^{n-1}}
|\widehat{\chi_K}(r\theta)|^2r^{p-1}\,d\theta\,dr.
\end{align*}
For $n\geq2$, the Brandolini--Hofmann--Iosevich estimate
\cite[Theorem 1.1]{BHI03} gives
\[
\int_{\s^{n-1}}
|\widehat{\chi_K}(r\theta)|^2\,d\theta
\lesssim r^{-n-1},
\qquad r\geq1.
\]
For $n=1$, the same bound follows directly from the Fourier transform of an interval. Therefore, by splitting the integral, and using that $|\widehat{\chi_K}(r\theta)|^2\leq\vol_n(K)^2$, we have
\begin{align*}
\frac{\vol_n(K)}{p}
\int_{\s^{n-1}}\rho_{F_pK}^p(\theta)\,d\theta
&\lesssim
\int_0^1 r^{p-1}\,dr
+
\int_1^\infty r^{p-n-2}\,dr
\\
&=
\frac1p-\frac1{p-n-1}
=
\frac{n+1}{p(n+1-p)}
<\infty,
\end{align*}
because $p<n+1$. We use this to show that $F_pK$ is an $L^n$ star: by H\"older's inequality, we have, since $n<p$,
\begin{align*}
\vol_n(F_pK)
&=\frac1n\int_{\s^{n-1}}\rho_{F_pK}^n(\theta)\,d\theta
\\
&\leq
\frac{\vol_{n-1}(\s^{n-1})^\frac{p-n}{p}}{n}
\left(
\int_{\s^{n-1}}\rho_{F_pK}^p(\theta)\,d\theta
\right)^\frac{n}{p} <+\infty.
\end{align*}

We now show the second claim. Consider a bounded, non-negative, measurable function $f$ on $[0,\infty)$ and define
    $$M_p(f) = \frac{p}{\|f\|_{L^\infty(\R_+)}}\int_0^\infty f(r)r^{p-1}dr.$$
    Then, it was proved by V. Milman and A. Pajor \cite[Lemma 2.1]{MP89} that $p\mapsto (M_p(f))^\frac{1}{p}$ is increasing on $\{p>0:0<M_p(f)<+\infty\}$, and is constant if and only if $f=c\chi_{[0,a]}$ for some $a,c>0$, up to null sets.
    
    Consequently, by setting $f(r)=\frac{1}{\vol_n(K)}|\widehat{\chi_K}(r\theta)|^2$, and observing that $\|f\|_{L^\infty(\R_+)} = \vol_n(K),$ we obtain the pointwise inequality $$\rho_{F_p K} \leq \vol_n(K)^{\frac{1}{p}-\frac{1}{q}}\rho_{F_q K}, \qquad 0<p<q<n+1,$$
    and the inequality is strict when both sides are finite. Thus, $$F_p K \subset \vol_n(K)^{\frac{1}{p}-\frac{1}{q}}F_q K.$$
    
    As a consequence, we deduce that $F_p K$ is an $L^n$-star for $0< p< n,$ since the preceding inclusion yields
    \[
    0 < \vol_n(F_p K) < \vol_n(K)^{\frac{n-p}{p}}\vol_n(F_nK) <+\infty.
    \]
    Finally, by H\"older's inequality, $F_p K$ is an $L^p$ star:
    \[
\int_{\s^{n-1}}\rho_{F_pK}^{p}\,d\theta
\leq
\left(\int_{\s^{n-1}}\rho_{F_pK}^{n}\,d\theta\right)^\frac{p}{n}
\vol_{n-1}\left(\s^{n-1}\right)^\frac{n-p}{n}
<\infty.
\]
    
     Finally, the affine-invariance of $K\mapsto \vol_n(K)^\frac{p-n}{p}\vol_n(F_p K)$ follows from Proposition~\ref{p:contravariance} and the fact that $0<\vol_n(F_pK) < +\infty$ for $p\in (0,n+1)$.
\end{proof}

An interesting consequence of Proposition~\ref{p:mono} is that $$p\longmapsto \vol_n(K)^{1-\frac{n}{p}}\vol_n(F_p K)$$ is increasing: for $0<p<q$,
\begin{align*}
        \vol_n(K)^{1-\frac{n}{p}}\vol_n(F_p K) &\leq \vol_n(K)^{1-\frac{n}{p}}\vol_n\left(\vol_n(K)^{\frac{1}{p}-\frac{1}{q}}F_q K\right)
        \\
        & = \vol_n(K)^{1-\frac{n}{q}}\vol_n\left(F_q K\right).
    \end{align*}

\subsection{Indices for Fourier Mean Bodies}
\label{sec:index}
We now focus our study on more qualitative characterizations of $F_p K$. We will need the following lemma.
\begin{lem}
\label{l:radial_usc}
Let $n\geq 1$ and let $L\subset\R^n$ be a compact, star-shaped set with non-empty interior. Then, its radial function $\rho_L$ is upper semicontinuous on $\s^{n-1}$.
\end{lem}
\begin{proof}
Let $\theta_j\to\theta$ in $\s^{n-1}$. Since $L$ is bounded, we may pass to a subsequence such that
\[
\rho_L(\theta_j)\to
\ell:=\limsup_{j\to\infty}\rho_L(\theta_j)<\infty.
\]
By compactness, $\rho_L(\theta_j)\theta_j\in L$, and therefore $\rho_L(\theta_j)\theta_j\to\ell\theta\in L$. It follows that
\[
\limsup_{j\to\infty}\rho_L(\theta_j)
=\ell\leq\rho_L(\theta),
\]
which proves upper semicontinuity.
\end{proof}

Recall that $q(K)$ from \eqref{eq:pck} is the supremum of the exponents $p>0$ for which $F_pK$ is bounded. We now show that it is also the threshold for $F_pK$ to be a star body.
\begin{proof}[Proof of Proposition~\ref{p:compact=star}]
For $p>0$, we have
\[
\rho_{F_p K}(\theta)^p 
=\frac{p}{\vol_n(K)}\int_0^\infty |\widehat{\chi_K}(r\theta)|^2r^{p-1}\,dr,
\qquad \theta\in\s^{n-1}.
\]
Notice that $\rho_{F_p K}^p$ is positive on $\s^{n-1}$ because $\widehat{\chi_K}(0)=\vol_n(K)>0$ and $\widehat{\chi_K}$ is continuous. Moreover, Fatou's lemma gives
\[
\rho_{F_p K}(\theta)^p \leq\liminf_{j\to\infty}\rho_{F_p K}(\theta_j)^p 
\qquad\text{whenever }\theta_j\to\theta.
\]
Thus, $\rho_{F_p K}^p$, and therefore $\rho_{F_p K}$ due to the continuity and monotonicity of $t\mapsto t^\frac{1}{p}$ away from zero, is lower semicontinuous on $\s^{n-1}$. 

With this available, we show that $F_p K$ being compact occurs precisely when $F_p K$ is a star body. Thus, we must show that, if $F_p K$ is compact, then its radial function is continuous. Indeed, this is the case: by the above, $\rho_{F_p K}$ is lower semicontinuous on $\s^{n-1}$, and, by Lemma~\ref{l:radial_usc}, $\rho_{F_p K}$ is upper semicontinuous on $\s^{n-1}$.

We now establish the asserted characterization of $q(K)$. Let $0<p<q(K)$. By the definition of $q(K)$, there exists $q>p$ such that
\[
M_q(K):=\sup_{\theta\in\s^{n-1}}\rho_{F_q K}(\theta)^q <\infty.
\]
For every $R\geq1$, we have
\begin{equation}
\label{eq:tail_estimate}
\sup_{\theta\in\s^{n-1}}\!\frac{p}{\vol_n(K)}\!
\int_R^\infty \!\!\!\!|\widehat{\chi_K}(r\theta)|^2r^{p-1}dr
\leq \frac{p}{q}R^{p-q}M_q(K)\longrightarrow0
\,\text{as }R\to\infty.
\end{equation}
For each fixed $R$, the function
\[
\theta\mapsto
\frac{p}{\vol_n(K)}\int_0^R|\widehat{\chi_K}(r\theta)|^2r^{p-1}\,dr
\]
is continuous by the dominated convergence theorem, using the dominating function $p\vol_n(K)r^{p-1}$ on $[0,R]$. The tail estimate \eqref{eq:tail_estimate} shows that these truncated integrals converge uniformly to $\rho_{F_p K}(\theta)^p$. Hence $\rho_{F_p K}(\theta)^p$ is finite and continuous on $\s^{n-1}$, and therefore $F_p K$ is a star body.

The converse direction is obvious: if $F_pK$ is a star body, then its radial function is continuous and hence bounded on the compact sphere. Thus, the bound defining $q(K)$ holds, so $p\leq q(K)$.
\end{proof}

We next consider the Fourier index of ellipsoids; as we will see, $\vol_n(F_p E)<\infty$ for every ellipsoid $E$ when $p \in (0,n+1)$. In fact, $F_p E$ is an origin-symmetric ellipsoid for such $E$ and $p$. We denote by $SO(n)$ the special orthogonal group on $\R^n$.
\begin{prop}
\label{p:ball}
    Let $n\geq 1$ and let $E\subset\R^n$ be an ellipsoid. Then, $F_p E$ is compact if and only if $p\in (0,n+1)$. In particular, for such $p$, if $E=T\B+x_0$ for some $T\in \operatorname{GL}_n(\R)$ and $x_0\in \R^n$, then $F_p E$ is the origin-symmetric ellipsoid given by 
    \[
    F_p E =\left(
\frac{(2\pi)^p\omega_{2n-p}\,\omega_{n-p}^{\,2}}{\omega_n\,\omega_p\,\omega_{2(n-p)}}
\right)^\frac{1}{p} |\operatorname{det}(T)|^\frac{1}{p} T^{-t}\B.
    \]
    In summary, $p(E)=q(E)=n+1$ and $F_p E$ is an origin-symmetric convex body for $0<p<n+1$. Moreover, $F_{n+1}(E)=\R^n$.
\end{prop}
\begin{proof}
Assume that $n\geq 2$. Then, for every $O\in SO(n)$, we have $$F_p \B = F_p O\B=O^{-t}F_p \B,$$ establishing that $F_p \B$ is rotationally invariant, i.e., either a centered Euclidean ball or all of $\R^n$. 

We will use the Bessel function $J_\mu$ from \eqref{eq:Bessel} and some of its related formulas in the Appendix~\ref{sec:appendix}. From \eqref{eq:stein} (with $\varphi=\chi_{[-1,1]}$ and $f=\chi_{\B}$) and \eqref{eq:fourier_identity} with ($\nu=0$ and $\mu=\frac{n-2}{2}$), it follows
\begin{equation*}
        \widehat{\chi_{\B}}(x) = (2\pi)^\frac{n}{2} |x|^{1-\frac{n}{2}} \int_0^1 J_{\frac{n-2}{2}}(|x|t)t^\frac{n}{2}dt = \left(\frac{2\pi}{|x|}\right)^\frac{n}{2}J_\frac{n}{2}(|x|).
    \end{equation*}
    By the Bessel asymptotic \eqref{eq:Bessel_asym_high}, $|\widehat{\chi_{\B}}(r\theta)|\lesssim r^{-\frac{n+1}{2}}$ uniformly in $\theta\in\s^{n-1}$ for $r\geq1$, and this exponent is optimal; thus, $p(\B)=n+1$.
    Next, we obtain the formula for $\rho_{F_p \B}$ by direct substitution:
    \begin{equation}
    \label{eq:F_pB}
        \rho_{F_p \B}(\theta) = \left(\frac{p(2\pi)^n}{\omega_n}\int_0^{+\infty}J_\frac{n}{2}(r)^2r^{p-n-1}dr\right)^\frac{1}{p}, \quad \theta\in\s^{n-1},
    \end{equation}
    which is a constant independent of $\theta$. By \eqref{eq:Bessel_mellin_2} with $\mu=\frac{n}{2}$ and $\nu=p-n$, \eqref{eq:F_pB} converges if and only if $p\in (0,n+1)$, and thus, $q(\B)=n+1$ and $F_{n+1}(\B)=\R^n$; furthermore,
    \[
    \int_0^{+\infty}J_\frac{n}{2}(r)^2r^{p-n-1}dr =  \frac{\Gamma\left(n-p+1\right)}{p2^{n-p}\Gamma\left(1+\frac{n-p}{2}\right)^2}\frac{\Gamma\left(1+\frac{p}{2}\right)}{\Gamma\left(n+1-\frac{p}{2}\right)}.
    \]
    Consequently, we have by \eqref{eq:F_pB},
    \[\rho_{F_p \B} = \left(
\frac{(2\pi)^p\omega_{2n-p}\,\omega_{n-p}^{\,2}}{\omega_n\,\omega_p\,\omega_{2(n-p)}}
\right)^\frac{1}{p} \Longrightarrow
    F_p \B =\left(
\frac{(2\pi)^p\omega_{2n-p}\,\omega_{n-p}^{\,2}}{\omega_n\,\omega_p\,\omega_{2(n-p)}}
\right)^\frac{1}{p} \B.
    \]

    For $n=1$, it is easy to see that $|\widehat{\chi_{[-1,1]}}(t)|^2=4\left|\frac{\sin t}{t}\right|^2, t>0$. The function $p\mapsto 2p\int_0^{\infty}\left|\frac{\sin t}{t}\right|^2 t^{p-1}\,dt$ is finite, in fact real-analytic, precisely for $p\in(0,2)$. We deduce $p([-1,1])=q([-1,1])=2$. Also, for $p\in(1,2)$,
\[
2p\int_0^{\infty}\sin^2(t)\,t^{p-3}\,dt
= -\frac{2p}{p-2}\int_0^{\infty} t^{p-2}\sin(2t)\,dt
= p\,\frac{\Gamma(p-2)}{2^{p-2}}\cos\!\left(\frac{\pi p}{2}\right),
\]
holds by integration by parts and \eqref{eq:sine_identity} (with $a=2$). The right-hand side has a removable
singularity at $p=1$; moreover, by \eqref{eq:gamma_cosine_reflection} and \eqref{eq:duplication},
$
\frac{\Gamma(p-2)}{2^{p-2}}\cos\!\left(\frac{\pi p}{2}\right)
=\frac{1}{2-p}\,
\frac{\Gamma\!\left(\frac{p}{2}\right)\Gamma\!\left(\frac12\right)}
{\Gamma\!\left(\frac{3-p}{2}\right)}.
$
Since both sides are real-analytic on $(0,2)$ and agree on $(1,2)$, the identity extends to
all $p\in(0,2)$:
\[
2p\int_0^{\infty}\sin^2(t)\,t^{p-3}\,dt
=\frac{p}{2-p}\,
\frac{\Gamma\!\left(\frac{p}{2}\right)\Gamma\!\left(\frac12\right)}
{\Gamma\!\left(\frac{3-p}{2}\right)}, \qquad 0<p<2.
\]
    Therefore, $\rho_{F_p[-1,1]} = \left( \frac{p}{2-p}\,
\frac{\Gamma\!\left(\frac{p}{2}\right)\Gamma\!\left(\frac12\right)}
{\Gamma\!\left(\frac{3-p}{2}\right)}\right)^\frac{1}{p}$, and, thus $$F_p [-1,1]=\left(\frac{p}{2-p}\,
\frac{\Gamma\!\left(\frac{p}{2}\right)\Gamma\!\left(\frac12\right)}
{\Gamma\!\left(\frac{3-p}{2}\right)}\right)^\frac{1}{p}[-1,1],$$ which is, in fact, the claimed formula.

    We translate these results to those for ellipsoids by applying Proposition~\ref{p:contravariance}.
\end{proof}

We establish our main theorem regarding the Fourier index $p(K)$ defined in \eqref{eq:pk}.

\begin{proof}[Proof of Theorem~\ref{t:compactness}]
Let $0<p<p(K)$. On the one hand, we have that $\widehat{\chi_K}$ is bounded by $\vol_n(K)$. On the other hand, by the definition of $p(K)$, there exist $\alpha>p$ and $C>0$ such that, for every $\theta\in\s^{n-1}$,
\[
|\widehat{\chi_K}(r\theta)|^2r^{p-1}
\leq
\begin{cases}
\vol_n(K)^2r^{p-1}, & 0<r\leq1,\\
Cr^{p-1-\alpha}, & r>1.
\end{cases}
\]
The right-hand side is integrable on $(0,\infty)$ because $p>0$ and $\alpha>p$. Therefore, we have shown that $0<\rho_{F_p K}<\infty$ on $\s^{n-1}$, where positivity follows from $\widehat{\chi_K}(0)=\vol_n(K)>0$ and continuity.

Moreover, the preceding arguments show we have a uniform integrable bound (in $\theta$) for the integrand $|\widehat{\chi_K}(r\theta)|^2 r^{p-1}$; since $\widehat{\chi_K}$ is a continuous function, the dominated convergence theorem yields that $\rho_{F_p K}$ is continuous on $\s^{n-1}$. In particular, since $\s^{n-1}$ is compact, $F_p K$ is bounded with non-empty interior. It follows that the star-shaped set $F_p K$ is, in fact, a star body. Proposition~\ref{p:contravariance} yields the affine invariance of $p(K)$.
\end{proof}

We now set out to show that $2\leq p(K)\leq n+1$.

\begin{proof}[Proof of Theorem~\ref{t:pk_iso}]
It is classical that $|\widehat{\chi_K}(\xi)|\lesssim |\xi|^{-1}$ for every convex body, and that this bound is sharp for the cube (see, e.g., \cite[Section 3.2.1]{BGT14}); hence $p(K)\ge2$, with $p([-1,1]^n)=2$. 

Notice also (see, e.g., \eqref{eq:cube_fourier} below) that
\[
\int_1^\infty
|\widehat{\chi_{[-1,1]^n}}(re_1)|^2r\,dr
=
2^{2n}\int_1^\infty\frac{\sin^2 r}{r}\,dr
=\infty,
\]
so $F_2[-1,1]^n$ is not compact. Comparing the weights $r^{p-1}$ and $r$ on $[1,\infty)$ gives divergence for every $p\geq 2$, and hence $q([-1,1]^n)=2$. When $n=1$, every convex body is an affine image of $[-1,1]$; Proposition~\ref{p:contravariance} therefore yields $p(K)=q(K)=2$ and the noncompactness of $F_2K$.

 We now show that $p(K) \leq p(\B)$. To this end, by Proposition \ref{p:ball} and Theorem~\ref{t:compactness}, we need to show that $q(K) \leq n+1$. By the definition \eqref{eq:pck}, it suffices to show that
 \begin{equation} \label{eq:assumption_goal}
        \int_{\s^{n-1}} \int_{0}^\infty
\big|\widehat{\chi_K}(r\theta)\big|^2r^{n}dr\,d\theta=\int_{\R^n} |\xi| | \widehat{\chi_K} (\xi) |^2 d \xi =  \infty.
    \end{equation}
 
 Suppose, toward contradiction, that
    \begin{equation} \label{eq:assumption-p}
        \int_{\R^n} |\xi| | \widehat{\chi_K} (\xi) |^2 d \xi = \int_{\s^{n-1}} \int_{0}^\infty
\big|\widehat{\chi_K}(r\theta)\big|^2r^{n}dr\,d\theta < \infty.
    \end{equation}
    Since $g_K$ is even, and since $\widehat{g_K}=|\widehat{\chi_K}|^2\in L^1(\R^n)$, we have by Fourier inversion
    \begin{equation}
    \label{eq:covariogram_fourier_inversion}
        g_K (x) = \frac{1}{(2\pi)^n} \int_{\R^n} \widehat{g_K} (\xi) e^{-i \langle x,\xi\rangle} d\xi = \frac{1}{(2\pi)^n} \int_{\R^n} | \widehat{\chi_K} (\xi) |^2 \cos \langle x,\xi\rangle d\xi.
    \end{equation}
    Write $x = tu$ for $t >0$ and $u \in \s^{n-1}$. Then,
    \[
        \frac{g_K(tu) - g_K(o)}{t}
        = \frac{1}{(2\pi)^n} \int_{\R^n} | \widehat{\chi_K} (\xi) |^2 \frac{\cos (t \langle u,\xi\rangle)  -1 }{t} d\xi.
    \]
    Since $1- \cos z \leq |z|$, we have
    \begin{align}
        \left| \int_{\R^n} | \widehat{\chi_K} (\xi) |^2 \frac{\cos (t \langle u,\xi\rangle)  -1 }{t} d\xi \right|
        & \leq \int_{\R^n} | \widehat{\chi_K} (\xi) |^2 \left|\frac{\cos (t \langle u,\xi\rangle)  -1 }{t} \right|d\xi
        \\
        & \leq \int_{\R^n} | \widehat{\chi_K} (\xi) |^2 \left|\xi \right|d\xi,
    \end{align}
    which is finite from \eqref{eq:assumption-p}. 
    Since $ 1- \cos z = 2\sin^2 (z/2) \leq z^2 /2$,
    we obtain
    \[
        \left|\frac{\cos (t \langle u,\xi\rangle)  -1 }{t} \right| \leq \frac{|t| |\langle u,\xi\rangle|^2}{2}.
    \]
    We may therefore apply the dominated convergence theorem to get
    \[
        \frac{g_K(tu) - g_K(o)}{t} \to 0, \qquad t \to 0^+.
    \]
    This is a contradiction, as it was shown by Matheron \cite{MA} (see also \cite{GZ98,LRZ22}) that
    \[
    \lim_{t\to0^+} \frac{g_K(tu) - g_K(o)}{t}=-\vol_{n-1}(P_{u^\perp} K)\neq 0.
    \]

Finally, we show that $F_{n+1} K$ is not compact. To see this, 
\eqref{eq:assumption_goal} implies that
\[
\int_{\s^{n-1}}\rho_{F_{n+1}K}(\theta)^{n+1}d\theta=+\infty.
\]
Therefore, $\|\rho_{F_{n+1}K}\|_{L^\infty(\s^{n-1})}=+\infty$; else, since $\s^{n-1}$ is compact, we would get $\rho_{F_{n+1}K}\in L^p(\s^{n-1})$ for all $p\geq 1$, a contradiction. Hence $F_{n+1} K$ is unbounded.
\end{proof}

\section{Fourier Characterization of Radial Mean Bodies}
In this section, we provide our Fourier characterization of radial mean bodies, which connects the Fourier transform of their radial functions to the appropriate power with the radial function of another class of star-shaped sets. The specific class depends on the range of $p$.

\label{sec:fourier_radial}
\subsection{The Range \texorpdfstring{$p>n-1$}{p>n-1}}
\label{sec:polar_fourier}
We start by proving Theorem~\ref{t:p_n_+_1}, which shows that the polar $(p-n)$th mean zonoids and radial $p$th mean bodies, for $p>n-1$, are connected by the Fourier transform of their radial functions to the appropriate powers. With this range in mind, we define the coefficient function
\begin{equation}
\label{eq:m}
    m(n,p) =\begin{cases}
        -\frac{p-n}{p}\frac{2}{\pi}\Gamma\left(p-n\right)\sin\left(\frac{\pi (p-n)}{2}\right), & \text{if} \; p\neq n+2k, k\in \N,
        \\
        \frac{1}{p}(p-n)!(-1)^\frac{p-n}{2}, & \text{if} \; p=n+2k, k\in\N.
    \end{cases} 
\end{equation}
We see that $m(n,n)=\frac1n$. First, we use the $q$th cosine transform from \eqref{eq:cosine_function}, to compute the Fourier transform of $\rho_{R_p K}^p$ and obtain the following result.
\begin{lem}
\label{l:fourier_large_p}
    Let $n\geq 1$ and $p>n-1$. Then, for $K\subset\R^n$ a convex body, with $\rho_{R_pK}^{p}$ interpreted as a tempered distribution:
    \begin{enumerate}
        \item If $p\neq n+2k$ for any $k\in \N$, then
        \begin{equation}
        \widehat{\rho_{R_p K}^p}(x)=\frac{1}{m(n,p)}\frac{1}{p}\int_{\s^{n-1}}|\langle u, x \rangle|^{p-n}\rho_{R_p K}^p(u)du.
        \label{eq:radial_large_p}
    \end{equation}
    \item If $p = n+2k$ for some $k\in \N$, then
 \begin{equation}
        \widehat{\rho_{R_p K}^p}(x)=\frac{1}{m(n,p)}\frac{1}{p}\int_{\s^{n-1}}\left( \alpha_{p-n}-\ln |\langle u,x \rangle|\right)|\langle u,x\rangle|^{p-n}\rho_{R_p K}^p(u)du.
        \label{eq:radial_large_p_even} 
    \end{equation}

    \end{enumerate}
    
\end{lem}
\begin{proof}
Both instances are direct applications of Proposition~\ref{p:Koldobsky_cosine} with the choice $f = \rho_{R_p K}^p$ and $q = p-n$.
\end{proof}

\begin{proof}[Proof of Theorem~\ref{t:p_n_+_1}]
    Observe that
    \begin{align*}
        \frac{1}{p}\int_{\s^{n-1}}|\langle u, x \rangle|^{p-n}\rho_{R_p K}^p(u)du 
        &= \int_{\s^{n-1}}|\langle u, x \rangle|^{p-n}\int_0^{\rho_{R_p K}(u)}t^{p-1}dtdu
        \\
         &= \int_{\s^{n-1}}\int_0^{\rho_{R_p K}(u)}|\langle tu, x \rangle|^{p-n}t^{n-1}dtdu
        \\
        & = \int_{R_p K}|\langle y, x \rangle|^{p-n}dy
        \\
        &= \vol_n(R_{p}K) \rho_{\Gamma_{p-n}^\circ (R_p K)}^{n-p}(x)
        \\
         &= \vol_n(K)\rho_{Z^\circ_{p-n} K}^{n-p}(x),
    \end{align*}
    where we used equation~\eqref{eq:mean_zonoid_relate}. The claim follows from Lemma~\ref{l:fourier_large_p}.
\end{proof}

We see there is overlap between Theorem~\ref{t:p_n_+_1} and Theorem~\ref{t:fourier_small_p} when $p\in (n-1,n)$; we verify that Lemma~\ref{l:fourier_large_p} yields the result that $\rho_{R_p K}^p$ is positive-definite for such $p$. Indeed, for $p$ in this range, $\rho_{R_pK}^p$, as a function acting on test functions via integration, agrees with its $(-p)$-homogeneous extension as a tempered distribution. Then, the condition $n-1<p<n$ rewrites as $-\frac{\pi}{2}<\frac{\pi}{2}(p-n)<0$. The function $\sin(t)$ is negative for such $t$. But $\Gamma(t)$ is also negative when $t\in (-1,0)$, and therefore we deduce that $\widehat{\rho_{R_p K}^p}$ is a positive distribution on $\R^n$.

\subsection{The Range \texorpdfstring{$0<p<n$}{0<p<n}}
\label{sec:fourier_fourier}
Having properly demonstrated in Section~\ref{sec:foundational} that $F_p$ is a well-defined operator from the set of convex bodies to the set of origin-symmetric $L^n$-stars when $p\in (0,n+1)$, we are almost ready to prove Theorem~\ref{t:fourier_small_p}. 
\begin{lem}
\label{l:pos_dist_alt}
    Let $n\geq 1$, $p\in (0,n)$, and let $L\subset \R^n$ be an $L^p$-star. Then, $\rho_{L}^p$ is a positive, tempered distribution on $\R^n$.
\end{lem}
\begin{proof}
 Recall we must show that $\rho_{L}^p\in L^1_{\operatorname{loc}}(\R^n)$ and that the measure $d\mu(x)=\rho_{L}^p(x)dx$ satisfies \eqref{eq:tempered} for some $\beta>0$.

  For the integrability requirement, it suffices to show that $$\int_{rB_2^n}\rho_{L}^p (z) dz <\infty, \qquad \forall \, r>0.$$ From polar coordinates, we deduce that this integral is
 \begin{align*}
     \int_{\s^{n-1}}\int_{0}^r\rho_{L}^p (t\theta) t^{n-1}dtd\theta &= \int_{\s^{n-1}}\rho_{L}^p (\theta)\left(\int_{0}^rt^{n-p-1}dt\right)d\theta
     \\
     &=\frac{r^{n-p}}{n-p}\int_{\s^{n-1}}\rho_{L}^p (\theta) d\theta,
 \end{align*}
which is finite by hypothesis. 

To show temperedness, fix $\beta>0$ to be determined and observe that, by polar coordinate integration and the $(-1)$-homogeneity of radial functions,
\begin{align*}
 \int_{\R^n}(1+|x|)^{-\beta}\rho_{L}^p(x)dx &= \int_{\s^{n-1}}\int_0^\infty(1+r)^{-\beta}\rho_{L}^p(r\theta)r^{n-1}drd\theta
 \\
 &=\int_{\s^{n-1}}\rho_{L}^p(\theta)d\theta\cdot \int_0^\infty(1+r)^{-\beta}r^{n-p-1}dr,
 \end{align*}
 where we used Fubini's theorem in the final line. The first integral is always finite by hypothesis, and the second integral is finite for any $\beta>n-p$. Indeed, in this case, one can use the variable substitution $u=\frac{r}{r+1}$ and get
 \[
 \int_0^\infty(1+r)^{-\beta}r^{n-p-1}dr = \frac{\Gamma(n-p)\Gamma(\beta-n+p)}{\Gamma(\beta)} <\infty
 \]
from an application of the Beta function. We conclude.
\end{proof}

We apply the above lemma to Fourier mean bodies.
\begin{cor}
\label{c:pos_dist}
    Let $n\geq 1$, $p\in (0,n)$, and let $K\subset \R^n$ be a convex body. Then, $\rho_{F_p K}^p$ is a positive, tempered distribution on $\R^n$.
\end{cor}
\begin{proof}
    By Proposition~\ref{p:mono}, $F_p K$ is an $L^p$ star. Then, the claim follows by Lemma~\ref{l:pos_dist_alt}.
\end{proof}

We now establish our Fourier characterization for $p\in (0,n)$.
\begin{proof}[Proof of Theorem~\ref{t:fourier_small_p}]
Let $\phi\in\mathcal{S}(\R^n)$ be a test function. Then, by Parseval's formula \eqref{eq:parseval_1}, the formula \eqref{eq:radial_ell} for the radial functions of $R_p K$, and Fubini's theorem, we deduce the equalities
\begin{align*}
        \langle \widehat{\rho_{R_p K}^p},\phi \rangle &= \langle \rho^p_{R_p K},\hat\phi \rangle = \int_{\R^n}\rho_{R_p K}^p(x)\hat{\phi}(x)dx
        \\
        & = \frac{p}{\vol_n(K)}\int_{\R^n}\left(\int_{0}^{+\infty}g_K(rx)r^{p-1}dr\right)\hat{\phi}(x)dx
        \\
        & = \frac{p}{\vol_n(K)}\int_{0}^{+\infty}\left(\int_{\R^n}g_K(rx)\hat{\phi}(x)dx\right)r^{p-1}dr.
\end{align*}
Next, we use Parseval's formula \eqref{eq:parseval_1}, and then the formula \eqref{eq:Fourier_factoring} concerning the Fourier transform of a dilate of a function to obtain
\begin{equation}
\begin{split}
        \langle \widehat{\rho_{R_p K}^p},\phi \rangle &= \frac{p}{\vol_n(K)}\int_{0}^{+\infty}\langle \widehat{g_K(rx)},\phi(x)\rangle r^{p-1}dr
        \\
        &= \frac{p}{\vol_n(K)}\int_{0}^{+\infty}\left\langle \widehat{g_K}\left(\frac{x}{r}\right),\phi(x)\right\rangle r^{p-n-1}dr
        \\
        &= \frac{p}{\vol_n(K)}\int_{0}^{+\infty}\left\langle \widehat{g_K}\left(rx\right),\phi(x)\right\rangle r^{n-p-1}dr
        \\
        &= \frac{p}{\vol_n(K)}\int_{0}^{+\infty}\left\langle |\widehat{\chi_K}\left(rx\right)|^2,\phi(x)\right\rangle r^{n-p-1}dr
        \\
        &= \frac{p}{n-p} \left\langle \rho_{F_{n-p}K}^{n-p},\phi\right\rangle,
    \end{split}
    \label{eq:RF_actions}
    \end{equation}
    where we used a change of variables $r\mapsto \frac{1}{r}$, the Fourier transform of $g_K$ \eqref{eq:covario_Fourier}, and, finally, Fubini's theorem in conjunction with the definition of $F_{n-p}K$ from \eqref{eq:fourier_body}. We have thus established, in the sense of distributions, the identity
    \[
    \rho_{F_{n-p}K}^{n-p} = \frac{n-p}{p}\widehat{\rho_{R_p K}^p}.
    \]
   We deduce from Definition~\ref{d:p_intersection_fourier}, after replacing $p$ with $n-p$, that $F_{p} K$ is a $p$-intersection star (of a dilate of $R_{n-p} K$). By Corollary~\ref{c:pos_dist}, $\rho_{F_{n-p}K}^{n-p}$ is a positive, tempered distribution, and therefore we have shown that $\rho_{R_p K}^p$ is a positive-definite distribution. 

    On the other hand, from an application of Parseval's formula \eqref{eq:parseval_2}, we have that for every even test function $\phi\in\mathcal{S}(\R^n)$, it holds by \eqref{eq:RF_actions}
\begin{equation}
\label{eq:last_parseval}
    \langle \widehat{\rho_{F_p K}^p}, \phi \rangle= \langle \rho_{F_p K}^p,\hat\phi \rangle = \frac{p}{n-p} \langle \widehat{\rho_{R_{n-p} K}^{n-p}},\hat\phi \rangle=\frac{p(2\pi)^n}{n-p}\langle \rho_{R_{n-p} K}^{n-p},\phi \rangle,
\end{equation}
    which yields the claim. It sufficed to only check \eqref{eq:last_parseval} along even test functions since $\widehat{\rho_{F_p K}^p}$ and $\rho_{R_{n-p} K}^{n-p}$ are, themselves, even distributions. Like in the previous case, we have by Definition~\ref{d:p_intersection_fourier} that $R_p K$ is a $p$-intersection body (of a dilate of $F_{n-p} K$). Since $R_{n-p} K$ is a convex body, $\rho_{R_{n-p} K}^{n-p}$ is a positive, tempered distribution. Thus, we have shown that $\rho_{F_p K}^p$ is a positive-definite distribution. The formula \eqref{eq:R_K_I_p_false_start} and \eqref{eq:F_K_I_p} follow from Definition \ref{d:p_intersection_fourier}.  We complete the proof.
\end{proof}

\subsection{The Range \texorpdfstring{$-1<p<0$}{-1<p<0}}
\label{sec:fourier_neg}

We now list our results for $p\in (-1,0)$. To this end, we recall the following; see \cite[Lemma 6.4 and Theorem 6.10]{AK05}.
\begin{defn}
\label{d:embed}
    Let $n\geq 1$ and let $M\subset \R^n$ be an origin-symmetric star body and let $q\in (0,2)$. We say $M$ embeds in $L^q$ if it satisfies the following equivalent conditions:
    \begin{enumerate}
        \item One has that $\Gamma\left(-\frac{q}{2}\right)\widehat{\rho_M^{-q}}$ is a positive distribution on $\R^n\setminus\{0\}$;
        \item There exists a finite Borel measure $\mu_q$ on $\s^{n-1}$ such that
        \[
        \rho_M=\left(\int_{\s^{n-1}}|\langle \cdot,u \rangle|^qd\mu_q(u)\right)^{-\frac{1}{q}}.
        \]
    \end{enumerate}
\end{defn}
Notice, for example, that every polar projection body $\Pi^\circ K$ of a convex body $K\subset \R^n$ embeds in $L^1$ by the Cauchy formula for $\vol_{n-1}(P_{\theta^\perp} K)$ and item (2). We will first show that $R_p K$ embeds in $L^{-p}$ for $p\in (-1,0)$, establishing a connection to $F_{n-p} K$ along the way, and then prove Corollary~\ref{c:radial_action}.

An immediate observation is that, in the planar case, the convexity of $R_p K$ already yields what we want.
\begin{prop}
\label{p:embeds}
Let $K\subset \R^2$ be a planar convex body. Then, for $p\in (-1,0)$, $R_p K$ embeds in $L^{-p}$.
\end{prop}
\begin{proof}
    It is well-known that every origin-symmetric planar convex body embeds in $L^1$ (see, e.g., \cite[Corollary 6.8]{AK05}). J. Haddad \cite{JH26} recently showed that $R_p K$ is an origin-symmetric convex body for all $p>-1$ when $n=2$. Therefore, $R_p K$ embeds in $L^1$. By \cite[Corollary 6.7]{AK05}, $R_p K$ embeds in $L^q$ for all $0<q<1$. In particular, for $q=-p$.
\end{proof}
Note that, for $n\geq 3$, not every origin-symmetric convex body embeds in $L^1$; thus, the convexity of $R_p K$ for $p\in (-1,0)$ established in \cite{DL26} is not sufficient for resolving the question in general. Consequently, we need a different approach. For $q>0$, we use the same integral definition of $\Gamma_q^\circ M$ for measurable sets $M$ satisfying $0<\vol_n(M)<\infty$ and $\int_M|y|^q\,dy<\infty$.

\begin{prop}
\label{p:new_radial_fourier_identity}
Let $n\geq1$ and let $K\subset\R^n$ be a convex body. Then, for $-1<p<0$,
\[
R_pK
=
\left(
\frac{(2\pi)^n}
{\Gamma(1+p)\cos\left(\frac{\pi p}{2}\right)
\vol_n(F_{n-p}K)}
\right)^{-\frac1p}
\Gamma_{-p}^\circ(F_{n-p}K).
\]
\end{prop}
\begin{proof}
First, Proposition~\ref{p:mono} shows that $F_{n-p}K$ has finite volume and, by polar coordinate integration,
\[
\int_{F_{n-p}K}|y|^{-p}\,dy
=
\frac1{n-p}
\int_{\s^{n-1}}
\rho_{F_{n-p}K}(\theta)^{n-p}\,d\theta
<\infty.
\]
Thus, the integral definition of
$\Gamma_{-p}^\circ(F_{n-p}K)$ applies.

By definition, we have
\[
\rho_{R_pK}(x)^p
=
\frac{-p}{\vol_n(K)}
\int_0^\infty
\bigl(\vol_n(K)-g_K(tx)\bigr)t^{p-1}\,dt.
\]
By \eqref{eq:covariogram_fourier_inversion}, we have
\[
\vol_n(K)-g_K(tx)
=
\frac1{(2\pi)^n}
\int_{\R^n}
\bigl(1-\cos(t\langle x,\xi\rangle)\bigr)
|\widehat{\chi_K}(\xi)|^2\,d\xi.
\]
The integrand is nonnegative, so Tonelli's theorem applies. Therefore, using
\[
(-p)\int_0^\infty(1-\cos t)t^{p-1}\,dt
=
\Gamma(1+p)\cos\left(\frac{\pi p}{2}\right),
\]
we obtain
\begin{equation}
\label{eq:radial_polar_identity_wow}
\rho_{R_pK}(x)^p
=
\frac{\Gamma(1+p)\cos\left(\frac{\pi p}{2}\right)}
{\vol_n(K)(2\pi)^n}
\int_{\R^n}
|\langle x,\xi\rangle|^{-p}
|\widehat{\chi_K}(\xi)|^2\,d\xi.
\end{equation}

On the other hand, by polar coordinate integration and
the definition of $F_{n-p}K$,
\begin{align*}
\int_{F_{n-p}K}|\langle x,y\rangle|^{-p}\,dy
&=
\frac1{n-p}
\int_{\s^{n-1}}
|\langle x,\theta\rangle|^{-p}
\rho_{F_{n-p}K}(\theta)^{n-p}\,d\theta
\\
&=
\frac1{\vol_n(K)}
\int_{\R^n}
|\langle x,\xi\rangle|^{-p}
|\widehat{\chi_K}(\xi)|^2\,d\xi.
\end{align*}
Therefore,
\begin{equation}
\label{eq:radial_fourier_identity_wow}
\rho_{R_pK}(x)^p
=
\frac{\Gamma(1+p)\cos\left(\frac{\pi p}{2}\right)}
{(2\pi)^n}
\int_{F_{n-p}K}|\langle x,y\rangle|^{-p}\,dy.
\end{equation}
Inserting the radial function of
$\Gamma_{-p}^\circ(F_{n-p}K)$ allows us to conclude.
\end{proof}

Our next result is the Fourier-analytic connection between $R_p K$ and $F_{n-p} K$ when $p\in (-1,0)$. Unlike in Theorem~\ref{t:fourier_small_p}, the identity below is understood in the sense of distributions on $\R^n\setminus\{o\}$. 
\begin{thm}
\label{t:radial_fourier_neg_p}
Let $n\geq1$ and let $K\subset\R^n$ be a convex body. For $-1<p<0$, $\Gamma(p/2)\widehat{\rho_{R_pK}^p}$ is a positive distribution on $\R^n\setminus\{0\}$. In particular, $R_pK$ embeds in $L^{-p}$. Moreover,
\begin{equation}
\label{eq:radial_fourier_neg_p}
\widehat{\rho_{R_pK}^p}
=
\frac{p}{n-p}\rho_{F_{n-p}K}^{n-p}
\end{equation}
in the sense of distributions on $\R^n\setminus\{o\}$.
\end{thm}
\begin{proof}
Define the Borel measure on $\s^{n-1}$ given by
\[
d\mu_p(\theta)
=
\frac{\Gamma(1+p)\cos\left(\frac{\pi p}{2}\right)}
{\vol_n(K)(2\pi)^n}
\left(
\int_0^\infty
r^{n-p-1}|\widehat{\chi_K}(r\theta)|^2\,dr
\right)d\theta.
\]
It follows from \eqref{eq:radial_polar_identity_wow} and polar coordinate integration that
\[
\rho_{R_pK}(x)^p
=
\int_{\s^{n-1}}
|\langle x,\theta\rangle|^{-p}\,d\mu_p(\theta).
\]
Moreover, Proposition~\ref{p:mono} gives
\[
\mu_p(\s^{n-1})
=
\frac{\Gamma(1+p)\cos\left(\frac{\pi p}{2}\right)}
{(n-p)(2\pi)^n}
\int_{\s^{n-1}}
\rho_{F_{n-p}K}(\theta)^{n-p}\,d\theta
<\infty.
\]
Therefore, by Definition~\ref{d:embed}, $R_pK$ embeds in $L^{-p}$.

To establish the Fourier identity, observe that $\rho_{F_{n-p}K}^{n-p}$ is an even function belonging to $L^1(\s^{n-1})$. By \eqref{eq:radial_fourier_identity_wow} and polar coordinate integration,
\[
\rho_{R_pK}^p
=
\frac{\Gamma(1+p)\cos\left(\frac{\pi p}{2}\right)}
{(n-p)(2\pi)^n}\mathcal{C}_{-p}(\rho_{F_{n-p}K}^{n-p}).
\]

Applying \eqref{eq:fourier_cosine}, which we may do by Remark~\ref{eq:koldob_extend}, and the Gamma reflection identity \eqref{eq:gamma_sine_reflection}, we obtain, away from the origin,
\begin{align*}
\widehat{\rho_{R_pK}^p}
&=
\frac{\Gamma(1+p)\Gamma(1-p)\sin(\pi p)}
{\pi(n-p)}
\rho_{F_{n-p}K}^{n-p}=
\frac{p}{n-p}\rho_{F_{n-p}K}^{n-p}.
\end{align*}
This proves \eqref{eq:radial_fourier_neg_p}. Finally, since $\Gamma(p/2)<0$ and $p/(n-p)<0$, $\Gamma(p/2)\widehat{\rho_{R_pK}^p}$ is a positive distribution on $\R^n\setminus\{0\}$.
\end{proof}

\begin{proof}[Proof of Corollary~\ref{c:radial_action}]
By Proposition~\ref{p:mono}, $\rho_{F_{n-p}K}^{n-p}\in L^1(\s^{n-1})$. Since $\rho_{F_{n-p}K}^{n-p}$ is homogeneous of degree $p-n$, the formula
\[
\langle T,\varphi\rangle
:=
\frac{p}{n-p}
\int_{\R^n}
\rho_{F_{n-p}K}(\xi)^{n-p}\bigl(\varphi(\xi)-\varphi(0)\bigr)\,d\xi
\]
defines a tempered distribution. Indeed, near the origin, $\varphi(\xi)-\varphi(0)=O(|\xi|)$, while this difference is bounded at infinity; the corresponding radial integrals are bounded by
\[
\int_0^1 r^p\,dr
\qquad\text{and}\qquad
\int_1^\infty r^{p-1}\,dr,
\]
respectively, and both are finite because $-1<p<0$.

A change of variables shows that $T$ is homogeneous of degree $p-n$. The same is true of $\widehat{\rho_{R_pK}^p}$, since $\rho_{R_pK}^p$ is homogeneous of degree $-p$. By Theorem~\ref{t:radial_fourier_neg_p}, these distributions agree away from the origin. Their difference is therefore supported at the origin.

Every distribution supported at the origin is a finite linear combination of derivatives of the Dirac distribution, whose homogeneity degrees are $-n-k$, with $k$ a nonnegative integer (see, e.g., \cite[Theorem 2.3.4]{HL}). None of these equals $p-n$ when $-1<p<0$. Thus, the difference vanishes, proving the asserted formula.
\end{proof}

\section{The Geometry of the Fourier Mean Bodies}
\label{sec:fourier_bodies}
The Fourier mean bodies themselves possess intricate geometric properties, some of which we ascertain here.
 \subsection{Isotropic Estimates for Fourier Mean Bodies}
\label{sec:isotropic}
This section is dedicated to proving Theorem~\ref{t:use_slicing}. Along the way, we establish a self-contained geometric proof of the fact that  $F_p K$ is compact when $p\in (0,1]$, without having to invoke Theorem~\ref{t:compactness}. We will need the following formula, which is an exercise obtained from \eqref{eq:stein}. Let $0<p<n$ and consider the locally integrable function on $\R^n$  $|\cdot|^{p-n}$. Then, its Fourier transform is the tempered distribution
\begin{equation}
    \label{eq:moments_fourier}
    \widehat{|\cdot|^{p-n}}=2^{p}\pi^\frac{n}{2}\frac{\Gamma\left(\frac{p}{2}\right)}{\Gamma\left(\frac{n-p}{2}\right)}|\cdot|^{-p}.
\end{equation}

We repeatedly use the case where $n=1$ and $p\in (0,1)$. The coefficients have an elegant representation in this case, with an application of the Gamma reflection formula (in the form of \eqref{eq:gamma_cosine_reflection}) and the Legendre duplication formula (in the form of \eqref{eq:duplication}):
\begin{equation}
    \label{eq:moments_fourier_1D}
    \widehat{|\cdot|^{p-1}}=\frac{2}{p}\cos\left(\frac{p\pi}{2}\right)\Gamma\left(p+1\right)|\cdot|^{-p}, \qquad p\in (0,1).
\end{equation}
We start our investigations with the following lemma.
\begin{lem}
\label{l:mollifer}
    Let $n\geq 1$ and let $K\subset \R^n$ be a convex body. Then, there exists an explicit positive constant $\kappa(n,p)$ such that, for $0<p<n$,
\[
\left\langle \widehat{g_K}, \frac{p}{n}|\cdot|^{p-n} \right\rangle= \kappa(n,p)\left\langle |\cdot|^{-p},g_K\right\rangle.
\]
Similarly, for $0<p<1$ and every $\theta\in \s^{n-1}$, it holds
\[
\left\langle\frac{\widehat{\mathcal{R}g_K(\theta;\cdot)}}{\vol_n(K)},\frac{p}{2}|\cdot|^{p-1}\right\rangle = \frac{\Gamma\left(p+1\right)\cos\left(\frac{\pi p}{2}\right)}{\vol_n(K)}\left\langle \mathcal{R}g_K(\theta;\cdot),|\cdot|^{-p} \right\rangle.
\]

\end{lem}
\begin{proof}
    We first notice that, in each instance, the left-hand side is a genuine inner-product, while the right-hand side is the action of the tempered distribution $|\cdot|^{-p}$. Because of this, we cannot simply appeal to Parseval's formula \eqref{eq:parseval_1} to move the Fourier transform onto the kernels $|\cdot|^{p-n}$ and $|\cdot|^{p-1}$. To overcome this obstacle, we use that Schwartz functions are dense in $L^1(\R^n)$. 
    
    Let $f\in \mathcal{S}(\R^n)$ be a non-negative and supported on $\B$ with $\|f\|_{L^1(\R^n)}=1$ (e.g., take $f$ proportional to $\exp\left(-\frac{1}{1-|\;\cdot\;|^2}\right)\chi_{\B}$). Define a sequence of \textit{Schwartz mollifiers} by $f_j(x)=j^nf(jx)$. Then, for all $j\geq1,$ $f_j\in \mathcal{S}(\R^n)$, $\supp(f_j)\subseteq\frac{1}{j}B_2^n$, and $\|f_j\|_{L^1(\R^n)}= 1$. 
    
    We now specialize our proof to the first identity. Define $\varphi_j=g_K\ast f_j$. Then, $\varphi_j \to g_K$ almost everywhere. Furthermore, $\hat{\varphi_j}=\hat{g_K}\hat{f_j}$ converges to $\hat{g_K}$ pointwise: for every $x\in \R^n,$
    \begin{equation}
    \label{eq:everywhere}
        |\widehat{\varphi_j}(x)-\widehat{g_K}(x)|=\left|\int_{\R^n}e^{-i\langle x,y\rangle}(\varphi_j(y)-g_K(y))dy\right| \leq \|\varphi_j-g_K\|_{L^1(\R^n)}.
    \end{equation}
    We use \eqref{eq:moments_fourier} to infer the existence of a constant $\kappa(n,p)$ such that $\frac{p}{n}\widehat{|\cdot|^{p-n}}=\kappa(n,p)|\cdot|^{-p}$ in the sense of distributions. We then have the formal computation, by \eqref{eq:parseval_1},
    \begin{equation}
    \label{eq:quer_fourier_final}
    \begin{split}
         \frac{p}{n} \left\langle |\cdot|^{p-n},\widehat{g_K} \right\rangle &=\frac{p}{n}\left\langle|\cdot|^{p-n},\lim_{j\to \infty}\widehat{\varphi_j}\right\rangle 
        =\frac{p}{n} \lim_{j\to\infty}\left\langle |\cdot|^{p-n},\widehat{\varphi_j} \right\rangle 
        \\
       & = \kappa(n,p) \lim_{j\to \infty} \left\langle |\cdot|^{-p},\varphi_j\right\rangle = \kappa(n,p)\left\langle |\cdot|^{-p},g_K\right\rangle.
        \end{split}
    \end{equation}
    
    We now justify the two instances of exchanging limits and integrals in \eqref{eq:quer_fourier_final}. For the first exchange, observe that $|\hat{f_j}|\leq \hat{f_j}(o)=\|f_j\|_{L^1(\R^n)}=1$ and $\hat f_j\to 1$. Consequently,
    \begin{equation}
    \label{eq:fourier_uniform}
    |x|^{p-n}|\widehat{\varphi_j}(x)|\leq |x|^{p-n}\widehat{g_K}(x), \qquad x\in \R^n.
    \end{equation}

    We use \eqref{eq:fourier_uniform} and \eqref{eq:fourier_body} to deduce, by polar coordinate integration,
    \begin{align*}
        |\left\langle |\cdot|^{p-n},\widehat{\varphi_j} \right\rangle| &\leq\int_{\R^n}|x|^{p-n}|\widehat{\varphi_j}(x)| dx 
        \leq \int_{\R^n}|x|^{p-n}|\widehat{g_K}(x)| dx 
        \\
        &=\frac{\vol_n(K)}{p}\int_{\s^{n-1}}\rho_{F_pK}(\theta)^pd\theta <+\infty
    \end{align*}
    by Proposition~\ref{p:mono}, since $F_pK$ is an $L^p$ star. Consequently, the first exchange is justified by the dominated convergence theorem. For the second exchange, we use that, since $g_K$ has compact support, and is in $L^1(\R^n)\cap L^\infty (\R^n)$, then $\varphi_j \in L^1(\R^n)\cap L^\infty (\R^n)$:
    \[
    |\varphi_j(x)| = \int_{\R^n}g_K(y)f_j(x-y)dy \leq \|g_K\|_{L^\infty(\R^n)}=\vol_n(K).
    \]
    Furthermore, $\supp(\varphi_j)\subseteq DK+\frac{1}{j}\B$, thus, $\supp(\varphi_j) \subset DK +\B$ for all $j.$ Consequently, the second exchange again uses the dominated convergence theorem (by $\vol_n(K)|x|^{-p}\chi_{DK+\B}$).

    For the second identity, we first deduce from \eqref{eq:parallel_fubini} the identity
\begin{equation}
\label{eq:radon_g_K_convol}
\begin{split}
\widehat{\mathcal{R}g_K(\theta;\cdot)}(r)  &= \widehat{\chi_K}(r\theta)\widehat{\chi_K}(-r\theta)
= \widehat{A_{K,\theta}}(r)\widehat{A_{K,\theta}}(-r) 
\\
&= \left(\widehat{A_{K,\theta}(t)\ast A_{K,\theta}(-t)}\right)(r).
\end{split}
\end{equation}
In particular, \eqref{eq:radon_g_K_convol} shows that $\widehat{\mathcal{R}g_K(\theta;\cdot)}(r)$ decays at least as fast as $r^{-2}$. Recalling also that $\widehat{\mathcal{R}g_K(\theta;\cdot)}(r)=\hat{g_K}(r\theta)$ by \eqref{eq:og_radon_fourier}, the proof of the claim follows line-by-line the same as that of the first identity with $n=1$.
\end{proof}

Next, we derive a new formula for the radial function of $F_p K$ in terms of the parallel section function $A_{K,\theta}$ from \eqref{eq:parallel_section}. Recall that we allow radial functions to take on the value $+\infty$.

\begin{prop}
\label{p:f_formulas}
    Let $n\geq 1$, $p>0$, and let $K\subset \R^n$ be a convex body and fix $\theta\in\s^{n-1}$. Then, 
    \begin{equation}
        \rho_{F_p K}(\theta) =\frac{1}{\sqrt{2}}\left(\frac{p}{\vol_n(K)}\int_{0}^{+\infty} \widehat{A_{K\times (-K),\left(\frac{\theta}{\sqrt{2}},\frac{\theta}{\sqrt{2}}\right)}}(r)r^{p-1}dr\right)^\frac{1}{p},
        \label{eq:fourier_body_v2}
    \end{equation}
and
\begin{equation}
         \rho_{F_p K}(\theta)= \left(\frac{p}{\vol_n(K)}\int_{0}^{+\infty} \widehat{\left(A_{K,\theta}(t)\ast A_{K,\theta}(-t)\right)}(r)r^{p-1}dr\right)^\frac{1}{p}.
        \label{eq:using_og_radon_2}
    \end{equation}
Next, if $p\in (0,1)$, 
\begin{equation}
         \rho_{F_p K}(\theta)\!=\! \left(\frac{2\Gamma\left(p+1\right)\cos\left(\frac{\pi p}{2}\right)}{\vol_n(K)}\!\int_{0}^{+\infty} \!\!\!\!\!\!\left(A_{K,\theta}(t)\ast A_{K,\theta}(-t)\right)(r)r^{-p}dr\right)^\frac{1}{p}.
        \label{eq:using_og_radon_3}
    \end{equation}
    
 \noindent   In fact, 
\begin{equation}
\label{eq:parellel_convol}
A_{K\times (-K),\left(\frac{\theta}{\sqrt{2}},\frac{\theta}{\sqrt{2}}\right)}\left(\frac{r}{\sqrt{2}}\right)  =\sqrt{2}\left(A_{K,\theta}(t)\ast A_{K,\theta}(-t)\right)(r), \quad r\in \R.
\end{equation}
Therefore, for $p\in (0,1)$,
\begin{equation}
         \rho_{F_p K}(\theta)= \!\frac{1}{\sqrt{2}}\left(\frac{2\Gamma\left(p+1\right)\cos\left(\frac{\pi p}{2}\right)}{\vol_n(K)}\int_{0}^{+\infty}\!\!\!\!\!\!\! A_{K\times (-K),\left(\frac{\theta}{\sqrt{2}},\frac{\theta}{\sqrt{2}}\right)}\!\!\left(r\right)r^{-p}dr\!\!\right)^\frac{1}{p}.
        \label{eq:parellel_convol_2}
    \end{equation}
\end{prop}

\begin{proof}
 Consider two convex bodies $K,D\subset\R^n$. Let $\xi_1,\xi_2 \in \s^{n-1}$ and define $\xi=\left(\frac{\xi_1}{\sqrt{2}},\frac{\xi_2}{\sqrt{2}}\right) \in \s^{2n-1}$. Then, from \eqref{eq:parallel_fubini}:
\begin{align*}
    \widehat{A_{K\times D,\xi}}(\sqrt{2}r) &= \int_{\R^n\times \R^n}e^{-i\left\langle  r(\xi_1,\xi_2),(x_1,x_2)\right\rangle}\chi_{K\times D}((x_1,x_2)) dx_1dx_2 
    \\
    &=\left(\int_{\R^n}e^{-i\left\langle  r\xi_1,x\right\rangle}\chi_{K}(x) dx\right)\left(\int_{\R^n}e^{-i\left\langle  r\xi_2,x\right\rangle}\chi_{D}(x) dx\right)
    \\
    &=\widehat{\chi_K}(r\xi_1)\widehat{\chi_D}(r\xi_2).
\end{align*}
In particular, by setting $D=-K$ and $\xi_1=\theta=\xi_2$, we obtain
\begin{equation}
    \widehat{A_{K\times (-K),\left(\frac{\theta}{\sqrt{2}},\frac{\theta}{\sqrt{2}}\right)}}(\sqrt{2}r) = \widehat{\chi_K}(r\theta)\widehat{\chi_{-K}}(r\theta)=|\widehat{\chi_K}(r\theta)|^2.
    \label{eq:parellel_double}
\end{equation}
The formula \eqref{eq:fourier_body_v2} follows. As for \eqref{eq:using_og_radon_2}, we start by writing the radial function of $F_p K$ as,
    \begin{equation}
        \rho_{F_p K}(\theta)=\left(\frac{p}{\vol_n(K)}\int_{0}^{+\infty}\widehat{g_K}(r\theta)r^{p-1}dr\right)^\frac{1}{p}, \quad \theta\in\s^{n-1}.
        \label{eq:step_to_distribution}
    \end{equation}
  Inserting \eqref{eq:radon_covariogram} into \eqref{eq:step_to_distribution} yields the identity
    \begin{equation}
         \rho_{F_p K}(\theta)=\left(\frac{p}{\vol_n(K)}\int_{0}^{+\infty}\widehat{\mathcal{R}g_K(\theta;\cdot)}(r)r^{p-1}dr\right)^\frac{1}{p}.
        \label{eq:using_og_radon}
    \end{equation}
The formulas \eqref{eq:using_og_radon} and  \eqref{eq:radon_g_K_convol} yield the claim. We next write \eqref{eq:using_og_radon} as the action of a tempered distribution:
\begin{equation}
    \rho_{F_p K}(\theta) =\left(\left\langle\frac{\widehat{\mathcal{R}g_K(\theta;\cdot)}}{\vol_n(K)},\frac{p}{2}|\cdot|^{p-1}\right\rangle\right)^\frac{1}{p}.
    \label{eq:using_og_radon_again}
\end{equation}
An application of Lemma~\ref{l:mollifer} then yields
 \begin{equation}
 \begin{split}
        \rho_{F_p K}(\theta)=\left(\frac{\Gamma\left(p+1\right)\cos\left(\frac{\pi p}{2}\right)}{\vol_n(K)}\left\langle \mathcal{R}g_K(\theta;\cdot),|\cdot|^{-p} \right\rangle\right)^\frac{1}{p}.
    \end{split}
        \label{eq:step_to_distribution_2}
    \end{equation}
     By applying Fourier inversion to \eqref{eq:radon_g_K_convol}, we deduce
\begin{equation}
\label{eq:radon_g_K_convol_2}
\mathcal{R}g_K(\theta;r)   = \left(A_{K,\theta}(t)\ast A_{K,\theta}(-t)\right)(r).
\end{equation}
We then obtain \eqref{eq:using_og_radon_3} from \eqref{eq:radon_g_K_convol_2} and \eqref{eq:step_to_distribution_2}. To deduce \eqref{eq:parellel_convol}, we equate \eqref{eq:parellel_double} and \eqref{eq:radon_g_K_convol} to obtain
\begin{align*}
\left(\widehat{A_{K,\theta}(t)\ast A_{K,\theta}(-t)}\right)(r) &=  \widehat{A_{K\times (-K),\left(\frac{\theta}{\sqrt{2}},\frac{\theta}{\sqrt{2}}\right)}}(\sqrt{2}r) 
\\
&= \frac{1}{\sqrt{2}}\widehat{A_{K\times (-K),\left(\frac{\theta}{\sqrt{2}},\frac{\theta}{\sqrt{2}}\right)}\left(\frac{t}{\sqrt{2}}\right)}(r)
\end{align*}
from \eqref{eq:Fourier_factoring}. We conclude by applying Fourier inversion.
\end{proof}

The following corollary of Proposition~\ref{p:f_formulas} establishes that $F_1 K$ is compact. In particular, it already yields that, if $K$ is origin-symmetric, then $F_1 K$ is an origin-symmetric convex body. Indeed, if $K$ is origin-symmetric, then $K \times K$ is an origin-symmetric convex body in $\R^{2n}$. By Busemann's theorem, its intersection body is convex and therefore, its sections are convex.
\begin{cor}
\label{cor:F_1K}
    If $n\geq 1$ and $K\subset\R^n$ is a convex body, then $F_1 K$ is an $n$-dimensional  section of $\frac{1}{\sqrt{2}}\frac{\pi}{\vol_n(K)} I \left((K-z) \times (z-K)\right),$ $\,\forall \,z\in \operatorname{int}(K)$.
\end{cor}
\begin{proof}
Let $\theta\in \s^{n-1}$ and $\xi_\theta:=\left(\frac{\theta}{\sqrt{2}},\frac{\theta}{\sqrt{2}}\right)$. From \eqref{eq:fourier_body_v2} with $p=1$ and \eqref{eq:inversion},
    \begin{align*}
     \rho_{F_1 K}(\theta) &= \frac{1}{\sqrt{2}}\frac{1}{\vol_n(K)}\int_{0}^{+\infty} \widehat{A_{K\times (-K),\xi_\theta}}(r)dr
     \\
     &= \frac{1}{2\sqrt{2}}\frac{1}{\vol_n(K)}\int_{-\infty}^{+\infty} \widehat{A_{K\times (-K),\xi_\theta}}(r)dr
    \\
    &=\frac{1}{\sqrt{2}}\frac{\pi}{\vol_n(K)}A_{K\times (-K),\xi_\theta}(0)
    \\
    &=\frac{1}{\sqrt{2}}\frac{\pi}{\vol_n(K)}\vol_{2n-1}\left(\left(K\times (-K)\right) \cap \xi_\theta^\perp\right).
\end{align*}
Suppose firstly that $K$ contains the origin in its interior, in which case, so too does $K\times (-K)$. Then, $$\vol_{2n-1}\left(\left(K\times (-K)\right) \cap \xi_\theta^\perp\right)=\rho_{I(K\times (-K))} \left(\xi_\theta\right).$$ This settles the claim in this case. For general $K$, we use the translation invariance of the $(2n-1)$-dimensional Lebesgue measure on $\xi_\theta^\perp$. Indeed, if $z\in \R^n$, then $(-z,z)\in \xi_\theta^\perp$ and
\[
\bigl((K-z)\times(z-K)\bigr)\cap\xi_\theta^\perp
=
\bigl((K\times(-K))\cap\xi_\theta^\perp\bigr)+(-z,z).
\]
Taking volumes produces, for $z\in \operatorname{int}(K)$,
\begin{align*}
    \rho_{F_1 K}(\theta) &=\frac{1}{\sqrt{2}}\frac{\pi}{\vol_n(K)}\vol_{2n-1}\left(\left(K\times (-K)\right) \cap \xi_\theta^\perp\right)
    \\
    &=\frac{1}{\sqrt{2}}\frac{\pi}{\vol_n(K)}\vol_{2n-1}\left(\left((K-z)\times (z-K)\right) \cap \xi_\theta^\perp\right)
    \\
    &=\frac{1}{\sqrt{2}}\frac{\pi}{\vol_n(K)}\rho_{I((K-z)\times (z-K))} \left(\xi_\theta\right).
\end{align*}
We conclude.
\end{proof}

Using the techniques of this section, we prove Theorem~\ref{t:use_slicing}.  First, we need the following lemma, which may be of independent interest. It is inspired by the approach of M. Anttila, K. Ball, and I. Perissinaki \cite[Lemma 3]{ABP03}.

\begin{lem}
\label{l:our_slicing_bound}
    Let $n\geq 1$ and let $K\subset\R^n$ be an isotropic convex body. Define 
    \[
K^\prime := K\times (-K) \times [- \sqrt{3}L_K, \sqrt{3}L_K]\subset \R^{2n+1}.
\]
Then,
\[ 1\le\vol_{2n}(K^\prime \cap \nu^\perp) \le \sqrt{6}, \qquad \forall\; \nu\in \s^{2n}.
\]

Equality is achieved in the first inequality if and only if either $\nu=\pm e_{2n+1}$ and $K$ is any isotropic body or if $\nu\in \{(\theta,0,0),(0,\theta,0):\theta\in \s^{n-1}\}$ and $K$ is an isotropic cylinder in the corresponding direction $\theta$.
\end{lem}
\begin{proof}
We note that, since $K$ is in isotropic position, then so too is $K\times (-K)$ and $L_{K\times (-K)}=L_{K}$. Consider the interval $[-a,a]$ for $a>0$. Then,
\begin{equation}\label{isotrseg}
I_a:=\int_{-a}^ax^2 dx = \frac{2}{3}a^3.
\end{equation}
As advertised, we pick $a = \sqrt{3}L_K$, so that $2aL_K^2 = I_a = 2\sqrt{3}L_K^3$. From the fact that $K$ is in isotropic position and \eqref{isotrseg}, we obtain
    \begin{equation}
    \int_{K^\prime}\langle z,\nu \rangle^2 dz = 2\sqrt{3}L_K^3, \quad \nu\in \s^{2n},
    \label{eq:using_iso}
    \end{equation}
   and $\vol_{2n+1}(K^\prime)=2a=2\sqrt{3}L_K$. We set $K^{\prime\prime}=\left(2\sqrt{3}L_K\right)^{-\frac{1}{2n+1}}K^\prime$, so that $\vol_{2n+1}(K^{\prime\prime})=1$ and thus $K^{\prime\prime}$ is in isotropic position. Then, \eqref{eq:using_iso} becomes, for $\nu\in \s^{2n}$,
    \begin{equation}
    \begin{split}
    \int_{K^{\prime\prime}}\langle z,\nu \rangle^2 dz &= \int_{K^\prime}\left(2\sqrt{3}L_K\right)^{-\frac{2n+3}{2n+1}}\langle z,\nu \rangle^2 dz 
    \\
    &= 
    \left(2\sqrt{3}L_K\right)^{-\frac{2n+3}{2n+1}}\left(2\sqrt{3}L_K^3\right).
    \end{split}
    \label{eq:using_iso_2}
    \end{equation}
    Therefore, 
    \begin{equation}
    \label{eq:iso_relate}
    L_{K^{\prime\prime}}=\left(2\sqrt{3}\right)^{-\frac{1}{2n+1}}L_K^\frac{2n}{2n+1}.
    \end{equation} 
    
 We apply Theorem~\ref{t:fradelizi} to $K^{\prime\prime}$; replacing $L_{K^{\prime\prime}}$ with its formula \eqref{eq:iso_relate} in terms of $L_K$, we obtain, recalling the constants $c(2n+1,2)$ and $c(2)=1/\left(2\sqrt{3}\right)$ from that theorem, for $\nu\in\s^{2n}$,
    $$
       (2\sqrt{3}L_K)^{-\frac{2n}{2n+1}} \leq \vol_{2n}(K^{\prime\prime}\cap \nu^\perp) \le c(2n+1,2) \cdot (2\sqrt{3})^{\frac{1}{2n+1}}L_K^{-\frac{2n}{2n+1}}.
    $$ 
    Since $K^{\prime\prime}=\left(2\sqrt{3}L_K\right)^{-\frac{1}{2n+1}}K^\prime$, it follows from the homogeneity of the Lebesgue measure that $\vol_{2n}(K^\prime \cap \nu^\perp) = (2\sqrt{3}L_K)^{\frac{2n}{2n+1}} \vol_{2n}(K^{\prime\prime}\cap \nu^\perp)$. The factors of $L_K$ cancel, and we deduce
    \[ 1\le\vol_{2n}(K^\prime \cap \nu^\perp) \le \sqrt{3}(2n+1)\binom{2n+3}{2}^{-\frac{1}{2}}, \qquad \forall\; \nu\in \s^{2n}.
\]
    Notice the function $$n\longmapsto(2n+1)\binom{2n+3}{2}^{-\frac{1}{2}} = \frac{2n+1}{\sqrt{2n^2 + 5n + 3}}$$
    is monotonically increasing to $\sqrt{2}$. Combining this bound with the inequality we have shown allows us to conclude.

    There is equality in the lower bound if and only if $K^{\prime}$ is a cylinder in the direction $\nu$. Decompose $\nu \in \s^{2n}$ as $\nu = (x, y, t) \in \R^n \times \R^n \times \R$. Then, by its product structure, $K^\prime$ is a cylinder in the direction $\nu$ only if exactly one of $x, y,$ or $t$ is non-zero:
    \begin{enumerate}
        \item $\nu = \pm e_{2n+1}$. Any convex body in $\R^{2n+1}$ of the form $L\times [-a,a]$, where $L\subset \R^{2n}$ and $a>0$, is a cylinder in this direction. Thus, equality holds in this case for any isotropic body $K$.
        \item $\nu = (\theta, 0, 0)$ for some $\theta \in \s^{n-1}$. Equality holds if and only if $K$ is a cylinder in the direction $\theta$.
        \item $\nu = (0, \theta, 0)$ for some $\theta \in \s^{n-1}$. Equality holds if and only if $-K$ (and thus $K$) is a cylinder in the direction $\theta$.
    \end{enumerate}
    We conclude.
\end{proof}

\begin{proof}[Proof of Theorem~\ref{t:use_slicing}]
Henceforth, fix $\theta\in\s^{n-1}$. We first consider the case when $p\in (0,1)$ and the inequality \eqref{eq:depend_p}.  Set $\xi_\theta=\left(\frac{\theta}{\sqrt{2}},\frac{\theta}{\sqrt{2}}\right) \in \mathbb{S}^{2n-1}$.

We see that using \eqref{eq:parellel_convol_2} yields 
    \[
    \sqrt{2}\left(\Gamma\left(p+1\right)\cos\left(\frac{\pi p}{2}\right)\right)^{-\frac{1}{p}}\rho_{F_p K}(\theta)= \left(2\int_{0}^{+\infty} A_{K\times (-K),\xi_\theta}\left(t\right)t^{-p}dt\right)^\frac{1}{p}.
        \]
   We perform integration by parts and obtain
    \begin{equation}
\label{eq:ball_iso}
 2\int_{0}^{+\infty}A_{K\times (-K),\xi_\theta}\left(t\right)t^{-p}dt = p\int_0^{\infty}\int_{-t}^{t} A_{K\times (-K),\xi_\theta}\left(s\right)ds\, t^{-1-p}dt.
\end{equation}
Here, we used the verifiable fact that the parallel section function of $K\times (-K)$ is still even in directions of the form $\xi_\theta$. Next, define 
$$
K_\theta(t) = \{x\in K\times (-K): |\langle x,\xi_\theta\rangle| \leq t\},
$$
so that
\[
\vol_{2n}(K_\theta(t))  =  \int_{-t}^{t} A_{K\times (-K),\left(\frac{\theta}{\sqrt{2}},\frac{\theta}{\sqrt{2}}\right)}\left(s\right)ds.
\]
We climb to $\R^{2n+1}$ by defining, for $t>0$, the map $u:\R_+\to \s^{2n}$ by $$u_\theta(t)\!=\!\left(\frac{a}{\sqrt{t^2+a^2}}\xi_\theta,\frac{t}{\sqrt{t^2+a^2}}\right).$$ Then,
\[
P_{e_{2n+1}^\perp} (K^\prime \cap u_\theta(t)^\perp) = K_\theta(t).
\]
Here, $e_{2n+1}^\perp$ is the embedding of $\R^{2n}$ into $\R^{2n+1}$, and $P_{e_{2n+1}^\perp}$ is the orthogonal projection operator onto this space. By Cauchy's area formula, we know that
\begin{equation}
\label{eq:ball_cauchy}
\begin{split}
\vol_{2n}(K_\theta(t)) &= \vol_{2n}(P_{e_{2n+1}^\perp} (K^\prime \cap u_\theta(t)^\perp))
\\
&=\vol_{2n}(K^\prime \cap u_\theta(t)^\perp)\frac{t}{\sqrt{t^2+a^2}}.
\end{split}
\end{equation}
We insert \eqref{eq:ball_cauchy} into \eqref{eq:ball_iso} and deduce, for all $\theta\in\s^{n-1}$, 
\begin{equation}
    \frac{\sqrt{2}\cdot \rho_{F_p K}(\theta)}{\left(\Gamma\left(p+1\right)\cos\left(\frac{\pi p}{2}\right)\right)^{\frac{1}{p}}}= \left(p\int_{0}^{+\infty} \vol_{2n}(K^\prime \cap u_\theta(t)^\perp)\frac{t^{-p}}{\sqrt{t^2+a^2}} dt\right)^\frac{1}{p}.
        \label{eq:parellel_convol_slicing_2}
    \end{equation}
    Therefore, applying the result of Lemma~\ref{l:our_slicing_bound} in \eqref{eq:parellel_convol_slicing_2}, we have the pointwise bound
    \begin{equation}
    \label{eq:almost_ellip}
    \begin{split}
        1 &\le \left( \left(\Gamma\left(p+1\right)\cos\left(\frac{\pi p}{2}\right)\right)\frac{p}{2^\frac{p}{2}}\int_{0}^\infty\frac{t^{-p}}{\sqrt{t^2+a^2}}dt \right)^{-\frac{1}{p}} \!\!\!\!\!\!\!\!\rho_{F_p K}
        \le 6^\frac{1}{2p}.
    \end{split}
    \end{equation}
    We perform the variable substitution  $t\to ta$ and integrate:
    \[
    p\int_{0}^\infty\frac{t^{-p}}{\sqrt{t^2+a^2}}dt
    =\frac{p}{a^p}\int_{0}^\infty\frac{t^{-p}}{\sqrt{t^2+1}}dt
    =\frac{1}{a^p}\frac{\Gamma\left(\frac{1}{2}\right)}{\cos\left(\frac{\pi p}{2}\right)} \frac{\Gamma\left(1+\frac{p}{2}\right)}{\Gamma\left(\frac{1+p}{2}\right)}.
    \]
    Consequently, substituting $a = \sqrt{3}L_K$ and using \eqref{eq:duplication}, we obtain
    \begin{align*}
      \cos\left(\frac{\pi p}{2}\right)\frac{p}{2^\frac{p}{2}}\int_{0}^\infty\frac{t^{-p}}{\sqrt{t^2+a^2}}dt & = \left(\frac{1}{\sqrt{3}L_K}\right)^p \frac{\Gamma\left(\frac{1}{2}\right)\Gamma\left(1+\frac{p}{2}\right)}{2^\frac{p}{2}\Gamma\left(\frac{1+p}{2}\right)}\\ &=\left(\sqrt{\frac{2}{3}}\frac{1}{L_K}\right)^p \frac{\Gamma\left(1+\frac{p}{2}\right)^2}{\Gamma\left(p+1\right)}.
    \end{align*}
    We combine this calculation with \eqref{eq:almost_ellip} and deduce:
    \begin{equation}
    \sqrt{\frac{2}{3}}\cdot\Gamma\left(1+\frac{p}{2}\right)^{\frac{2}{p}}\cdot L_K^{-1}\leq\rho_{F_pK} \leq\sqrt{\frac{2}{3}}\cdot \left(\sqrt{6}\cdot\Gamma\left(1+\frac{p}{2}\right)^{2}\right)^{\frac{1}{p}}\cdot L_K^{-1}.
    \label{eq:depend_p_weak}
    \end{equation}
    The upper bound is precisely what we wanted in \eqref{eq:depend_p} for $p\in (0,1)$; the lower bound is not sharp enough. Indeed, by the convexity of $\log\Gamma$ and the fact that $\log\Gamma(1)=0$,
\[
 \log\Gamma\!\left(1+\frac p2\right)\leq p\log\Gamma\!\left(\frac32\right),
 \qquad 0<p\leq1.
\]
Consequently,
\[
\sqrt{\frac{2}{3}}\Gamma\!\left(1+\frac p2\right)^\frac{2}{p}
\leq \frac{\pi}{\sqrt{24}}<\sqrt{2}e^{-\frac{3}{4}}.
\]

We continue our investigation. Henceforth, introduce the notation
\[
H(s)=\left|\widehat{\chi_K}\left(\frac{s}{L_K}\theta\right)\right|^2, \qquad s>0.
\]
For all $p\in (0,2)$, changing variables in the definition of $F_pK$ gives
\[
\left(L_K\rho_{F_pK}(\theta)\right)^p=p\int_0^\infty H(s)s^{p-1}\,ds.
\]
We now obtain \eqref{eq:depend_p_weak} for $p=1$. By the section formula in Corollary~\ref{cor:F_1K},
\[
\frac{1}{L_K}\int_0^\infty H(s)\,ds
=\rho_{F_1 K}(\theta)=\frac{\pi }{\sqrt2}
\vol_{2n-1}((K\times (-K))\cap\xi_\theta^\perp).
\]
Applying the Fradelizi-Hensley estimates in \eqref{eq:fradelizi_2} and \eqref{eq:dimensional_free} yields
\[
\frac{\pi}{\sqrt{24}}\leq\int_0^\infty H(s)\,ds
=L_K\rho_{F_1K}(\theta)\leq\frac{\pi}{2}.
\]
This also produces the upper bound in \eqref{eq:independ_p} for $0<p\leq1$: use that Proposition~\ref{p:mono} gives $\rho_{F_p K}(\theta)\leq \rho_{F_1 K}(\theta)$.

Next, we obtain the lower bound in \eqref{eq:depend_p} and \eqref{eq:independ_p}.  By isotropy (i.e., \eqref{eq:isotropic_constant}) and the identity $\cos t\geq1-t^2/2$,
\begin{align*}
\operatorname{Re}
\widehat{\chi_K}\left(\frac{s}{L_K}\theta\right)
&=\int_K
\cos\left(\frac{s}{L_K}\langle x,\theta\rangle\right)\,dx\\
&\geq
1-\frac{s^2}{2L_K^2}
\int_K\langle x,\theta\rangle^2\,dx
=1-\frac{s^2}{2}.
\end{align*}
Consequently, for $0\leq s\leq\sqrt2$,
\[
H(s)\geq
\left(
\operatorname{Re}
\widehat{\chi_K}\left(\frac{s}{L_K}\theta\right)
\right)^2
\geq
\left(1-\frac{s^2}{2}\right)^2.
\]
Thus, for all $p\in (0,2)$,
\begin{align*}
\left(L_K\rho_{F_pK}(\theta)\right)^p
&\geq p\int_0^{\sqrt2}
\left(1-\frac{s^2}{2}\right)^2s^{p-1}\,ds
\\
&=
\frac{2^{p/2}}{(1+p/2)(1+p/4)}
\geq
\left(\sqrt2e^{-3/4}\right)^p,
\end{align*}
where the last inequality follows from $\log(1+p/2)+\log(1+p/4)\leq3p/4$.

All that remains to show is the upper bound in \eqref{eq:depend_p} when $1<p<2$. Fradelizi's maximal-section estimate \cite[Corollary 2]{FM99} for the parallel section function $A_{K,\theta}$ gives
\[
\|A_{K,\theta}\|_\infty = \max_{t\in \R}A_{K,\theta}(t)\leq L_K^{-1}.
\]
When $n=1$, this also follows directly since $K$ is a unit
interval. The superlevel sets of $A_{K,\theta}$ are intervals. Indeed, this is direct when $n=1$ and, for $n \geq 2$, is by the Brunn-Minkowski inequality, which produces that $A_{K,\theta}^\frac{1}{n-1}$ is concave; the superlevel sets of a concave function are convex sets. The identity \eqref{eq:parallel_fubini} and layer-cake integration therefore give, for $r>0$,
\begin{align*}
|\widehat{\chi_K}(r\theta)|
&=|\widehat A_{K,\theta}(r)|\leq\int_0^{\|A_{K,\theta}\|_\infty}
\left|
\int_{\{t\in \R: A_{K,\theta}(t)>u\}}e^{-irt}\,dt
\right|\,du
\\
&\leq \frac{2\|A_{K,\theta}\|_\infty}{r}
\leq \frac{2}{L_Kr}.
\end{align*}
Hence $H(s)\leq4s^{-2}$ for $s>0$.

For $1<p<2$, set $T:=8/\pi$, so that
\[
\int_T^\infty4s^{-2}\,ds=\frac{\pi}{2}.
\]
Since $s^{p-1}-T^{p-1}$ is nonpositive below $T$ and
nonnegative above $T$, we obtain
\begin{align*}
\left(L_K\rho_{F_pK}(\theta)\right)^p
&=pT^{p-1}\int_0^\infty H(s)\,ds
+p\int_0^\infty
\bigl(s^{p-1}-T^{p-1}\bigr)H(s)\,ds\\
&\leq pT^{p-1}\frac{\pi}{2}
+4p\int_T^\infty
\bigl(s^{p-1}-T^{p-1}\bigr)s^{-2}\,ds\\
&=4p\int_T^\infty s^{p-3}\,ds=\frac{4p}{2-p}T^{p-2}
=\frac{p\,4^{p-1}}{2-p}
\left(\frac{\pi}{2}\right)^{2-p}.
\end{align*}
Taking $p$th roots completes the proof.
\end{proof}

\subsection{On Convexity of Fourier Mean Bodies for \texorpdfstring{$p\in(0,1]$}{p in (0,1]}}
\label{sec:convex}
This section is dedicated to establishing the convexity of $F_p K$ when $p\in (0,1]$, that is, Theorem~\ref{t:main_convexity}. We first provide a direct proof for when $p\in (0,1)$ and $K$ is origin-symmetric.
\begin{proof}[Proof of Theorem~\ref{t:main_convexity} in the origin-symmetric case when $p\in (0,1)$]
We know from Corollary~\ref{cor:F_1K} and Proposition~\ref{p:mono} that $F_p K$ is bounded for $p\in (0,1)$. Consequently, it suffices to show that $\rho_{F_p K}^{-1}$ is a convex function for $p\in (0,1)$. We use \eqref{eq:parellel_convol_2} and Fubini's theorem: for $\theta\in \s^{n-1}$ and $\xi_\theta=\left(\frac{\theta}{\sqrt{2}},\frac{\theta}{\sqrt{2}}\right) \in \mathbb{S}^{2n-1}$,

    \begin{align*}
        \frac{2^\frac{p}{2} \vol_n(K)}{\Gamma\left(p+1\right)\cos\left(\frac{\pi p}{2}\right)} \rho_{F_p K}^p(\theta)&= 2\int_{0}^{+\infty}\!\! \!\!\int_{\left\{(x_1,x_2)\in K\times (-K):\left\langle (x_1, x_2),\xi_\theta\right\rangle =r\right\}}\!\!\!\!\!\!\!\!\!\!\!\!\!\!\!\!\!\!\!\!\!\!\!\!\!\!\!\!\!\!\!\!\!\!dx_1dx_2\;r^{-p}dr
        \\
        &=\int_{K\times (-K)} \left|\left\langle (x_1, x_2),\xi_\theta\right\rangle\right|^{-p}dx_1dx_2
        \\
        &=\vol^2_n(K)\rho_{\Gamma_{-p}^\circ \left(K\times (-K)\right)}^p \left(\xi_\theta\right).
    \end{align*}
    We now use the symmetry assumption on $K$. Then, $K\times (-K)$ is an origin-symmetric convex body in $\R^{2n}$; by Theorem~\ref{t:berck}, $\Gamma_{-p}^\circ \left(K\times (-K)\right)$ is also an origin-symmetric convex body. In particular, the function on $\R^n$ given by the $1$-homogeneous extension of $V(\theta)=\rho_{\Gamma_{-p}^\circ \left(K\times (-K)\right)}\left(\xi_\theta\right)^{-1}$ is convex. Since the above computations show that $V$ and $\rho_{F_p K}^{-1}$ coincide up to positive constants, we deduce that the latter is convex and conclude.
\end{proof}

We break the proof of Theorem~\ref{t:main_convexity} for general $K$ into two lemmas. The first lemma establishes the convexity of $F_p K$ when $p\in (0,1)$.
\begin{lem}
\label{l:F_p<1}
Let $n\geq 1$ and let $K\subset \R^n$ be a convex body and let $p\in (0,1)$. Then, 
\begin{equation}
       F_{p} K=\left(\Gamma(1+p)\cos\left(\frac{\pi p}{2}\right)\vol_n(K)\right)^{\frac{1}{p}}Z^\circ_{-p} K.
    \end{equation}
In particular, $F_p K$ is an origin-symmetric convex body.
\end{lem}

\begin{proof}
   We have by Theorem~\ref{t:fourier_small_p} that $\rho_{F_{p} K}^{p}=\frac{p}{n-p}\widehat{\rho_{R_{n-p} K}^{n-p}}$. But, we can compute $\widehat{\rho_{R_{n-p} K}^{n-p}}$ another way by using Proposition~\ref{p:Koldobsky_cosine}, with $q=-p$:
    \begin{align*}
    \widehat{\rho_{R_{n-p} K}^{n-p}} &= \frac{\pi}{2}\frac{1}{\Gamma\left(1-p\right)\sin\left(\frac{\pi p}{2}\right)}\mathcal{C}_{-p}\left(\rho_{R_{n-p} K}^{n-p}\right)
    \\
    &= \frac{\pi}{2}\frac{\vol_n(R_{n-p}K)(n-p)}{\Gamma\left(1-p\right)\sin\left(\frac{\pi p}{2}\right)}\rho_{\Gamma^\circ_{-p}(R_{n-p} K)}^p
    \\
    &=(n-p)\Gamma(p)\cos\left(\frac{\pi p}{2}\right)\vol_n(R_{n-p}K)\rho_{\Gamma^\circ_{-p}(R_{n-p} K)}^p,
    \end{align*}
    where we used \eqref{eq:Lq_intersection_bodies} and then the identity \eqref{eq:gamma_ident} in the last line. We deduce that, for $p\in  (0,1)$,
    \begin{equation}
    \label{eq:F_p_neg}
        F_{p} K = \left(\Gamma(p+1)\cos\left(\frac{\pi p}{2}\right)\vol_n(R_{n-p}K)\right)^\frac{1}{p}\Gamma^\circ_{-p}(R_{n-p} K).
    \end{equation}
    By Theorem~\ref{t:berck}, $F_{p} K$ is therefore an origin-symmetric convex body. 
\end{proof}

Lemma~\ref{l:F_p<1} complements Theorem~\ref{t:p_n_+_1}. We now prove Lemma~\ref{l:F_1K_is_convex}, which, in particular, yields the $p=1$ case of Theorem~\ref{t:main_convexity}.
\begin{proof}[Proof of Lemma~\ref{l:F_1K_is_convex}]
    If $n=1$, Plancherel's theorem gives
    \[
    \rho_{F_1K}(\pm1)=\frac{1}{\vol_1(K)}\int_0^\infty|\widehat{\chi_K}(r)|^2\,dr=\pi.
    \]
    Since $R_0K$ is an origin-symmetric interval and $\vol_0(\{0\})=1$, we also have $I(R_0K)=[-1,1]$, proving the claim. 
    
    Henceforth, assume $n\geq 2$. By Theorem~\ref{t:fourier_small_p}, we have
    \[
    \widehat{\rho_{R_{n-1}K}^{n-1}} = (n-1)\rho_{F_{1}K}.
    \]
    By the Fourier-analytic classification of intersection bodies,
    \eqref{eq:intersection_fourier}, we obtain
    \[
    \pi(n-1)\rho_{I(R_{n-1}K)}=(n-1)\rho_{F_{1}K}.
    \]
    It follows that
    \begin{equation}
    \label{eq:F_1_limit}
    \begin{split}
    F_1K= \pi I(R_{n-1} K).
    \end{split}
    \end{equation}
    In particular, $F_{1} K $ is a dilate of an intersection body of an origin-symmetric convex body, and, therefore, by Busemann's theorem, is an origin-symmetric convex body.
\end{proof}

R. Gardner and A. Giannopoulos \cite{GG99} had defined, for $p>-1$, the $p$th cross-section body of $K$ to be the star body $C_p K$, given by the radial function, for $u\in \s^{n-1}$, and $p\neq 0$,
\begin{align*}
    \rho_{C_p K}(u) &= \left(\frac{1}{\vol_n(K)}\int_K \rho_{I(K-x)}^p(u)dx\right)^\frac{1}{p} 
    \\
    &= \left(\frac{1}{\vol_n(K)}\int_K \vol_{n-1}((K-x)\cap u^\perp)^pdx\right)^\frac{1}{p}.
\end{align*}
The case $p=0$ is the corresponding geometric mean, obtained by taking the limit in $p$. As $p\to +\infty$, $C_p K \to CK$, where $CK$ is the cross section body of $K$, first introduced by H. Martini \cite{MH92}. The body $CK$ is convex when $n=3$ \cite{MM99}, but not convex in general for $n\geq 4 $ \cite{UB99}. R. Gardner and A. Giannopoulos showed that $C_p K$ is not convex in general for $p$ large enough. Moreover, they established that $C_{1} K$ is convex, in fact they showed $C_1 K = I(R_{n-1} K).$ But, this means we must have
$
F_1 K = \pi C_1 K$. 
 R. Gardner and A. Giannopoulos conjectured that $C_p K$ is convex when $p$ is small enough, say $-1<p\leq 1$, or for all $p>-1$ by imposing symmetry. These conjectures remain open.

\section{Monotonicity and Isoperimetric Inequalities}
\label{sec:mono_iso}
\subsection{Polar Mean Zonoids}
\label{sec:polar_functions}
This section is dedicated to studying the $Z^\circ_p K$, the polar $p$th mean zonoids of $K$. First, we prove our isoperimetric-type inequality for $Z^\circ_p K$. 

\begin{proof}[Proof of Theorem~\ref{t:polar_BS}]
Given the assumptions on $p$, we can apply Theorem~\ref{t:LZ} with $M=R_{n+p}K$ and obtain from \eqref{eq:mean_zonoid_relate} the inequality
\[
\vol_n(K)\vol_n(Z_p^\circ K)\leq \vol_n(\B)\vol_n(\Gamma_p^\circ \B)\left(\frac{\vol_n(R_{n+p} K)}{\vol_n(K)}\right)^{-\left(1+\frac{n}{p}\right)}.
\]
We take into account the sign of $p$. If $p>0$, then the exponent $-\left(1+\frac{n}{p}\right)$ is negative and $n+p>n$; thus, we use \eqref{eq:iso_R_pK_two} to conclude. On the other hand, if $p\in (-1,0)$, then the exponent $-\left(1+\frac{n}{p}\right)$ is positive and $n+p <n$; in this case, we use \eqref{eq:iso_R_pK} to conclude. 

The equality conditions are inherited from the Haddad-Ludwig inequalities. Indeed, there is equality in the Haddad-Ludwig inequalities if and only if $K$ is an ellipsoid; observe, in turn, that this implies $R_{n+p} K$ is an origin-symmetric ellipsoid by Proposition~\ref{p:R_pE}, allowing equality in Theorem~\ref{t:LZ}.
\end{proof}

Observe that the radial function of $Z_p^\circ K$ can be given by, for $p>-1;$ $p\neq 0$
\begin{equation}
\label{eq:new_radial_Zstar}
\begin{split}
    \rho_{Z_{p}^\circ K}(\theta) &= \left(\frac{1}{\vol_n(K)(n+p)}\int_{\s^{n-1}}|\langle u, \theta \rangle|^{p}\rho_{R_{n+p} K}^{n+p}(u)du \right)^{-\frac{1}{p}}
    \\
    &=\left(\frac{1}{\vol_n(K)^2}\int_{\s^{n-1}}|\langle u, \theta \rangle|^{p}\int_{0}^\infty g_K(ru)r^{n+p-1}dr\,du \right)^{-\frac{1}{p}}
    \\
    &=\left(\frac{1}{\vol_n(K)^2}\int_{\s^{n-1}}\int_{0}^\infty g_K(ru)|\langle ru, \theta \rangle|^{p}r^{n-1}dr\,du \right)^{-\frac{1}{p}}
    \\
        &=\left(\frac{1}{\vol_n(K)^2}\int_{\R^n}|\langle \theta,z\rangle|^{p}g_K(z)dz\right)^{-\frac{1}{p}}, \quad \theta\in \s^{n-1},
\end{split}
\end{equation}
where, in the first line, we already used polar coordinates, like in the proof of Theorem~\ref{t:p_n_+_1}; then, we inserted the radial function of $R_{n+p} K$, used Fubini's theorem, and then polar coordinate integration again. Using limits, we can consistently define $Z_0^\circ K$ via the radial function
\begin{equation}
\rho_{Z_{0}^\circ K}(\theta)=\exp\left(-\frac{1}{\vol_n(K)^2}\int_{\R^n}\log |\langle \theta,z\rangle|g_K(z)dz\right), \quad  \theta\in \s^{n-1}.
\label{eq:Z_o_circ}
\end{equation}

\begin{rem} Recall the polar of a convex body $L$ containing the origin is $L^\circ=\{x\in \R^n: \langle x,y\rangle \leq 1 \, \forall y\in L\}$. We used the epithet ``polar $p$th mean zonoid'' for $Z^\circ_p K$ because D. Xi, L. Guo and G. Leng \cite{XGL14} showed that the set $Z_p K=(Z^\circ_p K)^\circ$, is the $L^p$ version of the mean zonoid $ZK$ introduced by G. Zhang \cite{Zhang91}; precisely, $Z_1 K=ZK$. In a similar strategy to the proof of Theorem~\ref{t:polar_BS}, one can use the $L^p$ Busemann-Petty centroid inequality by E. Lutwak, D. Yang and G. Zhang \cite{LYZ00} and the Haddad-Ludwig inequalities in Theorem~\ref{t:radial_isos} to prove \cite[Theorem 1]{XGL14}, which is a sharp affine isoperimetric inequality for the $p$th mean zonoids $Z_p K$ when $p\geq 1$. In turn, with this inequality in hand, one can combine it with the Blaschke-Santal\'o inequality (see e.g., the survey \cite{FMZ23}) for another proof of Theorem~\ref{t:polar_BS} when $p\geq 1$.
\end{rem}

We now establish the following set-inclusions, which are reminiscent of those in \eqref{eq:radial_set_inclusion} for $R_p K$. The origin-symmetric convex body $(DK)^\circ\subset \R^n$ is defined by the radial function $\rho_{(DK)^\circ} = \left(\max_{x\in DK} \langle x, \cdot \rangle\right)^{-1}$.
\begin{thm}
\label{t:z_opposite_chain}
    Let $n\geq 1$ and let $K\subset \R^n$ be a convex body. Then, for $-1<p<q<\infty$, we have $$\kappa(p)Z_{p}^\circ K \subseteq \kappa(q)Z_{q}^\circ K\subseteq (DK)^\circ, \quad \text{where }\kappa(p) = \begin{cases}
        \dbinom{2n+p}{p}^{-\frac{1}{p}}, & p\neq 0,
        \\
        e^{-H_{2n}} , & p=0.
    \end{cases}$$ 
    There is equality if and only if $n=1$.
\end{thm}
\begin{proof}
Fix $\theta\in\s^{n-1}$. We first observe that, by \eqref{eq:new_radial_Zstar} and the convolution identity \eqref{eq:convolution}, we have for all $p>-1,p\neq 0$, by Fubini's theorem
\begin{equation}
\label{eq:berwald_product_moment}
\begin{split}
&\rho_{Z_p^\circ K}(\theta)^{-p}
\\
&= \frac{1}{\vol_n(K)^2}\int_{\R^n}
|\langle z,\theta\rangle|^p g_K(z)\,dz \\
&=\frac{1}{\vol_n(K)^2}\int_{K\times(-K)}
|\langle x+y,\theta\rangle|^p\,dx\,dy\\
&=\int_{\{(x,y)\in K\times (-K): \langle x+y,\theta\rangle \geq 0 \}}\langle x+y,\theta\rangle^p\,\frac{2dx\,dy}{\vol_{2n}(K\times (-K))}.
\end{split}
\end{equation}
Similarly, for $p=0$, we get
\[
\rho_{Z_0^\circ K}(\theta)^{-1}=\exp\left(\int_{\{(x,y)\in K\times (-K): \langle x+y,\theta\rangle \geq 0 \}}\log\langle x+y,\theta\rangle\,\frac{2dx\,dy}{\vol_{2n}(K\times (-K))}\right).
\]

The linear isometry $(x,y)\longmapsto(-y,-x)$ preserves $K\times(-K)$ and changes the sign of $\langle x+y,\theta\rangle$. Consequently, the convex body $D_\theta\subset \R^{2n}$ given by
$$D_\theta = \left(K \times (-K)\right) \cap \{(x,y)\in \R^{2n}: \langle x+y,\theta\rangle \geq 0 \}$$
satisfies
\begin{equation}
\label{scrap:eq:berwald_half_volume}
\vol_{2n}(D_\theta)=\frac{1}{2}\vol_{2n}(K\times (-K)).
\end{equation}
Therefore, \eqref{eq:berwald_product_moment}, and the $p=0$ case afterwards, can be realized as
\[
\rho_{Z_p^\circ K}(\theta)^{-1}=\begin{cases}\left(\int_{D_\theta}\langle x+y,\theta\rangle^p\,\frac{dx\,dy}{\vol_{2n}(D_\theta)}\right)^\frac{1}{p},&  p>-1,\ p\neq 0,
\\
\exp\left(\int_{D_\theta}\log\langle x+y,\theta\rangle\,\frac{dx\,dy}{\vol_{2n}(D_\theta)}\right), &p=0.
\end{cases}
\]

Since the function $(x,y)\mapsto \langle x+y,\theta\rangle$ is concave on compact sets, Berwald's inequality, Theorem~\ref{t:berwald}, in dimension $d=2n$ says that the function
\begin{equation}
\label{:eq:berwald_normalized_moment}
\mathcal G(p)=\bigl(\kappa(p)\rho_{Z_p^\circ K}(\theta)\bigr)^{-1}
\end{equation}
is decreasing. Taking reciprocals yields the claimed inclusions. Next, since $\kappa(q)\to1$ as $q\to\infty$,
\[
\lim_{q\to\infty}\mathcal G(q)
=\max_{D_\theta} \langle x+y,\theta \rangle
=\max_{z\in DK}\langle z,\theta\rangle
=\rho_{(DK)^\circ}(\theta)^{-1},
\]
which yields the limiting body.

For the equality conditions, if $\kappa(p)Z_p^\circ K=(DK)^\circ$
for some $p>-1$, then, by the monotonicity proven, we have
$\kappa(q)Z_q^\circ K=\kappa(p)Z_p^\circ K$ for every $q>p$.
Thus, we are reduced to the case of two finite parameters.

Suppose that $\kappa(p)Z_p^\circ K=\kappa(q)Z_q^\circ K$
for some $-1<p<q<\infty$. Fix $\theta\in\s^{n-1}$ and write
\[
h(x,y):=\langle x+y,\theta\rangle,
\qquad
A:=\max_{D_\theta}h>0.
\]
By the equality conditions in Theorem~\ref{t:berwald}, there exists $z_0\in D_\theta$ such that
\[
h(z)
=A\ell_{D_\theta-z_0}(z-z_0),
\qquad z\in D_\theta.
\]
In particular, $h$ attains its maximum at precisely $z_0$.
Consequently, the maximum and minimum of
$x\mapsto\langle x,\theta\rangle$ on $K$ are attained at
unique points $a,b\in K$, respectively, and $z_0=(a,-b)$.

The superlevel sets of tent functions have the form
\[
\{z\in D_\theta:h(z)\geq tA\}
=(1-t)D_\theta +tz_0,
\qquad 0\leq t<1.
\]
Now, $\{a\}\times(-K)\subseteq D_\theta$. For any
$z=(a,y)\in\{a\}\times(-K)$ other than $z_0$, set
$t=h(z)/A<1$. By the preceding identity, we may write
\[
z=(1-t)w+tz_0,
\qquad w\in D_\theta.
\]
Comparing first coordinates gives $w=(a,v)$ for some $v\in-K$. Notice that $h$ is linear on its support. Therefore,
\[
tA=h(z)=h((1-t)w+tz_0)=(1-t)h(w)+th(z_0) = (1-t)h(w)+tA.
\]
Thus, $h(w)=0$; inserting $w=(a,v)$ into $h(x,y)=\langle x+y,\theta\rangle$, we deduce
\[
\langle -v,\theta\rangle=\langle a,\theta\rangle.
\]
By the uniqueness of $a$, we obtain $v=-a$. Therefore,
\[
\{a\}\times(-K)
\subseteq \operatorname{conv}\{(a,-a),(a,-b)\}.
\]
The set on the left has dimension $n$, whereas the segment
on the right has dimension one. Thus, $n=1$.

\begin{figure}[htbp]
\centering
\begin{tikzpicture}[
    font=\small,
    edge/.style={draw=blue!65!black, line width=1.3pt},
    ray/.style={draw=orange!85!black, line width=1pt},
    outline/.style={draw=black!75, line width=.7pt},
    ghost/.style={draw=black!30, densely dashed, line width=.5pt},
    point/.style={circle, fill=black, inner sep=1.5pt},
    apex/.style={circle, fill=orange!85!black, inner sep=1.8pt},
    hit/.style={circle, fill=blue!65!black, inner sep=1.8pt},
    >=Latex
]
\begin{scope}[yshift=2.3cm]
    \coordinate (L) at (0,3.2);       
    \coordinate (R) at (3.2,0);       
    \coordinate (O) at (3.2,3.2);     
    \coordinate (U) at ($(L)!.62!(R)$);
    \coordinate (Z) at ($(O)!.52!(U)$);

    \node at (1.6,4.25) {$\theta=1,\qquad h(x,y)=x+y$};
    \fill[blue!6] (L)--(O)--(R)--cycle;
    \draw[ghost] (L)--(0,0)--(R);
    \draw[outline] (L)--(O)--(R);
    \draw[edge] (L)--(R);
    \node[text=blue!65!black] at (.8,1.95) {$E$};
    \node[text=black!55] at (2.55,2.8) {$D_1$};

    \draw[ray] (O)--(U);
    \draw[ray,densely dashed,->] (U)--($(O)!1.22!(U)$);
    \node[apex,label=above:{$z_0=(b,-a)$}] at (O) {};
    \node[hit,label=left:{$u$}] at (U) {};
    \node[point,label=right:{$z$}] at (Z) {};
    \node[above left] at (L) {$(a,-a)$};
    \node[below right] at (R) {$(b,-b)$};
\end{scope}

\begin{scope}[xshift=6.3cm,yshift=2.3cm]
    \coordinate (L) at (0,3.2);
    \coordinate (R) at (3.2,0);
    \coordinate (O) at (0,0);         
    \coordinate (U) at ($(L)!.62!(R)$);
    \coordinate (Z) at ($(O)!.52!(U)$);

    \node at (1.6,4.25) {$\theta=-1,\qquad h(x,y)=-(x+y)$};
    \fill[blue!6] (L)--(O)--(R)--cycle;
    \draw[ghost] (L)--(3.2,3.2)--(R);
    \draw[outline] (L)--(O)--(R);
    \draw[edge] (L)--(R);
    \node[text=blue!65!black] at (.8,1.95) {$E$};
    \node[text=black!55] at (.55,1.05) {$D_{-1}$};

    \draw[ray] (O)--(U);
    \draw[ray,densely dashed,->] (U)--($(O)!1.22!(U)$);
    \node[apex,label=below:{$z_0=(a,-b)$}] at (O) {};
    \node[hit,label=below:{$u$}] at (U) {};
    \node[point,label=below right:{$z$}] at (Z) {};
    \node[above left] at (L) {$(a,-a)$};
    \node[below right] at (R) {$(b,-b)$};
\end{scope}

\begin{scope}[xshift=4.75cm]
    \draw[<->,black!65] (-4.25,0)--(4.25,0) node[right] {$\mathbb R$};
    \draw[line width=2pt,blue!65!black] (.75,0)--(3.25,0);
    \draw[line width=2pt,violet!75!black] (-3.25,0)--(-.75,0);
    \foreach \x/\lab in {-3.25/{-b},-.75/{-a},0/{0},.75/{a},3.25/{b}} {
        \draw (\x,-.08)--(\x,.08);
        \node[below=4pt] at (\x,0) {$\lab$};
    }
    \foreach \x in {.75,3.25}
        \fill[blue!65!black] (\x,0) circle (1.8pt);
    \foreach \x in {-3.25,-.75}
        \fill[violet!75!black] (\x,0) circle (1.8pt);
    \node[above=7pt,text=blue!65!black] at (2,0) {$K=[a,b]$};
    \node[above=7pt,text=violet!75!black] at (-2,0) {$-K=[-b,-a]$};
\end{scope}
\end{tikzpicture}
\caption{The two triangles $D_\theta$ for $K=[a,b]$,
illustrated with $0<a<b$. Their common edge $E$ lies
on $x+y=0$, and $z_0$ is the opposite vertex.
The ray from $z_0$ through $z$ meets $E$ at $u$.}
\label{fig:berwald_interval_tents}
\end{figure}

Conversely, if $n=1$, then $K$ is an interval $[a,b]$ and
\[
D_\theta
=\left([a,b]\times[-b,-a]\right)
\cap\{(x,y):\langle x+y,\theta\rangle\geq0\}
\]
is a triangle. The function $h$ vanishes on the edge
\[
E:=\operatorname{conv}\{(a,-a),(b,-b)\}
\]
and attains its maximum $A=b-a$ at the opposite vertex
\[
z_0=
\begin{cases}
(b,-a),&\theta=1,\\
(a,-b),&\theta=-1.
\end{cases}
\]

For any $z\in D_\theta\setminus\{z_0\}$, the ray from $z_0$ through $z$ meets $E$ at a point $u$, and we may write
\[
z=(1-s)z_0+su,\qquad 0<s\leq1.
\]
Since $h(z_0)=A$ and $h(u)=0$, the linearity of $h$ gives $h(z)=(1-s)A$. Moreover, $u-z_0$ is the last point of $D_\theta-z_0$ along this ray. Thus, for $t>0$,
\[
z-z_0\in t(D_\theta-z_0)
\quad\Longleftrightarrow\quad t\geq s,
\]
and hence $\|z-z_0\|_{D_\theta-z_0}=s$. Consequently,
\[
h(z)=A(1-s)
=A\left(1-\|z-z_0\|_{D_\theta-z_0}\right)
=A\ell_{D_\theta-z_0}(z-z_0).
\]
The same identity holds at $z=z_0$, where both sides equal $A$. Therefore, the equality conditions in Theorem~\ref{t:berwald} give equality for all $p,q$. Taking the limit as $q\to\infty$ gives equality with $(DK)^\circ$ as well.
\end{proof}

\subsection{Isoperimetric Inequalities for Fourier Mean Bodies}
\label{sec:mono}
We obtain from Lemma~\ref{l:F_p<1} and Theorem~\ref{t:z_opposite_chain} the following immediate corollary. Note that the value of the coefficient at $q=1$ is obtained by taking the limit.
\begin{cor}
\label{c:F_set_inclusions}
    Let $n\geq 1$, $0<p<q\leq 1$, and let $K\subset \R^n$ be a convex body. Then,
\[
\lambda(q)\vol_n(K)^{-\frac{1}{q}} F_q K \subseteq \lambda(p) \vol_n(K)^{-\frac{1}{p}}F_p K,
\]
where
\[
\lambda(p) =\begin{cases}
    \left(\frac{\binom{2n-p}{2n}}{\Gamma\left(p+1\right)\cos\left(\frac{\pi p}{2}\right)}\right)^\frac{1}{p}, & p\in (0,1),
    \\
     \frac{1}{\pi n}, & p=1.
\end{cases}  
\]
\end{cor}

We are therefore in a position to prove Proposition~\ref{p:decreasing_sets}.
\begin{proof}[Proof of Proposition~\ref{p:decreasing_sets}]
By \eqref{eq:growing_sets} and Corollary~\ref{c:F_set_inclusions}, we have, for $0<p<q\leq 1,$
\[
\frac{\lambda(q)}{\lambda(p)}\vol_n(K)^{\frac{1}{p}-\frac{1}{q}} F_q K \subseteq F_p K \subset \vol_n(K)^{\frac{1}{p}-\frac{1}{q}} F_q K. 
\]
Consequently, to establish the set inclusions \eqref{eq:decreasing_sets}, it suffices to prove that $\frac{\lambda(q)}{\lambda(p)} \geq \frac{2}{\pi}$, or, equivalently, $\frac{\lambda(p)}{\lambda(q)}\le \frac{\pi}{2}$, for $0<p<q\le 1$. To show this, it is enough to prove that $\lambda(p)$ is decreasing; then
    $$
      \frac{\lambda(p)}{\lambda(q)}\le   \frac{\lim_{p \to 0^+}\lambda(p)}{\lambda(1)} =n\pi e^{-\psi(2n+1)} < \frac{\pi}{2}.
    $$ 
    Using \eqref{eq:gamma_sine_reflection}, we have $$\lambda (p) = \left(\frac{2}{\pi p}\frac{\sin(\pi p /2) \Gamma (2n+1 -p)}{\Gamma (2n+1)} \right)^\frac{1}{p}.$$ Let 
\begin{align*}    
    f(p)&= \log \lambda (p)
    \\
    &=\frac{1}{p}\left(\log\Gamma (2n+1 -p)-\log \Gamma (2n+1)+\log\left( \frac{\sin \left(\frac{\pi p}{2}\right)}{\frac{\pi p}{2}}\right)\right).
    \end{align*} 
    Define the function $h(x)=\log \left(\frac{\sin x}{x}\right)-(x \cot (x)-1)$. Then,
    \begin{align*}
    p^2\cdot f'(p)=\int_0^p \psi(2 n+1-t) d t-p\cdot \psi(2 n+1-p) - h\left(\frac{\pi p}{2}\right).
    \end{align*}
   
    Using integration by parts and the fact that $$\psi' (z) =\sum_{k=0}^\infty\frac{1}{(z+k)^2} \leq\frac{1}{z-\frac{1}{2}} \qquad \text{ for }z >\frac{1}{2},$$ we obtain
    \begin{align*}
        &\left(\int_0^p \psi(2 n+1-t) d t\right)-p\cdot \psi(2 n+1-p) =\int_0^pt \psi'(2 n+1-t) d t 
        \\
        &\leq\int_0^p \frac{2t}{4n+1-2t} d t \leq \frac{2}{4n+1-2p}\int_0^p t d t=\frac{p^2}{4n+1-2p}.
    \end{align*}
    For the remaining term, we notice that  $h(x) \geq \frac{x^2}{6},$ for $ 0 < x \leq \pi/2$. Therefore, we have
    $$f'(p) \leq \frac{1}{p^2}\left(\frac{p^2}{4n + 1 - 2p} - \frac{\pi^2 p^2}{24}\right) = \frac{1}{4n + 1 - 2p} - \frac{\pi^2}{24},$$
    which is negative for all $n$. We conclude.
\end{proof}

\begin{rem}
 \label{r:no_equality}
 Assume $n\geq 2$. We saw that equality never occurs in Theorem~\ref{t:z_opposite_chain} and Corollary~\ref{c:F_set_inclusions}. We provide here an additional geometric reason. Firstly, there is equality for a single $p$ and $q$ if and only if there is equality for all $q$. By sending $q\to \infty$, the body $K$ must solve $Z_p^\circ K= \binom{2n+p}{p}^\frac{1}{p}\left(DK\right)^\circ$ for all $p>-1$. From an application of \eqref{eq:F_1_limit}, we deduce that we must have

    \begin{equation}
    \label{eq:equality_F_1}
     I\left(R_{n-1} K\right) =\frac{1}{\pi}F_1 K = n\vol_n(K) \left(DK\right)^\circ.
    \end{equation}
    By using that $R_{n-1} \B$ is a particular dilate of $\B$ from Proposition~\ref{p:R_pE}, one can verify that $\B$ does not solve \eqref{eq:equality_F_1}. Since \eqref{eq:equality_F_1} is invariant under linear transformations, we deduce that ellipsoids cannot solve \eqref{eq:equality_F_1}.
    
    Next, observe that, if $K$ is a polytope, then $(DK)^\circ$ is an origin-symmetric polytope. However, it follows from \cite{GZ99_4,CS99} that no origin-symmetric polytope is the intersection body of a star body when $n\geq 3$. 
\end{rem}

 Many of the operators we discussed enjoy a monotonicity property. For example, for convex bodies $K$ and $M$ such that $K\subset M$, it holds that $IK \subset I M$. This type of property cannot hold for the operator $F_p$; indeed, for $n\geq 2$, $[-1,1]^n\subset \sqrt{n}\B,$ but $F_2([-1, 1]^n)$ is an unbounded set by Theorem~\ref{t:pk_iso} and $F_2(\sqrt{n}B_2^n)$ is a bounded Euclidean ball by Proposition~\ref{p:ball}.  Nevertheless, $F_p$ satisfies an analogue of \eqref{eq:monotonic_R_pK} when $p \in (0,1]$.
 \begin{prop}
     Let $n\geq 1$ and $p\in (0,1]$. Then, if $K,M\subset \R^n$ are convex bodies such that $K\subset M$, it holds 
     \[
     F_p K \subset \left(\frac{\vol_n(M)}{\vol_n(K)}\right)^\frac{1}{p}F_p M.
     \]
 \end{prop}
 \begin{proof}
     For $p\in (0,1)$, we have the formula $$\rho_{F_{p} K}(\theta) = \left(\Gamma(p+1)\cos\left(\frac{\pi p}{2}\right)\int_{R_{n-p} K}|\langle x,\theta \rangle|^{-p} dx\right)^{\frac{1}{p}}$$ 
     from \eqref{eq:F_p_neg} and \eqref{eq:L_p}. 
     Thus, the set-inclusion 
     $$R_{n-p}K\subset \left(\frac{\vol_n(M)}{\vol_n(K)}\right)^\frac{1}{n-p} R_{n-p}M = R_{n-p}\left(\left(\frac{\vol_n(M)}{\vol_n(K)}\right)^\frac{1}{n-p}M\right)$$ 
     derived from \eqref{eq:monotonic_R_pK}, and the homogeneity of $F_p K$ from \eqref{eq:homogeneous}, implies the claim. If $p=1$ and $n=1$, then $F_1K=F_1M=[-\pi,\pi]$, so the claim follows from $\vol_1(K)\leq\vol_1(M)$. For $p=1$ and $n\geq 2$, Lemma~\ref{l:F_1K_is_convex} gives 
     \begin{align*}F_1 K = \pi I\left(R_{n-1}K\right) &\subset \pi I\left(R_{n-1} \left( \left(\frac{\vol_n(M)}{\vol_n(K)}\right)^\frac{1}{n-1}M\right)\right)
     \\
 &=F_1 \left( \left(\frac{\vol_n(M)}{\vol_n(K)}\right)^\frac{1}{n-1}M\right)\end{align*}
 which equals $\left(\frac{\vol_n(M)}{\vol_n(K)}\right) F_1M$, as claimed.
 \end{proof}

 \label{sec:F_pisos} 
 With Lemma~\ref{l:F_1K_is_convex} and Theorem~\ref{t:BI} available, we can prove Theorem~\ref{t:Fpk_affine_iso}.
 \begin{proof}[Proof of Theorem~\ref{t:Fpk_affine_iso}]
 As mentioned, the case when $p\neq 1$ follows from Theorem~\ref{t:polar_BS} and Lemma~\ref{l:F_p<1}. As for the case when $p=1$, we have by Lemma~\ref{l:F_1K_is_convex}, Theorem~\ref{t:BI} and \eqref{eq:iso_R_pK} the inequalities
     \begin{align*}
         \frac{\vol_n(F_1 K)}{\vol_n(K)^{n-1}}&=\pi^n  \frac{\vol_n(I(R_{n-1} K))}{\vol_n(K)^{n-1}}\leq \omega_n^2\left(\frac{\pi\omega_{n-1}}{\omega_{n}}\right)^n\left(\frac{\vol_n(R_{n-1}K)}{\vol_n(K)}\right)^{n-1}
         \\
&\leq\omega_n^2\left(\frac{\pi\omega_{n-1}}{\omega_{n}}\right)^n\left(\frac{\vol_n(R_{n-1}\B)}{\vol_n(\B)}\right)^{n-1}.
     \end{align*}
     For $p=1$ and $n\geq2$, equality in the last inequality forces $K$ to be an ellipsoid by the equality characterization in Theorem~\ref{t:radial_isos}, and conversely ellipsoids give equality throughout; the case $n=1$ is immediate since every convex body is an interval. By the equality characterization, the final constant is the right-hand side of \eqref{eq:F_p_affine}.
 \end{proof}

\section{Analysis of the Cube}
\label{sec:cube}
This section is dedicated to studying the behavior of the cube under the operators $R_p$ and $F_p$. First, we have:
\subsection{Radial Mean Bodies of the Cube}
\label{sec:radial_cube}
In light of Theorem~\ref{t:use_slicing}, it is natural to ask if $R_p K$ is close to an ellipsoid in the sense of Hensley's theorem. This section is dedicated to showing that this is not the case for $p\in [n-1,n)$. This fact is interesting by itself and, moreover, it illustrates that Theorem~\ref{t:use_slicing} does not follow immediately from Theorem~\ref{t:fourier_small_p}. As usual, the confounding example is the cube. We will need the covariogram of the interval $[-\frac{1}{2},\frac{1}{2}]:$ 
\begin{equation}g_{[-\frac{1}{2},\frac{1}{2}]}(t)=
\begin{cases}
    1-|t|, & |t|\leq 1,
    \\
    0, & \text{otherwise}.
    \end{cases}
    \label{eq:covario_line}
\end{equation}
We introduce the notation $Q_n:=\left[-\frac{1}{2},\frac{1}{2}\right]^n$ and denote by $e_i$ the usual canonical basis vectors. Recall the $\ell^\infty$ norm of $x\in\R^n$ is $\|x\|_\infty=\max_{1\leq i \leq n}|x_i|$.
\begin{prop}
    Fix $n\geq 1$ and $p>-1$. Then, for every $\theta\in\s^{n-1}$,
    \begin{equation}
\label{eq:radial_body_cube}
    \rho_{R_p Q_n}(\theta) \!=\! \begin{cases}
        \left(p\int_{0}^{\|\theta\|_{\infty}^{-1}} \prod_{i=1}^n(1-r|\theta_i|)r^{p-1}dr\right)^\frac{1}{p}, & p>0,
        \\
        \exp\left(\int_{0}^{\|\theta\|_{\infty}^{-1}}\!\!\! \left(\!\!\sum_{i=1}^n|\theta_i|\prod_{\substack{j=1 \\ j\neq i}}^n(1-r|\theta_j|)\right)\!\!\log(r)dr\right), & p=0,
        \\
        \left((-p)\int_{0}^{+\infty} \left(1-\prod_{i=1}^n(1-r|\theta_i|)_+\right)r^{p-1}dr\right)^\frac{1}{p}, & \!\!\!\!\!p \in (-1,0).
    \end{cases}
\end{equation}
In particular,
\begin{equation}
\label{eq:radial_cube_canonical}
    \rho_{R_p Q_n}(\pm e_i) =
    \begin{cases}
        \left(\frac{1}{p+1}\right)^\frac{1}{p}, & p>-1, p\neq 0,
        \\
       \frac{1}{e}, & p=0
    \end{cases}
\end{equation}
Finally, for every $\varepsilon\in \{-1,1\}^n:$
\begin{equation}
\label{eq:radial_cube_ham}
    \rho_{R_p Q_n}\left(\frac{\varepsilon}{\sqrt{n}}\right) =\sqrt{n}\rho_{R_p Q_n}\left(\varepsilon\right)=\begin{cases}
    \sqrt{n}\binom{n+p}{p}^{-\frac{1}{p}}, & p>-1, p\neq 0,
    \\
    \sqrt{n}e^{-H_n}, & p=0.
    \end{cases}
\end{equation}
\end{prop}
\begin{proof}
     We take advantage of the product structure to iteratively factor the covariogram of $g_{Q_n}$ using \eqref{eq:covario_line} and \eqref{eq:factor_g_K}: $g_{Q_n}((x_1,\dots,x_n)) = \prod_{i=1}^n(1-|x_i|)_+$.  We will also use that, if we write $\theta=(\theta_1,\dots,\theta_n)\in \s^{n-1}$, then, for almost every $r>0$,
\[
\frac{\partial}{\partial r}g_{Q_n}(r\theta) = -\sum_{i=1}^n|\theta_i|\chi_{[0,|\theta_i|^{-1}]}(r)\prod_{\substack{j=1 \\ j\neq i}}^n(1-r|\theta_j|)_+. 
\]
The claims follow from  \eqref{eq:radial_ell} and direct computation.
\end{proof}

\begin{figure}[ht]
\centering
\begin{subfigure}[b]{0.49\textwidth}
    \centering
    \includegraphics[width=\textwidth]{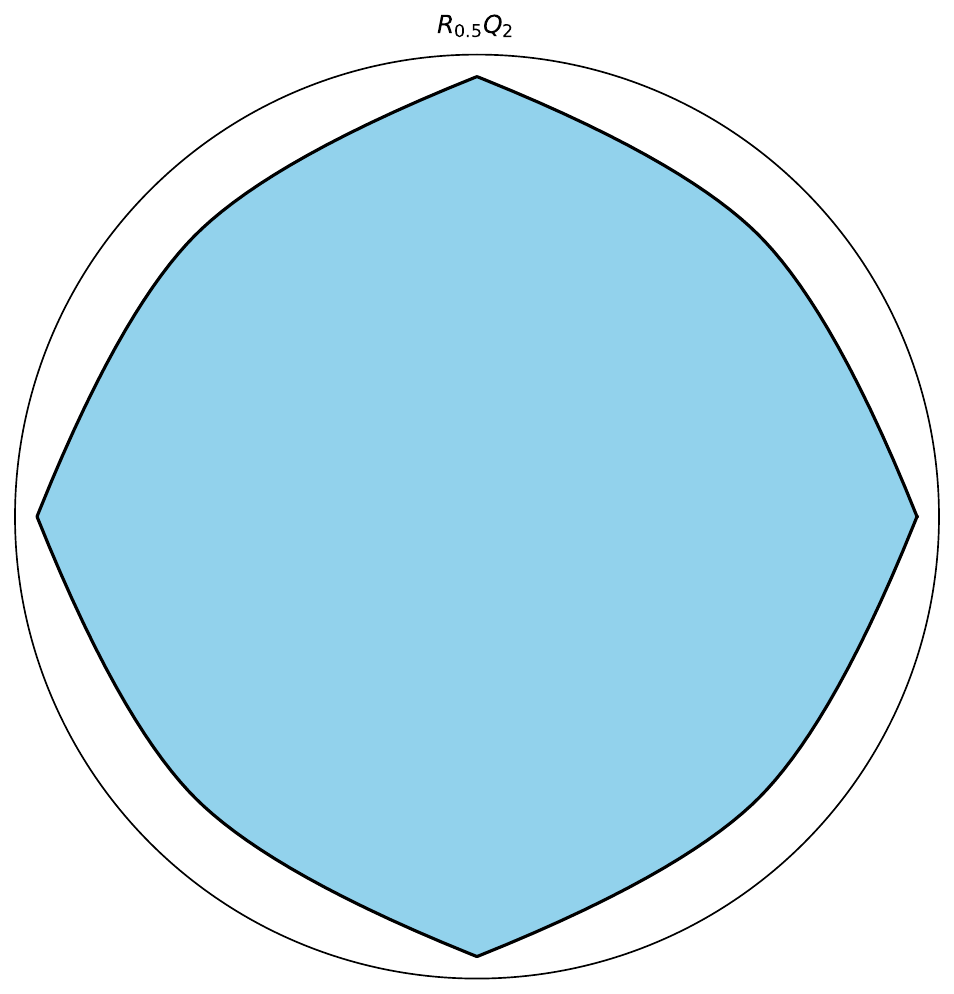}
    \label{fig:R0.5}
  \end{subfigure}
  \hfill
  \begin{subfigure}[b]{0.49\textwidth}
    \centering
    \includegraphics[width=\textwidth]{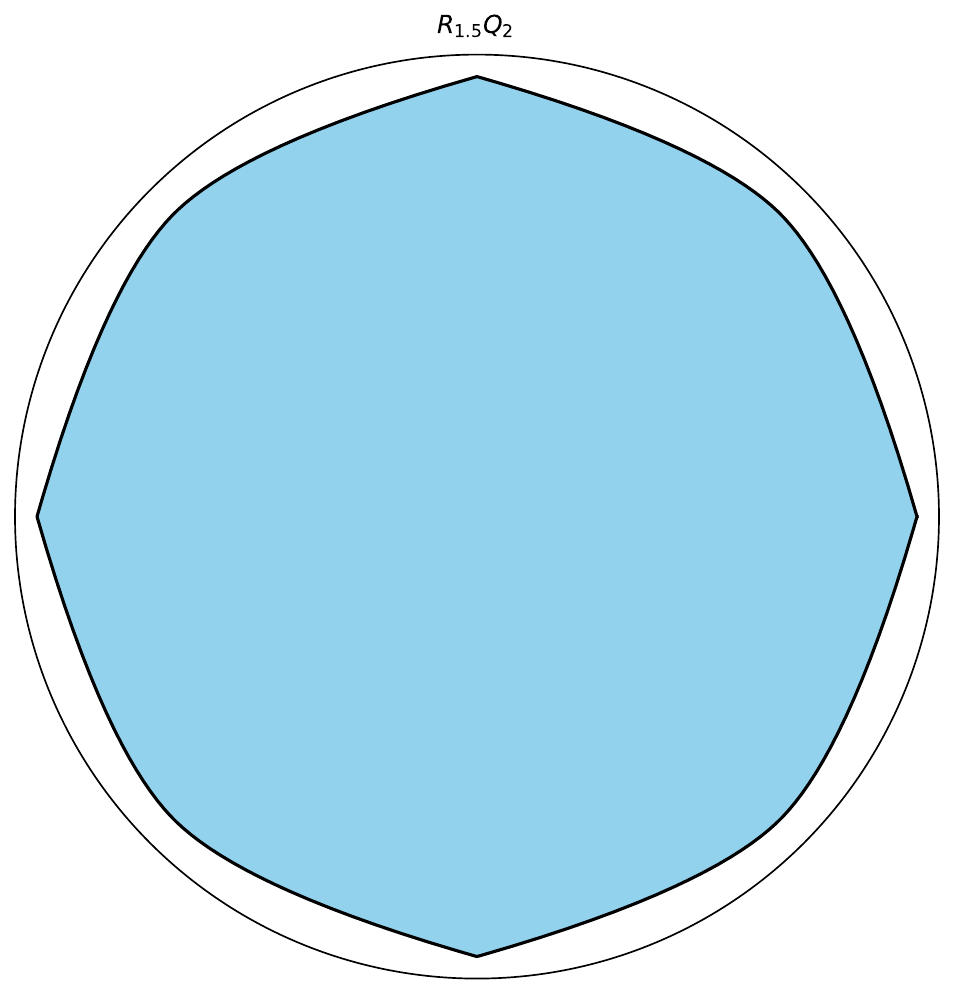}
    \label{fig:R1.5}
  \end{subfigure}
\caption{The body $R_p Q_2$ when $p=0.5$ and $p=1.5$.}
  \label{fig:R_p}
\end{figure}

We will show more strongly that there does not exist a position of $Q_n$ such that $R_p Q_n$ is isomorphic to a ball (with dimension-free radius) when $p\in [n-1,n)$. Recall that the Banach-Mazur distance between two origin-symmetric  bodies $K,M\subset \R^n$ is given by
\begin{equation}
    d_{\operatorname{BM}}(K,M) = \min \left\{d \ge 1 : T M\subseteq K\subseteq d TM, \quad \text{for some } T\in GL_n(\R)\right\}.
\end{equation}
The classical result by F. John \cite{FJ48} tells us that $d_{\operatorname{BM}}(K,B_2^n)\le \sqrt{n}$ for any origin-symmetric convex body $K$, thus $d_{\operatorname{BM}}(R_pQ_n,\B) \leq \sqrt{n},$ for $p>-1$. Next we will provide a lower bound. Recall the generalized binomial coefficients from \eqref{eq:generalized_binomial}.
\begin{prop}
    Fix $n\geq 1$ and $p>-1$. Then, $$\sqrt{n} \left(\frac{p+1}{\binom{n+p}{p}}\right)^\frac{1}{p}\leq d_{\operatorname{BM}}(R_pQ_n,\B).$$ At $p=0$, the constant should be understood as $\sqrt{n}e^{1-H_n}$. In particular, for $p\in [n-1,n)$, $ d_{\operatorname{BM}}(R_pQ_n,\B) \approx \sqrt{n}$ as $n\to \infty$.
\end{prop}
\begin{proof}
Let $\mathcal{E}$ be an ellipsoid such that $\mathcal{E} \subset R_p Q_n \subset d\mathcal{E}$ for some $d\ge 1$. Consider $T=(t_{k,i})_{k,i}\in GL_n(\R)$ such that $\rho_{\mathcal{E}}(x) = |Tx|^{-1}$. Therefore, we have $|Tx|^{2} \geq  \rho_{R_p Q_n}(x)^{-2} \geq d^{-2}|Tx|^{2}$. Picking $x = \frac{\varepsilon}{\sqrt{n}}$, $\varepsilon\in \{-1,1\}^n,$ we deduce from \eqref{eq:radial_cube_ham}
\[
\frac{1}{d^2}\sum_{k=1}^n\left(\sum_{i=1}^n\frac{\varepsilon_i}{\sqrt{n}}t_{k,i}\right)^2\leq\frac{\binom{n+p}{p}^{\frac{2}{p}}}{n}\leq \sum_{k=1}^n\left(\sum_{i=1}^n\frac{\varepsilon_i}{\sqrt{n}}t_{k,i}\right)^2.
\]
Simplifying, we have
\[
\frac{1}{d^2}\sum_{k=1}^n\left(\sum_{i=1}^n\varepsilon_it_{k,i}\right)^2\leq \binom{n+p}{p}^{\frac{2}{p}}\leq \sum_{k=1}^n\left(\sum_{i=1}^n\varepsilon_it_{k,i}\right)^2.
\]
We next take the average of the above inequality over $\varepsilon \in \{-1,1\}^n$. When we open the squared term, cross terms of the form $\varepsilon_i\varepsilon_j$, $j\neq i$, average to zero. Therefore:
\[
\frac{1}{d^2}\sum_{k=1}^n\sum_{i=1}^n|t_{k,i}|^2\leq \binom{n+p}{p}^{\frac{2}{p}}\leq \sum_{k=1}^n\sum_{i=1}^n|t_{k,i}|^2.
\]
We isolate the inequality:
\begin{equation}
\label{eq:lower_BM}
d^2 \geq \binom{n+p}{p}^{-\frac{2}{p}}\sum_{k=1}^n\sum_{i=1}^n|t_{k,i}|^2
\end{equation}

On the other hand, we can pick $x=\pm e_i$ and obtain from \eqref{eq:radial_cube_canonical}
$$\sum_{k=1}^n|t_{k,i}|^2 \geq  \left(p+1\right)^\frac{2}{p} \geq d^{-2}\sum_{k=1}^n|t_{k,i}|^2.$$
Inserting this bound into \eqref{eq:lower_BM}, we have $d_{\operatorname{BM}}(R_pQ_n,\B) \geq \sqrt{n} \left(\frac{p+1}{\binom{n+p}{p}}\right)^\frac{1}{p},$ as claimed.
When $p\in [n-1,n)$, this is asymptotically on the order of $\sqrt{n}$ by Stirling's approximation.
\end{proof}

\subsection{Fourier Mean Bodies of the Cube}
\label{sec:non}
We establish Theorem~\ref{t:cube_not}, which states that $F_p [-1,1]^n$ is not convex when $n\geq 2$ and $1<p<2$. We need the rudimentary identity, for $n\geq 1$,
    \begin{equation}
    \label{eq:cube_fourier}
    \widehat{\chi_{[-1,1]^n}}(x) = 2^n\prod_{j=1}^n\frac{\sin(x_j)}{x_j}, \quad \text{where} \quad x=(x_1,\dots,x_n).
    \end{equation}
Thus,  we obtain from \eqref{eq:fourier_body} that
\begin{equation}
    \rho_{F_p[-1,1]^n}(x)=\left(2^np\int_{0}^{+\infty}\prod_{j=1}^n\left|\frac{\sin(x_jr)}{x_jr}\right|^2r^{p-1}dr\right)^\frac{1}{p}.
\end{equation}

Next, we show that the case for $n>2$ of Theorem~\ref{t:cube_not} follows from the planar case. Indeed, let $H$ be the plane spanned by the first two canonical basis vectors $e_1$ and $e_2$. Then, 
\begin{align*}
    \rho_{F_p[-1,1]^n\cap H}(x)&=\left(2^np\int_{0}^{+\infty}\left|\frac{\sin(x_1r)\sin(x_2r)}{x_1x_2r^2}\right|^2r^{p-1}dr\right)^\frac{1}{p} 
    \\
    &= 2^\frac{n-2}{p}\rho_{F_p[-1,1]^2}((x_1,x_2)).
\end{align*}
Supposing that $F_p[-1,1]^2$ is not convex, then the origin-symmetric body $F_p[-1,1]^n$ has a non-convex central section, i.e., $F_p[-1,1]^n$ is not convex.

We restrict our focus to $n=2$. We first need the following technical result. Henceforth, we define the positive constant
\[
        d_p := \frac{\cos \left(\frac{\pi p}{2} \right)\Gamma(p+1)2^{1-p}}{(4-p)(3-p)(2-p)(1-p)}, \quad 1<p<2.
    \]
\begin{lem}
    \label{l:RF}
    Let $p \in (1,2)$. For $x_1 \geq x_2 >0$, we have
    \begin{align*}
        &p\int_0^{+\infty} \sin^2(x_1r)\sin^2(x_2r) r^{p-5} dr
        \\
        &=
        d_p
        \left(
        (x_1-x_2)^{4-p}
        +(x_1+x_2)^{4-p} -2x_1^{4-p}-2x_2^{4-p}
        \right).
    \end{align*}
\end{lem}

\begin{proof}
    Define the function $C(t) = 1-\frac{t^2}{2} - \cos t.$ Then, from trigonometric identities, we have 
    \begin{equation}
        \label{eq:inRF}
        \begin{split}
        &\sin^2(x_1r)\sin^2(x_2r) 
        \\
        &=\frac{1}{4}(1- \cos(2x_1r))(1-\cos(2x_2r) )
        \\
        &= \frac{1}{4}- \frac{\cos(2x_1r)}{4}-\frac{\cos(2x_2r)}{4}+\frac{\cos(2x_1r)\cos(2x_2r)}{4}
        \\
        &= \frac{1}{4} -\frac{\cos(2x_1r)}{4}-\frac{\cos(2x_2r)}{4} + \frac{\cos(2(x_1+x_2)r)}{8}  +\frac{\cos(2(x_1-x_2)r)}{8} 
        \\
        &=
        \frac{C(2x_1r)}{4} +\frac{C(2x_2r)}{4} - \frac{C(2(x_1+x_2)r)}{8} -\frac{C(2(x_1-x_2)r)}{8}. 
        \end{split}
    \end{equation}
  Using a change of variables and Lemma \ref{l:int-cos}, we have 
    \begin{align*}
        \int_0^{+\infty} C(2x_1r)r^{p-5} dr &= (2x_1)^{4-p} \int_0^{+\infty} t^{p-5} C(t) dt 
        \\
        &= -(2x_1)^{4-p}\frac{\cos \left(\frac{\pi p}{2} \right)\Gamma(p)}{(4-p)(3-p)(2-p)(1-p)}.
    \end{align*}
    Similarly, we integrate $C(2x_2r),C(2(x_1+x_2)r)$ and $C(2(x_1-x_2)r)$, and substitute them into the integration of \eqref{eq:inRF} to complete the proof.
\end{proof}

\begin{prop}
    $F_p[-1,1]^2$ is not a convex body for $p >1$.
\end{prop}

\begin{proof}
    First, we observe that when $p \geq 2$, the set is not compact, as the radial function in the direction of $e_1:=(1,0)$ is infinite. Indeed, this follows from the divergence of the integral $ \int_0^{+\infty} \sin ^2(r) r^{p-3} d r.$ We henceforth assume that $1<p <2$. 
    
    Using Lemma~\ref{l:RF} with $x_1=\cos\theta$, $x_2=\sin\theta$, where $\theta \in (0, \pi/4]$, we obtain the radial function of $F_p[-1,1]^2$ in terms of $\theta$ is given as follows:
    \begin{align*}
        &\rho (\theta) := \rho_{F_p[-1,1]^2} ((\cos\theta,\sin\theta))
        \\
        &=\left(\frac{4d_p
        \left(
        (\cos \theta-\sin \theta)^{4-p}
        +(\cos \theta+\sin \theta)^{4-p} -2\cos^{4-p} \theta-2\sin^{4-p} \theta
        \right)}{\cos^2 \theta \sin^2 \theta}\right)^\frac{1}{p}
        \\
        &=\!\!\left(\!\!\frac{8d_p
        \left(
        2^{\frac{2-p}{2}}\left(\sin\left(\frac{\pi}{4} -  \theta\right)^{4-p}
        +\sin\left(\frac{\pi}{4} +  \theta\right)^{4-p}\right) \!-\!\left(\cos^{4-p} \theta+\sin^{4-p} \theta
        \right)\right)}{\cos^2 \theta \sin^2 \theta}\!\!\right)^\frac{1}{p}
    \end{align*}
    We will show that the curvature of $F_p[-1,1]^2$ is negative near $e_1$, i.e., one has
    \begin{equation}
    \rho^2 + 2(\rho')^2 -\rho \rho''  <0.
    \label{eq:rho_function}
    \end{equation} 
    near $\theta=0$. It then follows that $F_p[-1,1]^2$ is not convex.  Denote $r(\theta) = \frac{\rho^p(\theta)}{8d_p}$. The differential inequality  \eqref{eq:rho_function} becomes, in terms of $r$ and its derivatives:
    \begin{equation}
    \label{eq:r_function}
       r^2 +\left(\frac{1}{p} + \frac{1}{p^2}\right) (r')^2 - \frac{1}{p} r r''<0.
    \end{equation}
    We expand the function $\theta\mapsto \sin (\pi/4 \pm \theta)$ into its Taylor series to get 
    \begin{align*}
        2^{\frac{2-p}{2}} \sin^{4-p} \left(\frac{\pi}{4} \pm \theta\right)
        &= \frac{1}{2} \left( 1\pm \theta - \frac{1}{2} \theta^2 
        \mp \frac{1}{6} \theta^3 + \frac{1}{24 } \theta^4 + \mathcal{O}(\theta^5) 
        \right)^{4-p}
        \\
        &= \frac{1}{2} + \frac{4-p}{2}\left( \pm \theta - \frac{1}{2} \theta^2 
        \mp \frac{1}{6} \theta^3 + \frac{1}{24 } \theta^4  
        \right) 
        \\
        &
        \,\,\,\,\,\,\,
        + \frac{(4-p)(3-p)}{4} \left( \pm \theta - \frac{1}{2} \theta^2 
        \mp \frac{1}{6} \theta^3 
        \right)^{2}
        \\
        &
        \,\,\,\,\,\,\,
         +
        \frac{(4-p)(3-p)(2-p)}{12} \left( \pm \theta - \frac{1}{2} \theta^2 \right)^{3}
        \\
        &
        \,\,\,\,\,\,\,
        +
        \frac{(4-p)(3-p)(2-p)(1-p)}{48} \left( \pm \theta \right)^{4} +\mathcal{O}(\theta^5).
    \end{align*}
    Hence,
    \begin{align*}
        &2^{\frac{
    2-p}{2}} \left(\sin^{4-p} \left(\frac{\pi}{4} - \theta\right) + \sin^{4-p} \left(\frac{\pi}{4} + \theta\right)\right) 
    \\
        &\quad
        = 1+\frac{(4-p)(2-p)}{2}\theta^2 
        + \frac{(4-p)(2-p)(p^2 +2p -16)}{24}\theta^4 +\mathcal{O}(\theta^6).
    \end{align*}
    Here, the expression on the left is even, so the expansion contains only even powers.
    Similarly, 
    \begin{align*}
        \sin^{4-p} (\theta)  
        &= \left( \theta 
        +\mathcal{O}(\theta^3) 
        \right)^{4-p} 
        = \theta^{4-p} \left(1 
        +\mathcal{O}(\theta^2)
        \right)^{4-p}
        \\
        &= \theta^{4-p} \left(1 
        +\mathcal{O}(\theta^2)
        \right)
        =  \theta^{4-p} 
        +\mathcal{O}(\theta^{6-p}),
    \end{align*}
    and
    \begin{align*}
        \cos^{4-p} (\theta)  &= \left( 1 - \frac{\theta^2}{2} 
        + \frac{\theta
        ^4}{24}+\mathcal{O}(\theta^6) 
        \right)^{4-p} 
        \\
        &
        =  1 - \frac{4-p}{2} \theta^2 
        +\frac{(4-p)(10-3p)}{24}\theta^4 +\mathcal{O}(\theta^6).
    \end{align*}
    Thus, the numerator of $r(\theta)$ has the expansion
    \begin{equation}
        \label{eq:num_r_function}
        \begin{split}
            &
        2^{\frac{2-p}{2}}\left(\sin^{4-p}\left(\frac{\pi}{4} -  \theta\right)
        +\sin^{4-p}\left(\frac{\pi}{4} +  \theta\right)\right) -\left(\cos^{4-p} \theta+\sin^{4-p} \theta\right)
        \\
        &= \frac{(4-p)(3-p)}{2}\theta^2-  \theta^{4-p} 
        +\frac{(4-p)(3-p)\left(p^2+3 p-14\right)}{24} \theta^4
        \\
        &
        \,\,\,\,\,\,\,
        +\mathcal{O}(\theta^{6-p}) .
        \end{split}   
    \end{equation}
    Using that, 
    \begin{align*}
        \frac{1}{\sin^2 \theta} &=\frac{1}{\left( \theta - \frac{\theta^3}{6} +\mathcal{O}(\theta^5)\right)^2} = \frac{1}{\theta^2\left( 1 - \frac{\theta^2}{6} +\mathcal{O}(\theta^4)\right)^2} 
        \\
        &= \frac{1}{\theta^2} \left(1 +\frac{\theta^2}{3} 
        +\mathcal{O}(\theta^4)\right),
    \end{align*}
    and
    \begin{align*}
        \frac{1}{\cos^2 \theta} &=  \frac{1}{\left(1 - \frac{\theta^2}{2} +\mathcal{O}(\theta^4) \right)^2} = 1+ \theta^2 +\mathcal{O}(\theta^4),
    \end{align*}
    we get, 
    \begin{equation}
    \label{eq:den_r_function}
        \frac{1}{\sin^2 \theta \cos^2 \theta} = \frac{1}{\theta^2} \left( 1+ \frac{4}{3} \theta^2 +\mathcal{O}(\theta^4) \right).
    \end{equation}
    Using \eqref{eq:num_r_function} and \eqref{eq:den_r_function}, we have
    $$
        r(\theta) = \frac{(4-p)(3-p)}{2} -  \theta^{2-p}
        + \frac{(4-p)(3-p)(p+2)(p+1)}{24} \theta^2 +\mathcal{O}(\theta^{4-p}).
    $$
    More precisely, the remainder obtained from these Taylor expansions has the form $\theta^4 u(\theta) + \theta^{4-p} v(\theta)$ where $u,v$ are analytic near zero.
    It therefore follows that
    \begin{equation}\label{derivr}        r'(\theta) = - (2-p) \theta^{1-p}  +\mathcal{O}(\theta)
        \quad
        \text{and}
        \quad
        r''(\theta) = - (2-p) (1-p) \theta^{-p}  +\mathcal{O}(1).
    \end{equation}
    Substituting \eqref{derivr} into the left hand side of \eqref{eq:r_function}, we get
    \begin{align*}
        &r^2 +\left(\frac{1}{p} + \frac{1}{p^2}\right) (r')^2 - \frac{1}{p} r r''
        \\
        &= \left( \frac{(4-p)(3-p)}{2} + \mathcal{O}(\theta^{2-p})\right)^2 + \left( \frac{1}{p} + \frac{1}{p^2 }\right)\left( - (2-p) \theta^{1-p}  
        +\mathcal{O}(\theta)\right)^2 
        \\
        &\,\,\,\,\,\,
        -\frac{1}{p}\left( \frac{(4-p)(3-p)}{2} 
        + \mathcal{O}(\theta^{2-p})\right)\left(- (2-p) (1-p) \theta^{-p}  
        +\mathcal{O}(1)\right).
    \end{align*}
    Multiplying by $\theta^p$ and taking the limit $ \theta \to 0^+$, we obtain 
    \[
        \frac{(4-p)(3-p)(2-p)(1-p)}{2p},
    \]
    which is less than $0$ since $1<p <2$. This means that near $e_1$ there is a region of the boundary of $F_{p}[-1,1]^2$ with negative curvature. 
\end{proof}

\subsection{Discussion on Valuations}
In this section, we establish that $K\mapsto F_pK$ is not an
$L^p$-valuation in the convex range $0<p<1$ nor in the range
$1<p<2$. Let $\mathcal{P}^n_o$ denote the set of full-dimensional
convex polytopes containing the origin in $\R^n$.

\begin{prop}
\label{p:valuation}
For every $n\geq1$ and $p\in(0,2)\setminus\{1\}$, the mapping
$K\mapsto F_pK$ on $\mathcal{P}^n_o$ is neither an $L^p$ valuation
nor an $L^p$ radial valuation.
\end{prop}

\begin{proof}
Fix $n\geq1$ and $p\in(0,2)\setminus\{1\}$. Recall that, for a non-empty compact set $M\subset\R^n$, its support function is given by $h_M(u):=\sup_{x\in M}\langle x,u\rangle$.

The $L^p$ and $L^p$ radial valuation properties would require, respectively, the pointwise identities
\[
h_{F_pK}^p+h_{F_pL}^p
=
h_{F_p(K\cup L)}^p+h_{F_p(K\cap L)}^p
\]
and
\[
\rho_{F_pK}^p+\rho_{F_pL}^p
=
\rho_{F_p(K\cup L)}^p+\rho_{F_p(K\cap L)}^p,
\]
whenever $K,L,K\cup L$, and $K\cap L$ belong to $\mathcal{P}^n_o$. Thus, it suffices to find such $K,L$ for which both identities fail in a single direction.

Set $Q=[-1,1]^{n-1}$ and consider the two translated cubes
\[
K=\left[-\frac32,\frac12\right]\times Q,
\qquad
L=\left[-\frac12,\frac32\right]\times Q.
\]
Then,
\[
K\cap L=\left[-\frac12,\frac12\right]\times Q,
\qquad
K\cup L=\left[-\frac32,\frac32\right]\times Q.
\]
When $n=1$, the factor $Q$ is omitted, and its volume is understood to be $1$.

For an interval $I$ of length $\ell>0$, Fourier factorization gives
\[
\widehat{\chi_{I\times Q}}(re_1)
=
2^{n-1}\widehat{\chi_I}(r).
\]
By the definition of $F_p$ and a change of variables,
\[
\rho_{F_p(I\times Q)}(e_1)^p
=
2^{n-1}C_p\ell^{1-p},
\qquad
C_p:=\rho_{F_p[-1/2,1/2]}(1)^p\in(0,\infty).
\]
When $p\in(0,1)$, the function $t\mapsto t^{1-p}$ is strictly concave on $(0,\infty)$, and therefore $2\cdot2^{1-p}>1+3^{1-p}$. Consequently,
\begin{equation}
\label{eq:not_valuation}
\begin{split}
\rho_{F_pK}(e_1)^p+\rho_{F_pL}(e_1)^p
&=2^{n-1}C_p\bigl(2\cdot2^{1-p}\bigr)\\
&>2^{n-1}C_p\bigl(1+3^{1-p}\bigr)\\
&=\rho_{F_p(K\cap L)}(e_1)^p
+\rho_{F_p(K\cup L)}(e_1)^p.
\end{split}
\end{equation}
When $p\in(1,2)$, the function $t\mapsto t^{1-p}$ is strictly convex, and the strict inequality in \eqref{eq:not_valuation} is reversed. In either case, $F_p$ is not an $L^p$ radial valuation.

To obtain the same conclusion for the support functions, we show that, for every box $B=I\times Q$ considered above,
\begin{equation}
\label{eq:box_identity}
h_{F_pB}(e_1)=\rho_{F_pB}(e_1).
\end{equation}
Indeed, write $\theta=(\theta_1,\theta')\in\s^{n-1}$. Since $|\widehat{\chi_Q}|\leq2^{n-1}$, Fourier factorization and a change of variables give, whenever $\theta_1\neq0$,
\begin{align*}
\rho_{F_pB}(\theta)^p
&=\frac{p}{\ell\,2^{n-1}}
\int_0^\infty
|\widehat{\chi_I}(r\theta_1)|^2
|\widehat{\chi_Q}(r\theta')|^2r^{p-1}\,dr\\
&\leq
\frac{p\,2^{n-1}}{\ell}
\int_0^\infty
|\widehat{\chi_I}(r\theta_1)|^2r^{p-1}\,dr\\
&=|\theta_1|^{-p}\rho_{F_pB}(e_1)^p.
\end{align*}
Hence
\[
|\theta_1|\rho_{F_pB}(\theta)\leq\rho_{F_pB}(e_1)
\]
for every $\theta\in\s^{n-1}$. Now take any nonzero $x\in F_pB$. Write $x=r\theta$, $\theta\in\s^{n-1}$. By the definition of the radial function, $r\leq\rho_{F_pB}(\theta).$ Therefore,
\[
\langle x,e_1\rangle
=r\theta_1
\leq r|\theta_1|
\leq\rho_{F_pB}(\theta)|\theta_1|
\leq \rho_{F_pB}(e_1).
\]
The same bound holds for $x=0$. Taking the supremum over all $x\in F_pB$ gives $h_{F_pB}(e_1)\leq \rho_{F_pB}(e_1)$. Conversely, $F_pB$ is a star body for $0<p<2$, so $\rho_{F_pB}(e_1)e_1 \in F_pB$. Thus, $$h_{F_pB}(e_1)\geq\langle \rho_{F_pB}(e_1) e_1,e_1\rangle=\rho_{F_pB}(e_1).$$ We have shown \eqref{eq:box_identity}. Consequently, the strict inequality in \eqref{eq:not_valuation}, with its direction reversed when $p\in(1,2)$, also holds with every radial function replaced by its support function. Thus, $F_p$ is not an $L^p$-valuation for $p\in(0,2)\setminus\{1\}$. In particular, $K\mapsto\operatorname{conv}(F_pK)$ is not an $L^p$-valuation either.
\end{proof}

\appendix
\numberwithin{equation}{section}
\numberwithin{conj}{section}
\renewcommand{\theconj}{\thesection\arabic{conj}}

\section{Some Analytic Formulas}
\label{sec:appendix}
\noindent We include here some properties for the Bessel function and the Gamma function used in this work. 

The $\mu$th Bessel function is given by the Poisson integral identity, for $\mu > -\frac{1}{2}$,
\begin{equation}
    J_\mu(z) = \frac{1}{\Gamma\left(\mu+\frac{1}{2}\right)\sqrt{\pi}}\left(\frac{z}{2}\right)^\mu\int_{-1}^{1}\!\!\!e^{-izt}(1-t^2)^{\mu-\frac{1}{2}}dt, \quad\!\! z\in\mathbb{C}.
    \label{eq:Bessel}
\end{equation}

As an application of the theory of Bessel functions, we may write the Fourier transform of rotationally invariant functions, from, say, \cite[Appendix B.5, pg. 577]{GL14}:
Let $\phi$ be an integrable function on $\R$ such that $f(x)=\phi(|x|)$ is in $L^1(\R^n)$. Then,
\begin{equation}
\label{eq:stein}
\hat f(x) = |x| \int_0^\infty \phi(t)J_{\frac{n-2}{2}}(|x|t)\left(\frac{2\pi t}{|x|}\right)^\frac{n}{2}dt.
\end{equation}
Note that the reference we provided uses a different normalization for the Fourier transform, so its representation of \eqref{eq:stein} differs by constants.

The Bessel function satisfies the integral identity (see, e.g., \cite[Appendix B.3 on pg. 576]{GL14}): for $\mu,\nu >-\frac{1}{2}$,
\begin{equation}
\label{eq:fourier_identity}
    \int_0^1 \!J_\mu(sz) s^{\mu+1}\!\left(1-s^2\right)^\nu d s\!=\!\left(\frac{2}{z}\right)^\nu \frac{\Gamma(\nu+1)J_{\mu+\nu+1}(z)}{z}, \quad z>0.
\end{equation}
The classical asymptotics for the Bessel function when $t\geq 1$ are (see, e.g., \cite[Section B.8 on pg. 580]{GL14}):
\begin{equation}
\begin{split}
&J_\mu(t) = \sqrt{\frac{2}{\pi t}}\cos\left(t-\frac{\pi \mu}{2}\!-\!\frac{\pi}{4}\right)+\mathcal{O}(t^{-\frac{3}{2}}), \\
&\text{thus,}
\\ &\qquad\qquad|J_\mu(t)| \lesssim t^{-\frac{1}{2}}\, \text{ as }\, t\to \infty.
\end{split}
\label{eq:Bessel_asym_high}
\end{equation}
  On the other hand, as $t\to 0^+$, we have (see. e.g.,  \cite[Appendix B.6 on page 578]{GL14}):
  \begin{equation}
    J_\mu(t)=\frac{t^\mu}{2^\mu\Gamma\left(\mu+1\right)} +\mathcal{O}(t^{\mu+1}), \quad \text{and, \;thus} \qquad|J_\mu(t)| \lesssim t^\mu \quad \text{as} \quad t\to 0^+.
    \label{eq:Bessel_asym}
    \end{equation}
    Using these rates, we deduce that the Mellin transform of $J_\mu^2$ exists pointwise in the following sharp ranges: for $ -\frac{1}{2}<-\frac{\nu}{2}<\mu$,
    \begin{equation}
        \label{eq:Bessel_mellin_2}
        \int_0^\infty J_\mu^2(t)t^{\nu-1} dt = \frac{\Gamma\left(1-\nu\right)}{2^{1-\nu}\Gamma\left(1-\frac{\nu}{2}\right)^2}\frac{\Gamma\left(\mu+\frac{\nu}{2}\right)}{\Gamma\left(\mu+1-\frac{\nu}{2}\right)}.
    \end{equation}
    The formula for this integral can be found in the seminal treatise by G. N. Watson \cite[Chapter 13, Section 41, eq. 2, pg. 403]{WGN44}.

    We need the Gamma reflection formula in the following form,
    \begin{equation}
        \sin\left(\frac{p\pi}{2}\right)\Gamma\left(1-\frac{p}{2}\right)\Gamma\left(1+\frac{p}{2}\right)=\frac{p\pi}{2}, \qquad p\notin 2\mathbb{Z}.
        \label{eq:gamma_sine_reflection}
    \end{equation}
The cosine Gamma reflection formula is obtained from \eqref{eq:gamma_sine_reflection} by replacing $p$ with $p+1$:
    \begin{equation}
        \cos\left(\frac{p\pi}{2}\right)\Gamma\left(\frac{1}{2}+\frac{p}{2}\right)\Gamma\left(\frac{1}{2}-\frac{p}{2}\right)=\pi, \qquad p\notin 2\mathbb{Z}+1.
        \label{eq:gamma_cosine_reflection}
    \end{equation}
    The Legendre duplication formula for the Gamma function is
\begin{equation}
    \sqrt{\pi}\,\Gamma(p+1)=2^p\,\Gamma\left(\frac{1}{2}+\frac{p}{2}\right)\Gamma\left(1+\frac{p}{2}\right).
    \label{eq:duplication}
\end{equation}
Combining \eqref{eq:gamma_sine_reflection}, \eqref{eq:gamma_cosine_reflection} and \eqref{eq:duplication} creates 
\begin{equation}
    \label{eq:gamma_ident}
    \left(\Gamma(1-p) \sin\left(\frac{\pi p}{2}\right)\right) \left(\Gamma(p) \cos\left(\frac{\pi p}{2}\right)\right) = \frac{\pi}{2}, \qquad p\not\in \mathbb{Z}.\end{equation}

    The following Dirichlet-type integral will arise a few times when calculating Fourier transforms.
    \begin{prop}
        Let $a>0$ and $1<p<2$. Then, 
        \begin{equation}
        \label{eq:sine_identity}
        \int_0^{+\infty}x^{p-2}\sin(ax)dx = \frac{1}{a^{p-1}}\frac{\Gamma\left(p\right) \cos\left(\frac{p\pi}{2}\right)}{1-p}.
    \end{equation}
    \end{prop} 
    \noindent For the convenience of the reader, we provide an elementary proof of this fact. 
    \begin{proof}
    Without loss of generality, we set $a=1$. We will use the fact that, for $x>0$, we may write $x^{p-2}= \frac{1}{\Gamma\left(2-p\right)} \int_0^{+\infty} t^{1-p} e^{-x t} d t$. Therefore, \eqref{eq:sine_identity} becomes
\begin{align*}
 \int_0^{+\infty}x^{p-2}\sin(x)dx & = \int_0^{+\infty} \frac{t^{1-p}}{\Gamma(2-p)} \int_0^{+\infty}  e^{-x t} \sin (x) d x d t 
 \\
 &= \frac{1}{\Gamma(2-p)} \int_0^{+\infty} \frac{t^{1-p}}{1+t^2} d t 
 \\
& = \frac{1}{\Gamma(2-p)} \int_0^{\pi/2} \tan^{1-p} (\theta) d \theta 
\\
&= \frac{1}{\Gamma(2-p)} \int_0^{\pi/2} \sin^{1-p} (\theta) \cos^{p-1} (\theta) d \theta 
\\
&= \frac{1}{\Gamma(2-p)} \int_0^{\pi/2} \left(\sin^2(\theta)\right)^{\frac{1-p}{2}} (1-\sin^{2} (\theta))^{\frac{p-1}{2}} d \theta.
\end{align*}
Formally speaking, the first integral converges only conditionally: $$\int_0^{+\infty}x^{p-2}|\sin(x)|dx=+\infty.$$ The swapping of integrals in the first line is still valid: observe that
\[
\int_0^{2\pi m}e^{-tx}\sin x\,dx
=\frac{1-e^{-2\pi mt}}{1+t^2}
\xrightarrow{m\to \infty}\frac1{1+t^2}.
\]
Thus, truncate the integral from $[0,2\pi m]$, apply Fubini's theorem for these truncated integrals, and then use monotone convergence to take the limit as $m\to \infty$.

Next, we use a variable substitution $u=\sin^2(\theta)$ to obtain
\begin{equation}
\label{eq:almost_dirichlet}
\begin{split}
\int_0^{+\infty}x^{p-2}\sin(x)dx &= \frac{1}{2\Gamma\left(2-p\right)} \int_0^{1} u^{-\frac{p}{2}} (1-u)^{\frac{p}{2}-1} d u
\\
&=  \frac{1}{p}\frac{\Gamma\left(1-\frac{p}{2}\right) \Gamma\left(1+\frac{p}{2}\right)}{\Gamma(2-p)}.
\end{split}
\end{equation}
    Inserting \eqref{eq:gamma_sine_reflection} into \eqref{eq:almost_dirichlet} yields
    \begin{equation}
    \int_0^{+\infty}x^{p-2}\sin(x)dx = \frac{\pi}{2}\frac{1}{\sin\left(\frac{\pi p}{2}\right)}\frac{1}{\Gamma(2-p)}.
    \label{eq:almost_dirichlet_2}
\end{equation}
    By inserting \eqref{eq:gamma_ident} into \eqref{eq:almost_dirichlet_2}, we deduce \eqref{eq:sine_identity}.
\end{proof}

Finally, we will need the following integral identity.
\begin{lem} \label{l:int-cos}
    Let $p \in (1,2)$. Then,
    $$
    \int_0^{+ \infty} x^{p -5 } 
    \left(1- \frac{x^2}{2} - \cos x\right)  dx = -\frac{\cos \left(\frac{\pi p}{2} \right)\Gamma(p+1)}{(p-4)(p-3)(p-2)(p-1)p}.
    $$
\end{lem}

\begin{proof}
    
     By three applications of integration by parts, we have 
    \begin{align*}
    \int_{0}^{+ \infty} x^{p-5}\left( 1 - \frac{x^2}{2} - \cos x \right) dx 
    &= 
    - \int_{0}^{+ \infty} \frac{x^{p-4}}{p-4} \left( -x + \sin x \right) dx
    \\
    &= 
    \int_{0}^{+ \infty} \frac{x^{p-3}}{(p-4)(p-3)} \left( -1 + \cos x \right) dx
    \\
    &= 
    \int_{0}^{+ \infty} \frac{x^{p-2}}{(p-4)(p-3)(p-2)} \sin x dx.
\end{align*}
    We conclude using \eqref{eq:sine_identity}.
\end{proof}

\medskip

{\bf Funding:} 
D. Langharst was funded by the U.S. National Science Foundation's Mathematical Sciences Post-doctoral Research Fellowship (MSPRF) via NSF grant DMS-2502744.

A. Manui was funded by the Thailand Development and Promotion of Science and Technology talent project (DPST) and the Chateaubriand Fellowship of the Office for Science \& Technology of the Embassy of France in the United States.

A. Manui and A. Zvavitch were funded by the U.S. National Science Foundation Grants DMS-2247771 and DMS-2604412, the United States-Israel Binational Science Foundation (BSF) Grant 2018115.

\medskip

{\bf Acknowledgments:} This project was initiated while the second and third authors were visiting M. Fradelizi at Universit\'e Gustave Eiffel. We would like to wholeheartedly thank the university and the staff for their hospitality and kindness.

We are very grateful to G. Bianchi for the valuable email correspondence concerning Fourier transforms of characteristic functions of convex bodies, which culminated in Theorem~\ref{t:compactness}; M. Fradelizi for many valuable discussions and for pointing out his work on the isotropic constant, which led to an improvement of Theorem \ref{t:use_slicing}; and Y. Zhao for discussions on the Fourier transform of $R_p K$ when $p \in (-1,0)$. Finally, we thank the anonymous referee, whose suggestions vastly improved the content and presentation of our work.

\medskip

{\bf AI Disclosure Statement:} The first draft of this work was written nearly a year ago, without the use of AI. Recently, D. Langharst and A. Zvavitch were given free access to OpenAI's ChatGPT for Academic Researchers. During the revision of this work, we used it to check the presentation of our work and the clarity of exposition. The authors thank OpenAI for providing us access to these resources. The authors take full ownership of the content of the present manuscript.

{\bf Data Disclosure Statement:} We attest that no data was used for this article.

\bibliographystyle{siam}

\clearpage
\noindent Dylan Langharst
\\
Department of Mathematical Sciences, Carnegie Mellon University, Pittsburgh, PA 15213, USA
\\
Email address: dlanghar@andrew.cmu.edu
\vspace{2mm}
\\
\noindent Auttawich Manui 
\\
Department of Mathematics, Syracuse University, Syracuse, NY 13244, USA
\\
E-mail address: amanui@syr.edu
\vspace{2mm}
\\
\noindent Artem Zvavitch
\\
Department of Mathematical Sciences, Kent State University, Kent, OH 44242, USA
\\ Email address: azvavitc@kent.edu
\end{document}